\documentclass{amsart}

\usepackage[T1]{fontenc}
\usepackage[lf]{Baskervaldx} % lining figures
\usepackage[bigdelims,vvarbb]{newtxmath} % math italic letters from Nimbus Roman
\usepackage[cal=boondoxo]{mathalfa} % mathcal from STIX, unslanted a bit

\usepackage[all]{xy}
\usepackage{latexsym}
\usepackage{amssymb}
\usepackage{amsmath}
\usepackage{amsthm}
\usepackage{amscd}
\usepackage{fancyhdr}
\usepackage{a4wide}
\usepackage{mathrsfs}
\usepackage{graphicx}
\usepackage{color}
\usepackage[colorlinks,final,hyperindex]{hyperref}
\usepackage[noabbrev,capitalize]{cleveref}
\usepackage{tikz}
\usepackage{tikz-cd}
\tikzset{math3d/.style=
    {x= {(-0.353cm,-0.353cm)}, z={(0cm,1cm)},y={(1cm,0cm)}}}

\definecolor{Chocolat}{rgb}{0.36, 0.2, 0.09}
\definecolor{BleuTresFonce}{rgb}{0.215, 0.215, 0.36}
\hypersetup{citecolor=BleuTresFonce, linkcolor=Chocolat}

\theoremstyle{plain}
\newtheorem{proposition}{Proposition}[section]
\newtheorem{theorem}{Theorem}[section]
\newtheorem{corollary}{Corollary}[section]

\newtheorem{lemma}{Lemma}[section]

\theoremstyle{definition}
\newtheorem{definition}{Definition}[section]

\newtheorem{remark}{\sc Remark}[section]
\newtheorem{exam}{\sc Example}[section]

\newtheorem{cond}{\sc Condition}
\def\QD{\textbf{QD}}
\def\ot{\otimes}
\def\L{\mathrm{L}}
\def\S{\mathrm{S}}
\def\Lie{\textbf{Lie-alg}}

\def\Com{\textbf{Com-alg}}
\def\Comco{\textbf{Com-coalg}}

\def\wgComco{\textbf{wg-Com-coalg}}
\def\A{\mathrm{A}}
\def\Ac{\mathrm{A}^\mathrm{c}}
\def\Sc{\mathrm{S}^\mathrm{c}}
\def\Ass{\textbf{Ass-alg}}
\def\Lie{\textbf{Lie-alg}}
\def\Com{\textbf{Com-alg}}
\def\wgAss{\textbf{wg-Ass-alg}}
\def\wgLie{\textbf{wg-Lie-alg}}

\def\wgCom{\textbf{wg-Com-alg}}
\def\Assco{\textbf{Ass-coalg}}
\def\wgAssco{\textbf{wg-Ass-coalg}}
\def\uot{\underline{\otimes}}
\newcommand{\acc}{\text{\raisebox{\depth}{\textexclamdown}}}%{\text{\rm !`}}
\def\Sy{\mathbb{S}}
\def\C{\textbf{C}}
\def\D{\textbf{D}}
\def\CC{\mathcal{C}}
\def\P{\mathcal{P}}
\def\O{\mathcal{O}}
\def\Fin{\textbf{Fin}}
\def\id{\mathrm{id}}
\def\F{\mathrm{F}}
\def\BQO{\textbf{BOQD}}
\def\g{\mathfrak{g}}

\def\t{\mathrm{t}}
\def\e{\mathrm{e}}
\DeclareMathOperator{\As}{As}
\DeclareMathOperator{\gr}{gr}
\def\QQ{\mathbb{Q}}
\def\GraS{\mathrm{Gra}_{S^1}}
\def\LG{\mathrm{LG}}

\newcommand{\Hom}{\ensuremath{\mathrm{Hom}}}

\def\un{\underline{n}}
\def\kHG{k\text{-}\mathrm{HG}}
\def\krHG{k\text{-}\overline{\mathrm{HG}}}
\def\kLHG{k\text{-}\mathrm{LHG}}
\def\EHKR{\mathrm{EHKR}}

\title{Monoidal structures on the categories of quadratic data}

\date{\today}

\thanks{B.V. was supported by the IUF and the grant ANR-14-CE25-0008-01 project SAT}

\author{Yuri I. Manin}
\address{Max--Planck--Institute for Mathematics, Vivatsgasse 7, Bonn 53111, Germany}
\email{manin@mpim-bonn.mpg.de}

\author{Bruno Vallette}
\address{Laboratoire Analyse, G\'eom\'etrie et Applications, Universit\'e Paris Nord 13, Sorbonne Paris Cit\'e, CNRS, UMR 7539, 93430 Villetaneuse, France.}
\email{vallette@math.univ-paris13.fr}

\subjclass[2010]{Primary 18D10; Secondary 16S37, 18D50}

\keywords{Monoidal categories, 2--monoidal categories, quadratic data, operads, black and white products, Koszul duality.}

\begin{document}

\maketitle

\begin{abstract}
The notion of 2--monoidal category used here was introduced
by B.~Vallette in 2007 for applications in the operadic context.
The starting point for this article was a remark by Yu. Manin that
in the category of quadratic algebras (that is, ``quantum linear spaces'') one can also define
2--monoidal structure(s) with rather unusual properties.
Here we give a detailed exposition of these constructions,
together with their generalisations to the case of quadratic operads.

Their parallel exposition was motivated by the following remark.
Several important operads/cooperads such as genus zero quantum cohomology operad,
the operad classifying Gerstenhaber algebras, and more generally,
(co)operads of homology/cohomology  of some topological operads, start with collections of
quadratic algebras/coalgebras rather than simply linear spaces.

Suggested here enrichments of the categories to which components of these operads belong, as well of the operadic structures themselves, might lead to the better understanding of these fundamental objects.
\end{abstract}

\setcounter{tocdepth}{1}
\tableofcontents

\section{Brief summary and plan of exposition}

A {\it monoidal category}, as it was defined and studied in \cite[Chapter VII]{MacLane98},
is a category $\mathcal{C}$ endowed with a bifunctor $\boxtimes :\, \mathcal{C}\times \mathcal{C}\to
\mathcal{C}$ satisfying the associativity axiom (``pentagon diagram'') and  equipped with a 
(left and right) unit object.

\smallskip

A  {\it lax 2--monoidal category}, as it was introduced in \cite{Vallette08},  is a category $\mathcal{C}$
endowed with {\it two} structures of monoidal category, whose respective bifunctors
$\boxtimes$ and $\otimes$ are related by the natural transformation
called {\it an interchange law}:
$$
\varphi_{AA^{\prime}BB^{\prime}} :\ 
(A\otimes A^{\prime}) \boxtimes (B\otimes B^{\prime}) \to 
(A\boxtimes B)\otimes  (A^{\prime}\boxtimes B^{\prime})
\eqno(1.1)
$$
compatible with associativity of $\boxtimes$ and $\boxtimes$--unit object
in the sense made explicit in the Proposition 2 of [Va08].
Inverting all arrows (i.e. working in the opposite category), one gets the notion of
{\it colax 2--monoidal category}. Finally, 
a $2$-monoidal category equipped with a  lax and a colax structure is simply called a {\it 2--monoidal category}.
A close  but more restricted notion, which now often called \emph{duoidal category} was coined by M. Aguiar and S. Mahajan in \cite{AguiarMahajan06}.
\smallskip

Notice that A. Joyal and R. Street, in the work \cite{JoyalStreet93} on braided
tensor categories, came up with a notion of a category endowed with two monoidal products but related by a natural isomorphism, which forces the two monoidal structures to be isomorphic. 
C. Balteanu, Z. Fiedorowicz, R.
Schw\"anzl and R. Vogt in \cite{BalteanuFiedorowiczSchwanzlVogt03} introduced a notion of iterated
monoidal category in order to study  iterated loop spaces. But in their framework, the units for the monoidal products should be equal. Neither of these two restrictions is imposed in our present examples. 

\smallskip

\cref{sec:DefQD} of our paper starts with a systematic formalization of the
general notions: ``algebra/operad defined by quadratic relations between
(graded) generators'' and their reduction to the notions of ``quadratic data''.
We then introduce various relevant categorical frameworks involving
monoidal structures on the categories of such data, various canonical
functors between them, and basic commutative diagrams
relating these functors.

\smallskip

          This is a development and generalization of constructions
introduced in \cite{Manin88} as an approach to quantum algebra: quantum
linear spaces, black and white products, bialgebras of their
quantum endomorphisms,  and quantum groups.

\smallskip

          The central result of \cref{sec:2MonoQ} is a new construction of
2-monoidal structures on the categories of quadratic data $\QD$  
(defined in \cref{subsec:DefQD}): we start with a simple construction of 2-monoidal
structure on the category of graded vector spaces, and then show that it lifts
to the category of quadratic data.

\smallskip

          The central result of \cref{sec:2MonoOp} is a generalization of this construction
to the categories of binary operadic quadratic data  $\BQO$
defined in \cref{subsec:BQO}.

\smallskip

         Finally, in \cref{Sec:HopfOp}, we return to the quantum picture of \cite{Manin88}
and generalize it to our framework, as was done in \cite{Manin18} for the simplest
case of the genus zero component of the quantum cohomology operad.

\smallskip

        The most important new feature of our picture is the fact that
there is an abundance of operads/cooperads with postulated
properties arising naturally in various geometric contexts.

\smallskip

         More precisely, any topological operad like the little discs operad (loop spaces) or the Deligne--Mumford operad of moduli spaces of stable genus 0 curves with marked points (quantum cohomology), induces 
a \emph{homology} operad in the category of cocommutative coalgebras, 
a \emph{cohomology} cooperad in the category of commutative algebras, 
and a \emph{``homotopy''} operad in the category of Lie algebras. 
It is difficult  to pass from one to another directly at the level of Lie algebras and (co)commutative (co)algebras.

\smallskip

 It is however well-known that the Koszul duality of \cite{GinzburgKapranov94} between the two operads $\mathrm{Lie}-\mathrm{Com}$  coincides with the duality Homotopy-(co)Homology in rational homotopy theory. 
Our idea here is to lift these operadic structures on the level of simple categories of quadratic data without any loss of information (under \cref{Cond:ConditionI} of \cref{prop:(Co)Ho(Co)Op}). In order to do so, we introduce the relevant notions of (symmetric, skew-symmetric) quadratic data together with suitable symmetric monoidal structures. On that level, we do have the Koszul duality and the linear duality functors. There are also ``realisation'' functors from these categories of quadratic data to categories of (co)algebras. Since all these functors are symmetric monoidal, they preserve (co)operad structures. 

\[\xymatrix@C=30pt@R=30pt{
(\QD^-, \oplus) \ar[r]^{\acc}  \ar[d]_{\textrm{L}} \ar@/^1.5pc/[rr]^{!}  & 
(\QD^+,\uot)   \ar[d]_{\Sc}  \ar[r]^{*} &
(\QD^+, \vee) \ar[d]_{\S}\\
(\Lie, \oplus)  & (\Comco,\ot)   &(\Com,\ot) \\
\boxed{\textstyle \textsc{Homotopy}\atop \textstyle \text{Lie Operads}} &
\boxed{\textstyle \textsc{Homology}\atop \textstyle \text{Hopf Operads}}& 
\boxed{\textstyle \textsc{Cohomology}\atop \textstyle \text{Hopf Cooperads}}
}
\]
The simplest case is when one has to deal with operads in the category of skew-symmetric quadratic data $\QD^-$,  where the underlying monoidal structure is particularly easy: the direct sum. So this category is our favorite site to describe operadic structures. Then, we get for free (co)operad structures in all the other symmetric monoidal categories. 
It turns out that, this way, one can recover many of the most important (co)operad structures present in the literature, like the graph (co)operads,  the ones related to the little discs/configuration spaces of points 
$\mathcal{D}_2(n)\sim \mathrm{Conf}_n(\mathbb{R}^2)$, 
the  real locus of the  moduli spaces of stable curves of genus $0$ with marked points
$\overline{\mathcal{M}_{0,n}}(\mathbb{R})$, and 
their non-commutative versions. This point has two main interests: it makes particularly easy the passage between Lie operads and Hopf (co)operads and it allows us to organise the various operad structures in a commun pattern. 
For instance, we construct a family of operads in skew--symmetric quadratic data whose first two cases are provided by the Drinfeld--Kohno quadratic data $\mathrm{Conf}_n(\mathbb{R}^2)$ and the Etingof--Henriques--Kamnitzer--Rains quadratic data $\overline{\mathcal{M}_{0,n}}(\mathbb{R})$; we also give them a canonical operadic interpretation. This gives a new family of operads quite similar to the $e_k$-operads, except that instead of having a degree $k-1$ (binary) Lie bracket, we have a degree $1$ ``Lie bracket'' of arity $k$.
\smallskip

With the same method, one can also study complex cases like the operad made up of the complex locus of the moduli spaces of stable curves of genus $0$ with marked points $\overline{\mathcal{M}_{0,n}}(\mathbb{C})$  \cite{KontsevichManin94, Getzler95, KontsevichManin96, Manin99}, whose cohomology rings admit a quadratic presentation by \cite{Keel92}. 
There is also its non-commutative version $\mathcal{B}(n)$ introduced in \cite{DotsenkoShadrinVallette15} by means of toric varieties called brick manifolds  and  the dihedral topological operad ${\mathcal{M}^\delta_{0,n}}(\mathbb{C})$ introduced by F. Brown in  \cite{Brown09} as a partial compactification with a view to understand multiple zeta values, see also \cite{DV17, AP17}. The details are left to an interested reader.

\subsection*{Acknowledgements}
We are  grateful to Clemens Berger, Ricardo Campos, Vladimir Dotsenko, Anton Khoroshkin, and Daniel Robert--Nicoud for interesting and useful discussions. 

\section{Quadratic data, monoidal structures, and their algebraic realisations}\label{sec:DefQD}

\subsection{Notations and conventions}
We work over a ground field $K$ of characteristic $\neq 2$ and over the underlying category of finite dimensional $\mathbb{Z}$-graded $K$-vector spaces equipped with their morphisms of degree zero. The linear dual $V^*$ is considered degree-wise: $(V^*)_{-n}\coloneqq\mathrm{Hom}(V_n, K)$. 
We equip this category with the usual tensor product $(M\otimes N)_n\coloneqq\bigoplus_{k+l=n} M_k\otimes N_l$ and with the natural isomorphisms $\sigma(x\otimes y)\coloneq (-1)^{|x||y|} y\otimes x$ in order to make it into a symmetric monoidal category denoted simply by $(\textbf{grVect}, \otimes)$. We denote by $s$ (respectively its linear dual $s^{-1}$) the one-dimensional graded vector space concentrated in degree $1$ (respectively $-1$) and the degree shift operator by $sV\coloneq s\ot V$ (respectively $s^{-1}V\coloneq s^{-1}\ot V$). 

\subsection{Categories of quadratic data}\label{subsec:DefQD} For any graded vector space $V$, we consider the canonical decomposition $V^{\ot 2}\cong V^{\odot 2} \oplus V^{\wedge 2}$, where
\begin{align*}
V^{\odot 2}\coloneq
\left\langle  x\odot y\coloneq x\ot y+ (-1)^{|x||y|}y\ot x \right\rangle \qquad  \text{and} \qquad
V^{\wedge 2}\coloneq
\left\langle  x\wedge y\coloneq x\ot y- (-1)^{|x||y|}y\ot x \right\rangle
\ .
\end{align*}

\begin{definition}[Quadratic data]
An object of the category $\QD$ of \emph{quadratic data} is a pair $(V, R)$ made up of a graded vector space $V$ and a subspace $R\subset V^{\otimes 2}$. A morphism $f : (V,R) \to (W,S)$ of quadratic data amounts to a morphism $f : V \to W$ of graded vector spaces satisfying $f^{\otimes 2}(R)\subset S$.

The category of \emph{symmetric quadratic data} $\QD^+$ (respectively  \emph{skew-symmetric quadratic data} $\QD^-$) is defined similarly with pairs $(V, R)$ such that $R\subset 
V^{\odot 2}$ (respectively $R\subset V^{\wedge 2}$) this time. 
\end{definition}

\subsection{Functors}\label{subsec:Func} There are first obvious ``realisation'' functors from the categories of quadratic data to the categories of unital associative algebras, unital commutative algebras, and Lie algebras respectively: 
\[
\begin{array}{llll}
\A \ :&\QD & \to & \Ass\\
& (V,R) & \mapsto & \frac{T(V)}{(R)}
\end{array} \ , \qquad 
\begin{array}{llll}
\S \ :&\QD^+ & \to & \Com\\
& (V,R) & \mapsto & \frac{S(V)}{(R)}
\end{array}  \ ,\qquad 
\begin{array}{llll}
\L \ :&\QD^- & \to & \Lie\\
& (V,R) & \mapsto & \frac{Lie(V)}{(R)}
\end{array}\ .
\]

\medskip

In order to lift the universal enveloping algebra functor 
\[
\begin{array}{llll}
U \ :&\Lie & \to & \Ass\\
& (\mathfrak{g}, [\; ,\,]) & \mapsto & U(\mathfrak{g})\coloneq\frac{T(\mathfrak{g})}{\left(x\ot y -(-1)^{|x||y|}y\ot x-[x,y]\right)}
\end{array} 
\]
to the quadratic data level, we consider the functor
\[
\begin{array}{llll}
\Lambda \ :&\QD^-& \to & \QD \\
& (V,R) & \mapsto & (V, \Lambda(R))\ , 
\end{array} 
\]
where $\Lambda(R)\in  V^{\wedge 2}\subset V^{\ot 2}$ is the natural inclusion.

Similarly, we lift the inclusion functor $\Com \hookrightarrow  \Ass$ 
to the quadratic data level by
\[
\begin{array}{llll}
\mathcal{S} \ :&\QD^+& \to & \QD \\
& (V,R) & \mapsto & (V, \Sigma(R)\oplus V^{\wedge 2})\ , 
\end{array} 
\]
where $\Sigma(R) \in V^{\odot 2} \subset V^{\ot 2}$ is the natural inclusion.
\medskip

One can notice that the images of these algebraic realisation functors always produce a \emph{weight graded} algebra, that is $A\cong \bigoplus_{n \in \mathbb{N}} A^{(n)}$, where each component $A^{(n)}$ is finite dimensional. We denote the associated categories respectively by $\wgAss$, $\wgCom$, and $\wgLie$. 

\medskip

Dually, we consider the two categories of weight graded counital coassociative coalgebras $\textbf{wg-}\allowbreak\textbf{Ass-}\allowbreak\textbf{coalg}$ and weight graded counital cocommutative coalgebras $\wgComco$, with finite dimensional components. 
There are also realisation functors from the categories of quadratic data to these two categories: 
\[
\begin{array}{llll}
\Ac \ :&\QD & \to & \wgAssco\\
& (V,R) & \mapsto & T^c(V,R)
\end{array}  \qquad \text{and} \qquad
\begin{array}{llll}
\S^c \ :&\QD^+ & \to & \wgComco\\
& (V,R) & \mapsto & S^c(V,R)
\end{array} \ ,
\]
where the quadratic coalgebra $T^c(V,R)$ (and similarly the quadratic cocommutative coalgebra $S^c(V,\allowbreak R)$) is 
initial object in the category of (conilpotent) counital coassociative  coalgebras under $T^c(V)$ such that the composite with the projection onto $\frac{T^c(V)}{R}$ vanishes: 
\[
\xymatrix{
C  \ar[r] \ar@{..>}[d]_(0.45){\exists} \ar@/^1pc/[rr]^0 & T^c(V)\ar@{->>}[r]& \frac{T^c(V)}{R} \\
 T^c(V,R) \ar[ur] \ar@/_1pc/[urr]_0& & } .
\]
It is explicitly given by 
\[
T^c(V,R)\cong K \oplus V \oplus R \oplus \cdots \oplus 
\left(\bigcap_{i+2+j=n} V^{\ot i} \ot R \ot V^{\ot j}
\right) \oplus \cdots \ ,
\]
see  \cite[Section~2]{Vallette08} or \cite[Section~3.1.3]{LodayVallette12} for more details. The category of  cocommutative coalgebras naturally imbeds into the category of coassociative coalgebras: $\Comco \hookrightarrow  \Assco$ and similarly $\wgComco \hookrightarrow  \wgAssco$. These functors lift on the level of quadratic data by 
\[
\begin{array}{llll}
\Sigma \ :&\QD^+& \to & \QD \\
& (V,R) & \mapsto & (V, \Sigma(R))\ .
\end{array} 
\]

There are first \emph{Koszul dual} functors  
\[
\begin{array}{llll}
\acc \ :&\QD& \to & \QD \\
& (V,R) & \mapsto & (sV, s^{2}R) 
\end{array} 
\qquad \text{and} \qquad 
\begin{array}{llll}
\acc \ :&\QD^\pm& \to & \QD^\mp \\
& (V,R) & \mapsto & (sV, s^{2}R)\ , 
\end{array} 
\]
where the double degree shift operator is defined by $s(x \otimes y)\coloneq (-1)^{|x|}s x \otimes s  y$ and which sends symmetric quadratic data to skew-symmetric quadratic data and vice versa. 
Notice that, all the above-mentioned functors are covariant. 

\medskip 

Now, we consider the linear dual \emph{contravariant} functors 
\[
\begin{array}{llll}
* \ :&\QD& \to & \QD \\
& (V,R) & \mapsto & (V^*, R^\perp) 
\end{array} 
\qquad \text{and} \qquad 
\begin{array}{llll}
* \ :&\QD^\pm& \to & \QD^\pm \\
& (V,R) & \mapsto & (V^*, R^\perp)\ , 
\end{array} 
\]
In the former case, since $R\subset V^{\otimes 2}$, its orthogonal is understoof in $R^\perp\subset (V^*)^{\otimes 2}\cong \left(V^{\ot 2}\right)^*$. In the latter case, since $R\subset V^{\odot 2}$ (respectively $R\subset V^{\wedge 2}$), its orthogonal is understood in $R^\perp\subset (V^*)^{\odot 2}$ (respectively in $R^\perp\subset (V^*)^{\wedge 2}$). 

One can iterate the above two types of functors to produce the \emph{second Koszul dual} (contravariant) functors: 
\[
\begin{array}{llll}
!\coloneq*\acc \ :&\QD& \to & \QD \\
& (V,R) & \mapsto & (s^{-1}V^*, s^{-2}R^\perp) 
\end{array} 
\qquad \text{and} \qquad \begin{array}{llll}
!\coloneq*\acc \ :&\QD^-& \to & \QD^+ \\
& (V,R) & \mapsto & (s^{-1}V^*, s^{-2}R^\perp)\ .
\end{array} 
\]

The weight-wise linear duality functor sends coalgebras to algebras (and vice-versa): 
\[
\begin{array}{llll}
* \ :&\wgAssco& \to & \wgAss  \\
& \bigoplus_{n\in \mathbb{N}} C^{(n)} & \mapsto &  \bigoplus_{n\in \mathbb{N}} \left(C^{(n)}\right)^*
\end{array} 
\qquad \text{and} \qquad 
\begin{array}{llll}
* \ :&\wgComco& \to & \wgCom  \\
& \bigoplus_{n\in \mathbb{N}} C^{(n)} & \mapsto &  \bigoplus_{n\in \mathbb{N}} \left(C^{(n)}\right)^*\ .
\end{array} 
\]

\begin{proposition}\label{Prop:ComDiag}
All these functors assemble into the following commutative diagram. 
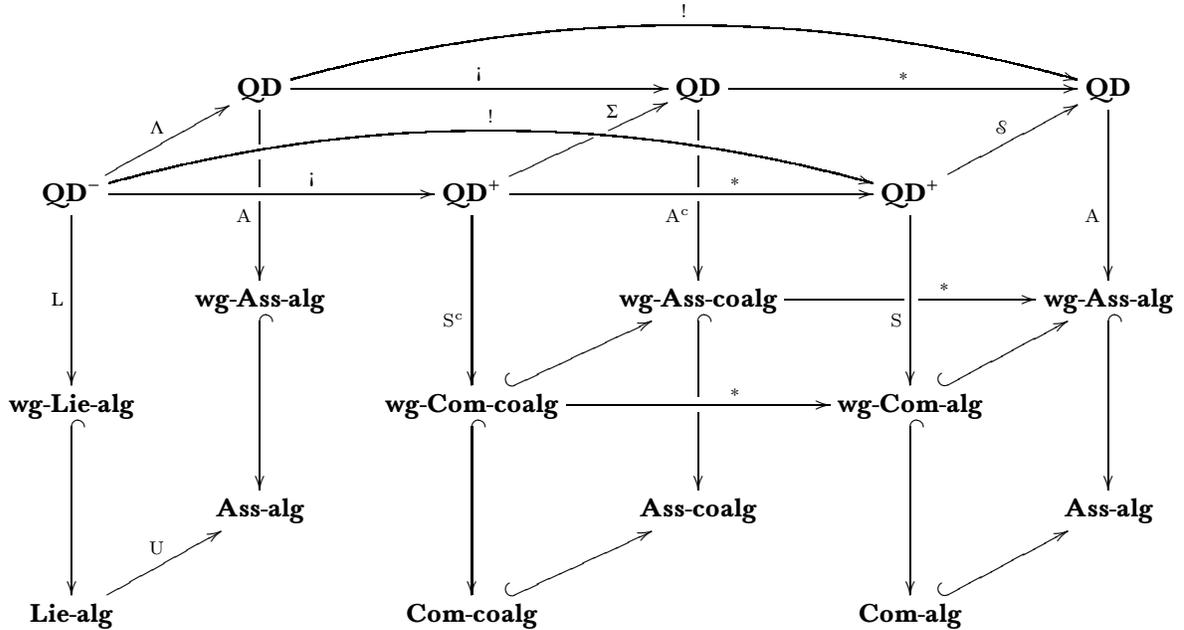
\rm
\[\xymatrix@C=17pt{
&\QD\ar[rr]^{\acc} \ar[dd]|(0.3)\hole|(0.5)\hole_(0.6){\A} \ar@/^2pc/[rrrr]^{!} & & \QD\ar[dd]_(0.6){\Ac}|(0.28)\hole|\hole \ar[rr]^{*} &&  \QD \ar[dd]_(0.6){\A} \\
\QD^- \ar[rr]^(0.6){\acc}  \ar[ru]^{\Lambda} \ar[dd]_{\textrm{L}} \ar@/^2pc/[rrrr]^{!}  & &
\QD^+ \ar[ru]^(0.65){\Sigma}|(0.57)\hole  \ar[dd]_(0.6){\Sc}  \ar[rr]^(0.6){*} & & \QD^+\ar[ru]^{\mathcal{S}} \ar[dd]_(0.6){\S}&\\
& \wgAss \ar@{^{(}->}[dd]& & \wgAssco \ar[rr]|(0.52)\hole^(0.6){*} \ar@{^{(}->}[dd]|\hole& & \wgAss\ar@{^{(}->}[dd]\\
\wgLie\ar@{^{(}->}[dd] & &\wgComco  \ar[rr]^(0.6){*} \ar@{^{(}->}[ru]\ar@{^{(}->}[dd]&  &\wgCom\ar@{^{(}->}[ru] \ar@{^{(}->}[dd]&\\
& \Ass&&\Assco   &  &\Ass 
\\
\Lie \ar[ru]^{\mathrm{U}} & &\Comco   \ar@{^{(}->}[ru]&  &\Com \ar@{^{(}->}[ru] 
&
}
\]
\end{proposition}

\begin{proof}
The commutativity of the extreme top  face amounts to 
$\Lambda(R)^\perp\cong \Sigma(s^{-2}R^\perp)\oplus (s^{-1}V^*)^{\wedge 2}$, 
the commutativity of the left top face amounts to 
$s^2\Lambda(R)\cong \Sigma(s^{2}R)$, and 
the commutativity of the right top face amounts to 
$\Sigma(R^\perp)\cong \Sigma(R)^\perp$.

The commutativity of the front face is given by $S^c(V,R)^*\cong S(V^*, R^\perp)$ 
and the commutativity of the back face is given by $T^c(V,R)^*\cong T(V^*, R^\perp)$. This comes from the fact that the  universal property satisfied by  quadratic algebras is categorically dual to the universal property defining  quadratic coalgebras, see \cite[Section~3.2.2]{LodayVallette12}.

The commutativity of the left-hand side vertical face comes from  $\mathrm{U}\left(\frac{Lie(V)}{(R)}\right)\cong \frac{T(V)}{(\Lambda(R))}$, and the commutativity of the right-hand side vertical  face comes from 
$\frac{S(V)}{(R)}\cong \frac{T(V)}{(\Sigma(R)\oplus V^{\wedge 2})}$. 
The commutativity of the central horizontal face is obvious. It induces the commutativity of the 
central vertical face: the isomorphism $S^c(V,R)\cong T^c(V, \Sigma(R))$ can be seen under the weight-wise linear dual from the above isomorphism and the dual characterisations of quadratic (co)algebras. (One can also prove that $S^c(V,R)$ satisfies the universal property  of coassociative quadratic coalgebra generated by $(V, \Sigma(R))$.)

The commutativity of the other faces involving only forgetful functors is straightforward.
\end{proof}

\begin{remark}
We could also consider Koszul dual and linear dual inverse functors, going in the opposite direction, defined by formulas like $(s^{-1}V, s^{-2}R)$ for $\acc$ and  $(sV^*, s^2R^\perp)$ for $*$\ . We keep the exposition to the present degree of details for reasons that will be apparent in \cref{Sec:HopfOp}, when dealing with (co)operad structures. So far, we would like the vertex labeled by the category $\QD^-$ to be the unique top vertex of this diagram.
\end{remark}

\subsection{Symmetric monoidal structures}\label{subsec:MonoStruc}
We now enrich the above categories with symmetric mo\-no\-idal structures. 
On the category $\QD$ of quadratic data, we consider the two symmetric monoidal products $\ot$ and $\uot$:
\[
(V,S)\ot (W,R)\coloneq(V\oplus W, R\oplus  [V,W]_- \oplus S)
\quad \text{and} \quad
(V,S)\, \underline{\ot}\,  (W,R)\coloneq(V\oplus W, R\oplus  [V,W]_+ \oplus S)
\ ,
\]
where $[V,W]_\pm\coloneq\left\langle v\ot w \pm (-1)^{|v||w|}w\ot v\ \big| \ v\ot w \in V\ot W \right\rangle$ .

\medskip

The category $\QD^+$ of symmetric quadratic data is endowed with 
\[
(V, R)\vee (W,S)\coloneq(V\oplus W, R\oplus S)  \qquad \text{and} \qquad (V, R)\, \uot\, (W,S)\coloneq(V\oplus W, R\oplus 
[V,W]_+ \oplus S) 
\]
and the category $\QD^-$ of skew-symmetric quadratic data is endowed with 
\[
(V, R)\oplus (W,S)\coloneq(V\oplus W, R\oplus [V, W]_- \oplus S)\ .
\]
The bottom categories of algebras are equipped with the following monoidal products. 
We consider the direct sum $\mathfrak{g}\oplus \mathfrak{h}$ of Lie algebras, where 
$[x+y, x'+y']\coloneq[x,x']+[y,y']$, for any $x,x'\in \mathfrak{g}$ and $y,y'\in \mathfrak{h}$. This is the categorical product in the category $\Lie$. The underlying tensor product $\ot$ of two associative algebras  $A,B \in \Ass$ carries a natural associative product: 
$\mu(a\ot b,a'\ot b')\coloneq(-1)^{|a'||b|}\mu_A(a,a')\ot \mu_B(b, b')$. If the two algebras happen to be commutative, so is their tensor product. The same holds true for the tensor product of coassociative or cocommutative coalgebras and in the weight graded case. 

\begin{proposition}\label{Prop:SymMono}
The above-mentioned monoidal products endow their respective categories with a  symmetric monoidal structure. 
\end{proposition}

\begin{proof}
Recall from \cite[Section~XI.1]{MacLane98} that to get a (strong) symmetric monoidal category 
besides monoidal products
described above we have to define coherent objects (units)
 and coherent natural isomorphisms (associator, left and right unitors, and braiding). For the five categories 
$(\QD, \ot)$, $(\QD, \uot)$,   $(\QD^+, \vee)$, $(\QD^+, \uot)$, and $(\QD^-, \oplus)$ of quadratic data, the unit is $(0,0)$, the associator is $(V\oplus W)\oplus Z \cong V\oplus (W\oplus Z)$, the unitors are $0\oplus V \cong V \cong V \oplus 0$, and the braiding is $V\oplus W \cong W \oplus V$. The symmetric monoidal structure on $(\Lie, \oplus)$ is given by a similar unit and by similar maps. 
For  all monoidal categories of (possibly weight graded) algebras and coalgebras, the unit is $K$, the associator is $(V\ot W)\ot Z \cong V\ot (W\ot Z)$, the unitors are $K\ot V \cong V \cong V \ot K$, and the braiding is $V\ot W \cong W \ot V$. 
The various coherence diagrams are then straightforward to check. 
\end{proof}

\begin{remark}
Notice that the categories $(\QD^-, \oplus)$, $(\QD^+,\uot)$,
$(\Lie, \oplus)$ and the category of 
coaugmented (weight-graded) cocommutative coalgebras with $\otimes$ are cartesian, that is their symmetric monoidal structure is given by their product and their terminal object. 
Dually the category $(\QD^+, \vee)$ and the category of augmented (weight-graded) commutative algebras with $\otimes$ are cocartesian, that is their symmetric monoidal structure is given by their coproduct and their initial object. 

\end{remark}

\subsection{Symmetric monoidal functors}\label{subsec:SymMonFun} 
We can now check the possible coherence between the various 
functors and symmetric monoidal structures  introduced above. 

\begin{theorem}\label{Thm:ComDiag}
The commutative diagram described on \cref{Prop:ComDiag}
          is made up of strong symmetric monoidal functors.
\rm 
\[\xymatrix@C=0pt{
&(\QD, \ot) \ar[rr]^{\acc} \ar[dd]|(0.3)\hole|(0.5)\hole_(0.6){\A} \ar@/^2pc/[rrrr]^{!} & & (\QD, \uot)\ar[dd]_(0.6){\Ac}|(0.28)\hole|\hole \ar[rr]^{*} &&  (\QD, \ot) \ar[dd]_(0.6){\A} \\
(\QD^-, \oplus) \ar[rr]^(0.6){\acc}  \ar[ru]^{\Lambda} \ar[dd]_{\textrm{L}} \ar@/^2pc/[rrrr]^{!}  & &
(\QD^+,\uot) \ar[ru]^(0.65){\Sigma}|(0.56)\hole  \ar[dd]_(0.6){\Sc}  \ar[rr]^(0.6){*} & & (\QD^+, \vee)\ar[ru]^{\mathcal{S}} \ar[dd]_(0.6){\S}&\\
&  (\wgAss,\ot)\ar@{^{(}->}[dd]& & (\wgAssco, \ot) \ar[rr]|(0.52)\hole^(0.6){*} \ar@{^{(}->}[dd]|\hole& & (\wgAss,\ot)\ar@{^{(}->}[dd]\\
 (\wgLie,\oplus)\ar@{^{(}->}[dd]& &(\wgComco, \ot)  \ar[rr]^(0.6){*} \ar@{^{(}->}[ru]\ar@{^{(}->}[dd]&  &(\wgCom, \ot)\ar@{^{(}->}[ru] \ar@{^{(}->}[dd]&\\
& (\Ass, \ot)&&(\Assco, \ot)   &  &(\Ass , \ot)
\\
(\Lie, \oplus) \ar[ru]^{\mathrm{U}} & &(\Comco,\ot)   \ar@{^{(}->}[ru]&  &(\Com,\ot) \ar@{^{(}->}[ru] 
&
}
\]
\end{theorem}

\begin{proof}
Recall from \cite[Section~XI.2]{MacLane98} that a strong symmetric monoidal functor $\F : (\C, \ot_\C, 1_\C) \to (\textbf{D}, \ot_\textbf{D}, 1_\textbf{D})$, is a covariant functor between  monoidal categories equipped with  natural isomorphisms 
\[\psi : 1_\textbf{D} \xrightarrow{\cong} \F(1_\C) \qquad  \text{and} \qquad \varphi_{A,B} : \F(A)\ot_\textbf{D} \F(B) 
\xrightarrow{\cong} \F(A\ot_\C B)\ ,
\] subject to coherence diagrams with respect to the various associators, unitors, and braidings. 
Recall that the opposite of a  symmetric monoidal category 
is again a symmetric monoidal category. 
A contravariant functor $\F : \textbf{C} \to \textbf{D}$ is called  strong symmetric monoidal, when the induced covariant functor $\F^{\text{op}} : \textbf{C}^{\text{op}} \to \textbf{D}$ is strong symmetric monoidal.

Let us begin with the top faces functors. There, all the units are equal to $(0,0)$ and preserved by the various functors. The structural isomorphisms for the monoidal functor $\Lambda$ are given by 
\begin{align*}
\Lambda(V,R)\ot \Lambda(W,S) &\cong\left(V\oplus W, \Lambda(R)\oplus [V,W]_-\oplus \Lambda(S)\right)\\ &\cong 
\left(V\oplus W, \Lambda(R \oplus [V,W]_-\oplus S)\right)\\ &\cong\Lambda((V,R)\oplus(W,S)) \ ,
\end{align*}
the ones for the monoidal functor $\Sigma$ are given by 
\begin{align*}
\Sigma(V,R)\, \uot\, \Sigma(W,S) &\cong\left(V\oplus W, \Sigma(R)\oplus [V,W]_+\oplus \Sigma(S)\right)\\ &\cong 
\left(V\oplus W, \Sigma(R \oplus [V,W]_+\oplus S)\right)\\ 
&\cong\Sigma((V,R)\uot(W,S)) \ ,
\end{align*}
and the ones for the monoidal functor $\mathcal{S}$ are given by 
\begin{align*}
\mathcal{S}(V,R)\ot \mathcal{S}(W,S) &\cong\left(V\oplus W, \Sigma(R)\oplus 
V^{\wedge 2}
\oplus [V,W]_-\oplus \Sigma(S)\oplus W^{\wedge 2}\right)\\Ê& \cong 
\left(V\oplus W, \Sigma(R \oplus  S)\oplus (V\oplus W)^{\wedge 2}\right)\\ &\cong
\mathcal{S}((V,R)\vee(W,S)) \ ,
\end{align*}
since 
\begin{align*}
(V\oplus W)^{\ot 2}\cong (V\oplus W)^{\odot 2}\oplus (V\oplus W)^{\wedge 2} \cong 
\left(V^{\odot 2} \oplus [V,W]_+ \oplus W^{\odot 2}  \right)
\oplus 
\left(V^{\wedge 2} \oplus [V,W]_- \oplus W^{\wedge 2}  \right)\ .
\end{align*}

The natural isomorphisms for the first Koszul duality functors $\acc  : (\QD^-, \oplus) \to (\QD^+, \uot)$ 
 are given by
\begin{align*}
(V,R)^\acc\, \uot\,  (W,S)^\acc&\cong(sV \oplus sW, s^{2}R\oplus [sV, sW]_+  \oplus s^{2}S)\\&\cong
(s(V\oplus W), s^2(R\oplus [V,W]_-\oplus  S))\\&\cong \left( (V,R)\oplus (W,S)\right)^\acc;
\end{align*}
the ones for $\acc  : (\QD, \ot) \to (\QD, \uot)$ are similar. 

For the the linear dual functors, we consider the following isomorphisms 
\begin{align*}
(V,R)^*\vee (W,S)^*\cong(V^*\oplus W^*, R^\perp\oplus S^\perp)\cong
(V\oplus W, R\oplus [V,W]_+\oplus  S)^*\cong \left( (V,R)\, \uot\, (W,S)\right)^*
\end{align*}
and 
\begin{align*}
(V,R)^*\ot (W,S)^*\cong(V^*\oplus W^*, R^\perp\oplus [V^*,W^*]_- \oplus S^\perp)\cong
(V\oplus W, R\oplus [V,W]_+\oplus  S)^*\cong \left( (V,R)\, \uot\, (W,S)\right)^*\ .
\end{align*}
Since the second Koszul duality functors $!$ are the composites of two strong symmetric  functors (see below), they are also strong symmetric monoidal.

Each of the functors of the the left-hand side face sends directly the unit to the unit, since $\mathrm{L}(0,0)=0$, $\mathrm{A}(0,0)=K$, and $\mathrm{U}(0)=K$. The new structural isomorphisms are 
\begin{align*}
&\mathrm{L}(V,R)\oplus \mathrm{L}(W,S)\cong \mathrm{L}(V\oplus W, R\oplus [V,W]_-\oplus S)\cong \mathrm{L}
\left(
(V,R)\oplus (W,S)
\right)\ , \\
& \mathrm{A}(V,R)\ot \mathrm{A}(W,S)\cong \mathrm{A}(V\oplus W, R\oplus [V,W]_-\oplus S)\cong \mathrm{A}
\left(
(V,R)\ot (W,S)
\right)\ , 
\end{align*}
and 
\[
\mathrm{U}(\mathfrak{g}\oplus \mathfrak{h})\cong \mathrm{U}(\mathfrak{g}) \otimes \mathrm{U}(\mathfrak{h})\ .
\]

Regarding the right-hand side face, the respective units $(0,0)$ and 
$K$, are again directly sent to one another. In this case, the structural isomorphisms are 
\begin{align*}
&\mathrm{S}(V,R)\ot \mathrm{S}(W,S) \cong \mathrm{S}(V\oplus W, R\oplus S)\cong \mathrm{S}
\left(
(V,R)\vee (W,S)
\right)
\end{align*}
and the identities for the forgetful functors.

In the middle horizontal face, we also consider  the identities for  the forgetful functors. 
For the two weight-wise  linear dualisation functors $* : \wgComco\hookrightarrow\wgCom$ and $* : \wgAssco\hookrightarrow\wgAss$ the 
natural maps  $C^*\ot D^* \cong (C\ot D)^*$ are isomorphisms since we are working with spaces with finite dimensional weight components.

Finally, the two vertical coalgebra realisations functors send the unit $(0,0)$ to the unit $K$. The structural isomorphisms are respectively 
\begin{align*}
&\Sc(V,R)\ot \Sc(W,S) \cong \Sc(V\oplus W, R\oplus [V,W]_+\oplus S)\cong\Sc((V,R)\, \uot\,  (W,S))\ , \ \\
&\Ac(V,R)\ot \Ac(W,S) \cong \Ac(V\oplus W, R\oplus [V,W]_+\oplus S)\cong\Ac((V,R)\, \uot\,  (W,S)) \ . 
\end{align*}
They can be proved on two ways. One can first consider their respective weight-wise linear duals 
and apply the respective isomorphisms 
\begin{align*}
&\frac{S(V^*)}{(R^\perp)}\ot \frac{S(W^*)}{(S^\perp)} \cong
\frac{S(V^*\oplus W^*)}{(R^\perp\oplus S^\perp)}
\qquad \text{and} \qquad 
\frac{T(V^*)}{(R^\perp)}\ot \frac{T(W^*)}{(S^\perp)} \cong
\frac{T(V^*\oplus W^*)}{(R^\perp\oplus [V^*, W^*]_- \oplus S^\perp)}\ . 
\end{align*}
 One can also show that the left-hand side coalgebra satisfies each time the universal property of quadratic coalgebras. 

The commutativity of the  various coherence diagrams for the symmetric monoidal functors on the top face come from the fact that their underlying functors on the category $(\textbf{grVect}, \oplus)$ is either the identity, the degree shift, or the linear duality functor, which are symmetric monoidal. The bottom functors, namely the universal enveloping algebra functor, the inclusions, and the linear duality functor,  are known to be symmetric monoidal. Finally, it is straightforward to check the various coherence diagrams satisfied by the symmetric  monoidal functors $\mathrm{A}$, $\mathrm{S}$, $\mathrm{L}$, $\Sc$, and $\Ac$. 
\end{proof}

\section{2--monoidal structures upon quadratic data}\label{sec:2MonoQ}

\subsection{Notation and setting} As in \cite{Manin88, Manin18}, we will identify
the category of  quadratic algebras  over $K$ with the category $\QD$ of quadratic data, whose elements will be denoted by $A=(A_1, R(A))$. 

\smallskip 

In \cite[Section~1]{Vallette08}, we introduced the notion of a \emph{lax 2-monoidal category} as a category $\textbf{C}$ endowed with 
two monoidal products $\boxtimes$ and $\otimes$ such that the functor $\otimes : (\textbf{C}, \boxtimes)^2 \to (\textbf{C}, \boxtimes)$ is lax monoidal functor. Therefore, such a structure amounts to a natural transformation, called the \emph{interchange law},  
\[
\varphi_{AA^{\prime}BB^{\prime}} :\ 
(A\otimes A^{\prime}) \boxtimes (B\otimes B^{\prime}) \to (A\boxtimes B)\otimes  (A^{\prime}\boxtimes B^{\prime})
\]
satisfying the usual coherence diagrams. (Notice that the natural transformation might not be made up of isomorphisms, on the contrary to the strong monoidal functors considered in \cref{subsec:SymMonFun}.)
Dually, the notion of a \emph{colax 2-monoidal category} is obtained by a colax (or oplax) monoidal functor $\otimes : (\textbf{C}, \boxtimes)^2 \to (\textbf{C}, \boxtimes)$, that is by a natural transformation 
\[
\psi_{AA^{\prime}BB^{\prime}} :\ 
 (A\boxtimes B)\otimes  (A^{\prime}\boxtimes B^{\prime})
\to
(A\otimes A^{\prime}) \boxtimes (B\otimes B^{\prime}) \ .
\]
satisfying the opposite coherent diagrams. A category equipped with two monoidal structures carrying two compatible lax and  colax 2-monoidal structures is called a \emph{2 monoidal-category}. 

\begin{remark}
 In the \cite[Proposition~2]{Vallette08} (and its arXiv version as well), there are
two  misprints in the 
commutative triangle expressing compatibility with unit morphisms. First, the product
$(I\otimes A)\boxtimes (I\otimes A^{\prime})$ should be replaced by
$(I\boxtimes A)\otimes (I\boxtimes A^{\prime})$. Second, $F(A)$ should be replaced
by $F(A^{\prime})$ .
\end{remark}

\subsection{The interchange laws in $(\textbf{grVect}, \otimes , \oplus )$}\label{subsec:IntLawQD}
The category $\textbf{grVect}$
is endowed with two simple mo\-no\-idal structures: {\it tensor product $\otimes$} over $K$
and {\it direct sum $\oplus$.} They have unit objects and  standard associativity
morphisms. We may even assume these structures to be {\it strict} ones,
and sometimes will do it for simplicity. 

The interchange law
in this context must be a natural monoidal transformation
\begin{equation}\label{eq:2.1}
\varphi_{AA^{\prime}BB^{\prime}} :\ 
(A\oplus A^{\prime}) \otimes (B\oplus B^{\prime}) \to (A\otimes B)\oplus  (A^{\prime}\otimes B^{\prime})
\end{equation}
with the following notation change: %Vallette's
$(\boxtimes, \otimes )$  are replaced here respectively by $(\otimes , \oplus )$.

\smallskip

Explicitly, we have natural identifications
\begin{equation}\label{eq:2.2}
(A\oplus A^{\prime}) \otimes (B\oplus B^{\prime}) \cong
(A\otimes B)\oplus (A^{\prime}\otimes B)\oplus (A\otimes B^{\prime})\oplus
(A^{\prime}\otimes B^{\prime})\ ,
\end{equation}
and we define the interchange law \eqref{eq:2.1} as the projection $pr_{14}$ of r.h.s. of 
\eqref{eq:2.2}
onto the sum of its first and fourth direct summand
\begin{equation}\label{eq:2.3}
pr_{14}:\   (A\oplus A^{\prime}) \otimes (B\oplus B^{\prime}) \to (A\otimes B)\oplus  (A^{\prime}\otimes B^{\prime})\ .
%\eqno(2.3)
\end{equation}

In the other way round, the inclusion of the first and fourth direct summand defines the natural transformation
\begin{equation*}%\label{eq:2.3}
\psi_{AA^{\prime}BB^{\prime}} :\ 
 (A\otimes B)\oplus  (A^{\prime}\otimes B^{\prime})
\to
(A\oplus A^{\prime}) \otimes (B\oplus B^{\prime}) \ .
\end{equation*}

The compatibility of  this projection and this inclusion with associativity and unit morphisms for $\otimes$ defined
in \cite[Proposition~2]{Vallette08} quoted above can be checked in a straightforward way. So the data $(\textbf{grVect}, \otimes , \oplus, \varphi, \psi)$ form a $2$-monoidal category. 

\smallskip

Now we will state and prove the main result of this section.

\subsection{Final notation change: lifting \eqref{eq:2.1} to quadratic data} 
Let us now reinterpret
the players of the interchange law \eqref{eq:2.1} as the first components of objects of $\QD$, that is
$A=(A_1, R(A))$ etc. We lift the monoidal structures $(\otimes , \oplus )$ on 1--components
of quadratic data to monoidal structures in $\QD$ denoted in \cite{Manin88} as 
$(\bullet , \underline{ \otimes} )$, where 
$$
(A_1, R(A))\bullet (B_1, R(B))\coloneq (A_1\otimes B_1, S_{(23)}(R(A)\otimes R(B) ))\ .
$$
Here $R(A) \otimes R(B) \subset A_1^{\otimes 2} \otimes B_1^{\otimes 2}$, and $S_{(23)}$
interchanges the middle two components of the tensor product, so
that the result lands in $(A_1\otimes B_1)^{\otimes 2} $ as it should be.
For a  description of $\underline{\otimes}$ refer to the beginning of \cref{subsec:MonoStruc}. 

\begin{proposition} \label{prop:2monocatQD}
The interchange morphism $pr_{14}$ applied to  1--components
of 
\begin{equation}\label{eq:2.4}
(A \,\underline{\otimes}\, A^{\prime}) \bullet (B \,\underline{\otimes}\, B^{\prime}) \to (A\bullet B)
\,\underline{\otimes}\,  (A^{\prime}\bullet B^{\prime})
\end{equation}
and then lifted to the tensor squares of these 1--components as $pr_{14}^{\otimes 2}$,
sends the subspace of relations of the l.h.s. to the subspace of relations of the r.h.s.:
\begin{equation}\label{eq:2.5}
R((A\,\underline{\otimes}\, A^{\prime}) \bullet (B\,\underline{\otimes}\, B^{\prime})) \to R((A\bullet B)
\,\underline{\otimes}\,  (A^{\prime}\bullet B^{\prime}))
\end{equation}
and thus lifts to an interchange morphism in $\emph{\QD}$. The quadruple $(\emph{\QD}, \bullet, \underline{\otimes}, \varphi)$ forms a lax $2$--monoidal category.
\end{proposition}

\medskip

\begin{proof}
 (i) {\it Preparation.}  Since from now on the four graded vector spaces in \eqref{eq:2.1}--\eqref{eq:2.4} will be 1--components of
quadratic data, we will add the subscript 1 in their notation.  According to the definitions
(5) and (7) on p.~19 of \cite{Manin88}, the source of the arrow \eqref{eq:2.5} can be explicitly written as
\begin{equation}\label{eq:2.6}
R((A\,\underline{\otimes\,} A^{\prime}) \bullet (B\,\underline{\otimes}\, B^{\prime}))
= S_{(23)}(\{ R(A)\oplus [A_1,A_1^{\prime}]_{+} \oplus R(A^{\prime})\} \otimes
\{ R(B)\oplus [B_1,B_1^{\prime}]_{+} \oplus R(B^{\prime})\})\ .
\end{equation}

\smallskip

On the other hand, the target becomes
\begin{equation}\label{eq:2.7}
R((A\bullet B)
\,\underline{\otimes}\,  (A^{\prime}\bullet B^{\prime}))
=S_{(23)}(R(A)\otimes R(B)) \oplus  [A_1\otimes B_1,  A_1^{\prime}\otimes B_1^{\prime}]_{+}
\oplus S_{(23)}(R(A^{\prime})\otimes R(B^{\prime}))\ .
\end{equation}
On the respective 1--components of \eqref{eq:2.4}, the kernel of  
$$ 
pr_{14}:\ 
(A_1\oplus A_1^{\prime}) \otimes  (B_1\oplus B_1^{\prime}) \to A_1\otimes B_1 \oplus
A_1^{\prime}\otimes B_1^{\prime}
$$
is
\begin{equation}\label{eq:2.9}
A_1\otimes B_1^{\prime}\oplus A_1^{\prime}\otimes B_1 \ .
\end{equation}
Therefore, whenever we apply $pr_{14}^{\otimes 2}$
to a tensor monomial of degree four, whose first two or last two
divisors (or both)  look like $a\otimes b^{\prime}$ or $a^{\prime}\otimes b$,
then it is annihilated.

\smallskip

For brevity, we will call such monomials in  
$[(A_1\oplus A_1^{\prime}) \otimes  (B_1\oplus B_1^{\prime})]^{\otimes 2}$
{\it vanishing ones.}
\smallskip

In the following sections of the proof we will apply this remark successively to various 
summands of \eqref{eq:2.6}:
$$
S_{(23)}(\{ R(A)\oplus [A_1,A_1^{\prime}]_{+} \oplus R(A^{\prime})\} \otimes
\{ R(B)\oplus [B_1,B_1^{\prime}]_{+} \oplus R(B^{\prime})\})\  .
$$

\medskip

(ii) {\it Summands annihilated by $pr_{14}^{\otimes 2}$. } 

\smallskip

(a) First, check that the whole subspace $S_{(23)}(R(A)\otimes R(B^{\prime}))$ is annihilated.
In fact any element of it is a linear combination of vanishing monomials because
after interchanging two middle elements  in $a_1\otimes a_2 \otimes b_1^{\prime}\otimes b_2^{\prime}$
where $a_1, a_2 \in A_1,\ b_1^{\prime} , b_2^{\prime}   \in B_1^{\prime}$, and
both left half and right half 
binary products  land in \eqref{eq:2.9}.

\smallskip

Essentially the same argument shows that $S_{(23)}(R(A^{\prime})\otimes R(B))$ is annihilated,
and 
$R(A^{\prime}) \otimes
[B_1,B_1^{\prime}]_{+}$
and
$[A_1,A_1^{\prime}]_{+} \otimes
 R(B^{\prime})$
as well.

\smallskip

One can treat in the same way $R(A)\otimes [B_1, B_1^{\prime} ]_{+}$
and $[A_1, A_1^{\prime} ]_{+}\otimes R(B)$. The only difference is that 
after applying $S_{(23)}$ to the respective monomials
only {\it one half of the result}, either to the left, or to the right of the middle $\otimes$
lands in the tensor square of \eqref{eq:2.9}.

\medskip 

(iii) {\it Summands upon which $pr_{14}^{\otimes 2}$ is injective.} A  direct observation
shows that  $pr_{14}^{\otimes 2}$ restricted to $S_{(23)}\allowbreak(R(A)\otimes R(B))$ is injective
and in fact identifies it with the respective summand of \eqref{eq:2.7}. 

\smallskip
Similarly,  $pr_{14}^{\otimes 2}$ identifies $S_{(23)}(R(A^{\prime})\otimes R(B^{\prime}))$
identifies it with the respective summand of \eqref{eq:2.7}. 

\medskip

(iv) {\it Remaining terms.}

\smallskip 

It remains to compare the terms 
\begin{equation}\label{eq:2.10}
S_{(23)} ([A_1,A_1^{\prime}]_{+}\otimes [B_1,B_1^{\prime}]_{+}) 
\end{equation}
in the source of \eqref{eq:2.5} with terms 
\begin{equation}\label{eq:2.11}
[A_1\otimes B_1,  A_1^{\prime} \otimes B_1^{\prime}]_{+}
\end{equation}
in its target.

\smallskip
The space \eqref{eq:2.10} is spanned by linear combinations
\begin{align}\label{eq:2.12}
&S_{(23)} \left(\left(a_i\otimes a_j^{\prime} + (-1)^{|a_i||a'_j|}a_j^{\prime}\otimes a_i\right) \otimes
\left(b_k\otimes b_l^{\prime} + (-1)^{|b_k||b'_l|} b_l^{\prime}\otimes b_k\right)\right)\notag
\\
&= (-1)^{|a'_j||b_k|}a_i\otimes b_k\otimes a_j^{\prime} \otimes b_l^{\prime}
+ (-1)^{|b'_l|(|a'_j|+|b_k|)} a_i\otimes b_l^{\prime}\otimes a_j^{\prime}\otimes b_k\notag \\
&+(-1)^{|a_i|(|a'_j|+|b_k|)} a_j^{\prime} \otimes b_k\otimes a_i\otimes b_l^{\prime}
 + (-1)^{|b_k||b'_l|+|a_i|(|a'_j|+b'_l|)} a_j^{\prime} \otimes b_l^{\prime}\otimes a_i\otimes b_k\ ,
\end{align}
where 
$$
a_i\in A_1,\  a_j^{\prime}\in A_1^{\prime}, \ 
b_k\in B_1,\  b_l^{\prime}\in B_1^{\prime} \ .
$$
Two middle terms in \eqref{eq:2.12} are vanishing ones.

\smallskip

With the same notation, \eqref{eq:2.11} is spanned by linear combinations
$$
(-1)^{|a'_j||b_k|} a_i\otimes b_k\otimes a_j^{\prime} \otimes b_l^{\prime}
+ (-1)^{|b_k||b'_l|+|a_i|(|a'_j|+b'_l|)}a_j^{\prime} \otimes b_l^{\prime}\otimes a_i\otimes b_k\ ,
$$
which are exactly images of sums of the two remaining terms of \eqref{eq:2.12}
after application of $pr_{14}^{\otimes 2}.$

\smallskip

This completes the proof of the \cref{prop:2monocatQD}.
\end{proof}

\subsection{The dual picture}

\begin{corollary}\label{cor:White2Lax}
 The quadruple $(\emph{\QD}, \otimes, \circ, \psi)$ is a lax 2--monoidal
category as well.
\end{corollary}

\begin{proof} In the case when all involved quadratic data are finite--dimensional, the interchange law in  $(\QD, \otimes, \circ, \psi )$ can be formally obtained by
applying the linear duality functor $*$  to the diagrams \eqref{eq:2.4} and \eqref{eq:2.5}.
Similarly, the commutativity of all relevant diagrams  (compatibility
with associativity of $\otimes$ and with unity for $\circ$) follows by duality from
the respective facts for $(\QD, \bullet, \underline{\otimes}, \varphi)$.

\smallskip

However, the statement itself of \cref{cor:White2Lax} remains true even without assumption
of finite--dim\-en\-sion\-ali\-ty: to prove it one should develop detailed arguments
parallel to those in given in the proof of \cref{prop:2monocatQD}.

\smallskip

Below we will only sketch the check that interchange laws in $(\QD, \otimes, \circ, \psi )$
and  $(\QD, \bullet, \underline{\otimes}, \varphi)$ are $*$--dual.

\smallskip

Applying formally $*$ to \eqref{eq:2.4} and rewriting the left and right hand sides with the help
of identifications, collected in \cite[Section~3]{Manin88}, especially in its subsection 5,
we obtain a morphism
\begin{equation}\label{eq:2.13}
\{(A\bullet B) \,\underline{\otimes}\, (A^{\prime}\bullet B^{\prime})\}^*
\to
\{(A \,\underline{\otimes}\, A^{\prime})\bullet (B \,\underline{\otimes}\, B^{\prime})\}^*\ .
\end{equation}

The l.~h.~s. of \eqref{eq:2.13} can be rewritten as
$$
(A \bullet B)^* \otimes (A^{\prime}\bullet B^{\prime})^* \cong
(A^*\circ B^*) \otimes (A^{\prime *}\circ B^{\prime *})\ .
$$
Similarly, the r.~h.~s. of \eqref{eq:2.13} is
$$
(A \,\underline{\otimes}\,A^{\prime})^* \circ (B \,\underline{\otimes}\, B^{\prime})^* \cong
(A^*\otimes A^{\prime *})\circ (B^*\otimes B^{\prime *})\ .
$$
So finally \eqref{eq:2.13} becomes
\begin{equation}\label{eq:2.14}
(A^*\circ B^*) \otimes (A^{\prime *}\circ B^{\prime *})\to
(A^*\otimes A^{\prime *})\circ (B^*\otimes B^{\prime *})\ .
\end{equation}
Since $*$ is a contravariant quasi--involution of $\QD$ that is, $**$
is equivalent to $\mathrm{id}$, \eqref{eq:2.14} is the required interchange morphism,
written for generic arguments. 
\end{proof}

\begin{remark}
%As we have already mentioned, in Sec.~1 of [Va08], a category endowed with two
%monoidal structures $(\boxtimes, \otimes )$ and the relevant interchange
%law as above, is called {\it lax 2--monoidal category}. If one inverts arrows in the
%diagrams involving interchange law, one gets the notion of
%{\it colax 2--monoidal category}.
%
%\smallskip
%
Our two examples $(\QD, \bullet , \underline{\otimes}, \varphi )$
and  $(\QD, \otimes , \circ, \psi)$ are lax 2-monoidal categories, but fail to be  colax since the interchange laws $\psi$ and $\varphi$ from the $2$--monoidal category $\textbf{grVect}$ do not lift to the appropriate level. 

\smallskip
Several other pairs, consisting of $\bullet$ and one of the monoidal
structures from \cite{Manin88}, are either simultaneously lax and colax, or
neither lax/nor colax. These are less interesting cases. We will present
in \cref{sec:2MonoOp} their more interesting operadic versions.
\end{remark}

\subsection{Applications}

\begin{corollary}\label{cor:MonoProduct} \leavevmode

\begin{enumerate}
\item Let $M,N$ be two monoids in $\emph{\QD}$
wrt the black product $\bullet$. Then $M\underline{\otimes} N$ also has a natural
structure of such a monoid.

\item Similarly, let $M,N$ be two monoids in $\emph{\QD}$
wrt the tensor product $\otimes$. Then $M\circ N$ also has a natural
structure of such a monoid.
\end{enumerate}
\end{corollary}

\begin{proof} 
These statements are direct applications of the fact that lax monoidal functors preserve monoids. They are actually special cases of \cite[Proposition~3]{Vallette08}, which was a motivation for the definition of the notion of lax $2$-monoidal category. 
\end{proof}

\begin{exam}
 Let $A\coloneq(A_1, R(A_1))$ be a quadratic data. The canonical map $+ : A_1\oplus A_1 \to A_1$ induces a morphism of quadratic data $A \otimes A \to A$ 
if and only if $[A_1, A_1]_-\subset R(A_1)$. Quadratic data satisfying this property actually form the image of the functor $\mathcal{S}$ from symmetric quadratic data. \cref{cor:MonoProduct} shows that their white product carries again a canonical $\otimes$-monoid structure.

\smallskip

This canonical $\ot$-monoid structure on the quadratic data living in the image of the functor $\mathcal{S}$  actually comes from the  monoid structure on any  symmetric quadratic data $(V, R)\in (\QD^+,\vee)$ given by $+ : V \oplus V \to V$. These  two monoid structures induce respectively the concatenation product under the symmetric monoidal functor of algebraic realisations $\S(V,R)$ and $A(V,R)$. 

\end{exam}
\section{2--monoidal structures upon  operadic quadratic data}\label{sec:2MonoOp}

\subsection{Operads}
Let us recall the \emph{coordinate-free partial definition} of an operad. We denote by $\Fin$ the category of finite sets with bijections.
Given any subset $X\subset Y$, we use the notation $Y/X\coloneq (Y\backslash X) \sqcup \{*\}$.
Let $(\C, \ot, 1_\C, \alpha, \lambda, \rho, \tau)$ be a  symmetric monoidal category.

\begin{definition}[Operad]
An \emph{operad} in $\C$ is a presheaf $\P : \Fin^{\text{op}}\to \C$ endowed with partial operadic compositions 
$\circ_{X\subset Y}\ : \ \P(Y/X)\ot \P(X) \to \P(Y)$, for any $X\subset Y$,
and a unit  
$\eta \ : \ 1_\C \to \P(\{*\})$
such that the following diagrams  commute.
\begin{description}
\item[{\sc Sequential axiom}] For any $X\subset Y \subset Z$, 
\[\xymatrix@C=32pt{
\left(\P(Z/Y)\ot \P(Y/X)\right) \ot \P(X) \ar[r]^\alpha \ar[d]^\cong& \P(Z/Y)\ot \left(\P(Y/X) \ot \P(X)\right)  \ar[r]^(0.61){\id\ot \circ_{X\subset Y}}& \P(Z/Y)\ot \P(Y) \ar[d]^{\circ_{Y\subset Z}}\\
\left(\P((Z/X)/(Y/X))\ot \P(Y/X)\right) \ot \P(X) \ar[r]^(0.63){\circ_{Y/X\subset Z/X}\ot \id}  & \P(Z/X)\ot \P(X)  \ar[r]^(0.6){\circ_{X\subset Z}}   &\  \P(Z) \ .
}
\]

\item[\sc{Parallel axiom}]
For any $X\sqcup Y \subset Z$, 
\[\xymatrix{
\P(((Z/X)/Y))\ot \left(\P(X) \ot \P(Y)\right)  \ar[d]^{\alpha^{-1}} \ar[r]^{\id\ot \tau}&   
\P(((Z/X)/Y))\ot \left(\P(Y) \ot \P(X)\right) \ar[r]^(0.5){\alpha^{-1}}& 
\left(\P(((Z/X)/Y))\ot \P(Y)\right) \ot \P(X) \ar[d]^{\circ_{Y\subset Z/X}\ot \id}\\
\left(\P(((Z/X)/Y))\ot \P(X)\right) \ot \P(Y) \ar[d]^\cong&&\P(Z/X)\ot \P(X)\ar[d]^{\circ_{X\subset Z}}
\\
\left(\P(((Z/Y)/X))\ot \P(X)\right) \ot \P(Y)  \ar[r]^(0.63){\circ_{X\subset Z/Y}\ot \id}  & \P(Z/Y)\ot \P(Y)  \ar[r]^(0.6){\circ_{Y\subset Z}}   &\  \P(Z) \ .
}
\]

\item[\sc{Left/Right unital axioms}] For any $X\in \Fin$ and any $x\in X$,  
\[
\vcenter{\xymatrix{
1_\C\ot \P(X) \ar[r]^(0.42){\eta\ot \id}\ar[d]^\lambda&\P(\{*\})\ot  \P(X)\ar[d]^\cong \\
\P(X) & \P(X/X)\ot \P(X)\ar[l]^(0.6){\circ_{X\subset X}}
}}
\qquad 
\vcenter{\xymatrix{
\P(X)\ot 1_\C \ar[r]^(0.42){\id \ot\eta}\ar[d]^\rho& \P(X)\ot \P(\{*\})\ar[d]^\cong \\
\P(X) & \P(X/\{x\})\ot \P(\{x\})\ar[l]^(0.65){\circ_{\{x\}\subset X}}\ .
}}
\]

\item[\sc{Top equivariance}] For any subset $X\subset Y$ and any bijection $f : X \to X$, we consider the induced bijection $\bar{f} : Y \to Y$, which leaves the elements of $Y\backslash X$ invariant,
\[\xymatrix{
\P(Y/X)\ot \P(X) \ar[r]^(0.65){\circ_{X\subset Y}}\ar[d]^{\id \ot \P(f)}& \P(Y) \ar[d]^{\P(\bar{f})}\\
\P(Y/X)\ot \P(X) \ar[r]^(0.65){\circ_{X\subset Y}} & \ \P(Y) \ .
}\]

\item[\sc{Bottom equivariance}] For any subset $X\subset Y$ and any bijection $f : Y/X \to Y/X$, we consider 
$\widetilde{Y}\coloneq Y$ if $f(*)=*$ and 
$\widetilde{Y}\coloneq Y\backslash \{f(*)\} \sqcup \{*\}$ otherwise. 
We also consider the bijection $\widetilde{Y}\to Y$ which 
which sends $*$ to $f(*)$ and is equal to the identity otherwise. We denote by 
$\mathring{f} : \widetilde{Y}/X \to Y/X$ the induced bijection. Finally, we denote by 
$\widetilde{f} : Y \to \widetilde{Y}$ the bijection which coincides with $f$ except for the assignment $f(*)\mapsto *$.
\[\xymatrix@C=40pt{
\P(Y/X)\ot \P(X) \ar[r]^{\P(\mathring{f})\ot \id}\ar[d]^{\P(f)\ot \id}&
\P(\widetilde{Y}/X)\ot \P(X) \ar[r]^(0.6){\circ_{X\subset \widetilde{Y}}} &
\P(\widetilde{Y}) \ar[d]^{\P(\widetilde{f})} \\
\P(Y/X)\ot \P(X)\ar[rr]^{\circ_{X\subset Y}} & & \ \P(Y) \ .
}\]
\end{description}
\end{definition}
 
The skeletal category of $\Fin$ is the groupoid $\Sy$ whose objects are the sets $\{1, \ldots, n\}$, for $n\in\mathbb{N},$ and whose morphisms are the elements of the symmetric groups $\Sy_n$. A presheaf on $\Fin$ is thus equivalent to a collection $\{\P(n)\}_{n\in \mathbb{N}}$ of right $\Sy_n$-modules, see \cite[Section~1.1]{KapranovManin01}. In these terms, the above structure of an operad is equivalent to partial composition products 
$\circ_i : \P(n)\otimes\P(m)\to \P(n+m-1)$, for $1\leqslant i\leqslant n$, and a unit map $\eta : 1_\C \to \P(1)$ satisfying the analoguous axioms, given in  \cite[Section~5.3.4]{LodayVallette12} for example.

\subsection{Operadic quadratic data}\label{subsec:BQO}
The notions of black and white products were generalised to binary quadratic operads in 
\cite{GinzburgKapranov94, GinzburgKapranov95} and then to quadratic operads (and cooperads) in \cite{Vallette08}. 
\smallskip
In this section, we will work with the following analogous operadic notion of quadratic data. 

\begin{definition}[Binary operadic quadratic data]\leavevmode
  A \emph{binary operadic quadratic data} is a pair $\mathcal{A}=(\mathcal{A}_1, R(\mathcal{A}))$
where $\mathcal{A}_1$ is a graded $K$--linear representation of the symmetric group $\Sy_2$ (that is, an
$\Sy_2$--module) 
and $R(\mathcal{A})$ is a $\Sy_3$--submodule of  the part of arity 3 of the free operad 
$\mathcal{T} (\mathcal{A}_1)$ generated 
by $\mathcal{A}_1$.

 A morphism $f:\, (\mathcal{A}_1, R(\mathcal{A}))\to (\mathcal{B}_1, R(\mathcal{B}))$ is a map of 
$\Sy_2$--modules $\mathcal{A}_1\to \mathcal{B}_1$ whose extension $\mathcal{T}(f)$ restricted to the arity 3 part
of $\mathcal{T} (\mathcal{A}_1)$ sends $R(\mathcal{A})$ to $R(\mathcal{B})$. 
This category is denoted by $\BQO$.
\end{definition}

If we assume additionally that our graded $\Sy_2$--modules are finite--dimensional, we can imitate the definition of the linear dualisation functor in our new context
as the functor
$$
\mathcal{A}=(\mathcal{A}_1, R(\mathcal{A}))\mapsto  \mathcal{A}^* \coloneq(\mathcal{A}_1^*, R(\mathcal{A})^{\perp})\ .
 $$
 (Notice that this functor was denoted by $!$ in \cite[Section~2]{Vallette08}.)
 Otherwise, we can drop the finite--dimensionality restriction, and consider the Koszul dual functor $\acc$ which produces quadratic cooperads, see \cite[Section~7.1]{LodayVallette12}.

\subsection{The interchange laws  on the category of graded $\Sy_2$--modules}
The category of graded $\Sy_2$--modules is endowed with two
monoidal structures: the (Hadamard) {\it tensor product} $\otimes$ 
 and the {\it direct sum} $\oplus$, see for instance 
 \cite{KapranovManin01}. 
 The unit of the former one
is given by the trivial representation of $\Sy_2$ and the unit of the latter one is given by the 
 $\Sy_2$--module $\{0\}$.
We refer the reader to Section~1.4 and Appendix~A of \cite{Vallette08} for more details.  

\smallskip

 We will consider the following two interchange maps:
\begin{equation}\label{eq:3.1}
 \varphi_{\mathcal{A}_1\mathcal{A}_1^{\prime}\mathcal{B}_1\mathcal{B}_1^{\prime}} : 
(\mathcal{A}_1\oplus \mathcal{A}_1^{\prime})\otimes (\mathcal{B}_1\oplus \mathcal{B}_1^{\prime}) \to
(\mathcal{A}_1\otimes \mathcal{B}_1)\oplus  (\mathcal{A}_1^{\prime}\otimes \mathcal{B}_1^{\prime})
\end{equation}
 and
\begin{equation}\label{eq:3.2}
  \psi_{\mathcal{A}_1\mathcal{A}_1^{\prime}\mathcal{B}_1\mathcal{B}_1^{\prime}} : 
(\mathcal{A}_1\otimes \mathcal{B}_1)\oplus (\mathcal{A}_1^{\prime}\otimes \mathcal{B}_1^{\prime}) \to
(\mathcal{A}_1\oplus \mathcal{A}_1^{\prime})\otimes  (\mathcal{B}_1\oplus \mathcal{B}_1^{\prime}) \ .
\end{equation}
As in \cref{subsec:IntLawQD} above, they can be naturally seen in the context of canonical
identifications
$$
(\mathcal{A}_1\oplus \mathcal{A}_1^{\prime})\otimes (\mathcal{B}_1\oplus \mathcal{B}_1^{\prime}) \cong
(\mathcal{A}_1\otimes \mathcal{B}_1)\oplus (\mathcal{A}_1^{\prime}\otimes \mathcal{B}_1) \oplus
(\mathcal{A}_1\otimes \mathcal{B}_1^{\prime})\oplus  (\mathcal{A}_1^{\prime}\otimes \mathcal{B}_1^{\prime})\ .
$$
Namely, the law $\varphi$ is the projection $pr_{14}$ onto the first and the fourth summand
of the right hand side, whereas the law $\psi$ is the injection $inj_{14}$ of the sum of the first
and the fourth summands in the right hand side.

\smallskip

Together with the two interchange laws $\varphi$ and $\psi$, we obtain a 2--monoidal structure on the category
of graded $\Sy_2$--modules.

\subsection{The first product on the category of binary operadic quadratic data} We will now
start preparing the construction of the black product on the category $\BQO$. The central piece of the construction is the analog of the map $f^{\otimes 2}$
from \cref{subsec:DefQD}, which was denoted $\mathcal{T}(f)$ in the above definition of binary operadic quadratic data.
\smallskip

Let $(\mathcal{A}_1, R(\mathcal{A}))$ be an object of $\BQO$ and let $a\in \mathcal{A}_1$. Since
the arity of $a$ is 2, we will temporarily use the notation $a(x,y)$ where $x,y$
run over elements of an arbitrary algebra over the operad generated by $(\mathcal{A}_1, R(\mathcal{A}))$. 
Similarly, elements $c$ of arity 3 of such an operad can be written as $c(x,y,z)$ etc.
\smallskip

With this notation, there exists a basis of $\mathcal{T}(\mathcal{A}_1)(3)$ consisting 
of bilinear expressions in $a, a^{\prime} \in \mathcal{A}_1$
$$
\tau_1(a,a^{\prime})\ ,\quad  \tau_2(a,a^{\prime})\ ,\quad  \tau_3(a,a^{\prime})\ ,
$$
such that
$$
\tau_1(a,a^{\prime})(x,y,z)\coloneq a(a^{\prime}(x,y),z)\ ,\quad
\tau_2(a,a^{\prime})(x,y,z)\coloneq a(a^{\prime}(y,z),x)\ ,\quad
\tau_3(a,a^{\prime})(x,y,z)\coloneq a(a^{\prime}(z,x,),y)\ .
$$
This is a rewriting of the definition in \cite[Section~4]{Vallette08}, where the language
of planar rooted trees is used.

\smallskip

Now, for two binary operadic quadratic data   $(\mathcal{A}_1, R(\mathcal{A}))$ and $(\mathcal{B}_1, R(\mathcal{B}))$,
we can calculate in terms of these bases the map of $\Sy_3$--modules
$$
\Psi_{\mathcal{A}_1\mathcal{B}_1} :\quad \mathcal{T}(\mathcal{A}_1(3))\otimes 
\mathcal{T}(\mathcal{B}_1(3))\otimes \mathrm{sgn}_{\Sy_3}\quad \longrightarrow
\quad \mathcal{T}(\mathcal{A}_1\otimes 
\mathcal{B}_1\otimes \mathrm{sgn}_{\Sy_2}) (3)
$$
introduced in \cite{GinzburgKapranov94, GinzburgKapranov95}; our presentation is due to \cite{Vallette08}.
\smallskip

Namely, 
$$
\Psi_{\mathcal{A}_1\mathcal{B}_1}(\tau_i(a,a^{\prime})\otimes \tau_j(b,b^{\prime}))\coloneq
\delta_{ij}\tau_i(a\otimes b, a^{\prime}\otimes b^{\prime})\ .
$$

\smallskip

The main statement of this subsection is the following one. 

\smallskip

\begin{lemma} 
There exists a well defined monoidal structure
on $\emph{\BQO}$, called the black product $\bullet$,
given on objects by the formula
\begin{equation}\label{eq:3.3}
\mathcal{A}\bullet \mathcal{B} \coloneq (\mathcal{A}_1\otimes \mathcal{B}_1, \Psi_{\mathcal{A}_1\mathcal{B}_1}
(R(\mathcal{A})\otimes R(\mathcal{B})))\  .
\end{equation}
\end{lemma}

\subsection{The second product on the category of binary quadratic operads}\label{subsec:SecProd}
We will now define the product of  binary operadic quadratic data
 $\underline{\circledcirc}$ which will serve as an analog
of the product $\underline{\otimes}$ on the category of quadratic data.

\smallskip

First of all, for a graded $\Sy_2$--module $\mathcal{A}_1$, denote by
$\mathcal{A}_1^+$, resp. $\mathcal{A}_1^-$, the submodule of $\Sy_2$--invariant 
elements, resp. (2,1)--antiinvariant elements of $\mathcal{A}_1$. 

\smallskip

Furthermore, denote by $\{\mathcal{A}_1,\mathcal{B}_1\}$ the sub-$\Sy_3$-module of
$\mathcal{T}(\mathcal{A}_1\oplus\mathcal{B}_1)(3)$ spanned by the elements
$\tau_1(a,b)+\tau_1(b,a)$, where either $(a,b)\in  \mathcal{A}_1^+ \times \mathcal{B}_1^+$,
or  $(a,b)\in  \mathcal{A}_1^- \times \mathcal{B}_1^-$.
\smallskip
Finally, put
$$
\mathcal{A} \,\underline{\circledcirc}\, \mathcal{B} \coloneq (\mathcal{A}_1 \oplus \mathcal{B}_1, R(\mathcal{A})\oplus \{\mathcal{A}_1, \mathcal{B}_1\}
\oplus R(\mathcal{B}))\ .
$$

We can now state and prove the analog of \cref{prop:2monocatQD} in the operadic
setting.

\medskip

\begin{proposition}\label{prop:2MonoOP}
   The interchange law $\varphi$ on the category of graded 
$\Sy_2$--modules  induces  morphisms in the category
of binary operadic quadratic data
$$
\varphi_{\mathcal{A}\mathcal{A}^{\prime}\mathcal{B}\mathcal{B}^{\prime} }:\quad
(\mathcal{A} \,\underline{\circledcirc}\, \mathcal{A}^{\prime}) \bullet
(\mathcal{B} \,\underline{\circledcirc}\, \mathcal{B}^{\prime}) \to
(\mathcal{A}\bullet \mathcal{B}) \,\underline{\circledcirc}\, (\mathcal{A}^{\prime}\bullet \mathcal{B}^{\prime})
$$
which define on $(\emph{\BQO},\bullet , \underline{\circledcirc},\varphi )$
the structure of a lax 2--monoidal category and on $(\emph{\BQO}, \underline{\circledcirc},\bullet ,\varphi )$
the structure of a colax 2--monoidal category.
\end{proposition}

\begin{proof}
 As in the case of algebras (\cref{prop:2monocatQD}), we have to check that
the morphism of graded $\Sy_2$--modules \eqref{eq:3.1} induces a well defined morphism
of graded $\Sy_3$--modules of relations
$$
R((\mathcal{A} \,\underline{\circledcirc}\, \mathcal{A}^{\prime}) \bullet
(\mathcal{B} \,\underline{\circledcirc}\, \mathcal{B}^{\prime}))\to
R((\mathcal{A}\bullet \mathcal{B}) \,\underline{\circledcirc}\, (\mathcal{A}^{\prime}\bullet \mathcal{B}^{\prime}))
$$

The left hand side can be rewritten as
$$
\Psi (( R(\mathcal{A}) \oplus \{\mathcal{A}_1, \mathcal{A}_1^{\prime}\}\oplus R(\mathcal{A}^{\prime}))\otimes
( R(\mathcal{B}) \oplus \{\mathcal{B}_1, \mathcal{B}_1^{\prime}\}\oplus R(\mathcal{B}^{\prime})))\ ,
$$
and the right hand side as
$$
\Psi (R(\mathcal{A})\otimes R(\mathcal{B})) \oplus 
\{\mathcal{A}_1\otimes \mathcal{B}_1, \mathcal{A}_1^{\prime}\otimes \mathcal{B}_1^{\prime}\}\oplus
\Psi (R(\mathcal{A}^{\prime})\otimes R(\mathcal{B}^{\prime}))\ .
$$
\smallskip

From \eqref{eq:3.3}, using the same arguments as in the case of quadratic data we 
conclude that the following summands of the left hand side get
annihilated:
\begin{align*}
& \Psi (R(\mathcal{A}) \otimes \{\mathcal{B}_1, \mathcal{B}_1^{\prime}\})\ , \quad 
\Psi (R(\mathcal{A})\otimes R(\mathcal{B}^{\prime}))\ , \quad \Psi (\{\mathcal{A}_1, \mathcal{A}_1^{\prime}\} \otimes
R(\mathcal{B}))\ ,
\\
& \Psi (\{\mathcal{A}_1, \mathcal{A}_1^{\prime}\} \otimes
R(\mathcal{B^{\prime}}))\ , \quad    \Psi (R(\mathcal{A^{\prime}})\otimes
R(\mathcal{B}))\ , \quad \Psi (R(\mathcal{A}^{\prime}) \otimes \{\mathcal{B}_1, \mathcal{B}_1^{\prime}\})\ .
\end{align*}
The two summands in the left hand side
$$
\Psi (R(\mathcal{A}) \otimes R(\mathcal{B})) \quad \text{and} \quad  \Psi (R(\mathcal{A}^{\prime}) \otimes R(\mathcal{B}^{\prime}))
$$
map identically to the first and the last summands of the right hand side
respectively.

\smallskip

It remains to show that  the summand 
$\Psi (\{\mathcal{A}_1, \mathcal{A}_1^{\prime}\} \otimes \Psi (\{\mathcal{B}_1, \mathcal{B}_1^{\prime}\})$
lands in  
$$
\{\mathcal{A}_1, \mathcal{B}_1\} \otimes \{\mathcal{A}_1^{\prime}, \mathcal{B}_1^{\prime}\} .
$$

From the definition of the brackets $\{\mathcal{A}_1, \mathcal{B}_1\}$ given in \cref{subsec:SecProd},
it follows that the graded $\Sy_3$--module  $\{\mathcal{A}_1, \mathcal{A}_1^{\prime}\}$ is linearly spanned by the expressions
$\tau_i(a,a^{\prime})$ where $i =1,2,3$, and $a\in \mathcal{A}_1$, $a^{\prime}\in \mathcal{A}_1^{\prime}$
are either simultaneously $\Sy_2$--even, or simultaneously $\Sy_2$--odd. This comes from the following facts 
$$\tau_1(a,a^{\prime})^{(12)}=\pm \tau_1(a,a^{\prime})\ , \quad 
\tau_2(a,a^{\prime})^{(12)}=\pm \tau_3(a,a^{\prime})\ ,$$
$$ 
\tau_1(a,a^{\prime})^{(123)}=\tau_2(a,a^{\prime})\ , \quad 
\tau_2(a,a^{\prime})^{(123)}=\tau_3(a,a^{\prime})\ .
$$
Moreover,
\begin{align*}
&\Psi\big((\tau_i(a,a^{\prime})+ \tau_i(a^{\prime},a))\otimes (\tau_i(b,b^{\prime})+ \tau_i(b^{\prime},b))\big) \\
&=
(-1)^{|a'||b|}\tau_i(a\otimes b,a^{\prime}\otimes b^{\prime})+ 
(-1)^{|a||b'|} \tau_i(a^{\prime}\otimes b^{\prime}, a\otimes b) \in
\{\mathcal{A}_1 \otimes \mathcal{B}_1,  \mathcal{A}_1^{\prime} \otimes \mathcal{B}_1^{\prime} \}\ ,
\end{align*}
whereas for $i\neq j$,
$$
\Psi\big((\tau_i(a,a^{\prime})+ \tau_i(a^{\prime},a))\otimes (\tau_j(a,a^{\prime})+ \tau_j(a^{\prime},a))\big) = 0.
$$
This concludes the proof. 
\end{proof}

\subsection{White product in $\BQO$, yet another product, and the interchange law $\psi$} 
Similarly
to what happens in the category of quadratic data, we can introduce the following
white product in $\BQO$:
$$
\mathcal{A}\circ\mathcal{B} \coloneq (\mathcal{A}_1\otimes \mathcal{B}_1,
\Phi^{-1} (R(\mathcal{A})\otimes \mathcal{T}(\mathcal{B}_1)(3) +  \mathcal{T}(\mathcal{A}_1)(3)\otimes R(\mathcal{B})))\ ,
$$
where $\Phi$ is the natural map
$$
\Phi = \Phi_{\mathcal{A}_1\mathcal{B}_1} \coloneq \mathcal{T}(\mathcal{A}_1\otimes \mathcal{B}_1)(3) \to
\mathcal{T}(\mathcal{A}_1)(3) \otimes \mathcal{T}( \mathcal{B}_1)(3) \ ,
$$
which duplicates the underlying tree. 
Black and white products are also related to each other by the operadic  duality functor $*$\ . 

\smallskip

Similarly, the product $\underline{\circledcirc}$ defined in \cref{subsec:SecProd} is sent to the following product ${\circledcirc}$ under the operadic Koszul duality functor $*$.
We first consider the  sub-$\Sy_3$--module
$$
[\mathcal{A}_1,\mathcal{B}_1] \subset \mathcal{T}(\mathcal{A}_1\oplus \mathcal{B}_1)(3) ,
$$
spanned  by the elements
$\tau_1(a,b) -\tau_1(b,a)$ whenever $a,b$ are simultaneously
$\Sy_2$--even or odd, and in addition by the expressions
$\tau_1(a,b)$ when one of the arguments $a,b$ is even and another
is odd. Then, we define 
$$
\mathcal{A} \,{\circledcirc}\, \mathcal{B} \coloneq (\mathcal{A}_1 \oplus \mathcal{B}_1, R(\mathcal{A})\oplus [\mathcal{A}_1, \mathcal{B}_1]
\oplus R(\mathcal{B}))\ .
$$

These two monoidal products are related by the interchange law $\psi$ induced by \eqref{eq:3.2}.

\begin{proposition} 
The interchange law $\psi$ in the category
of graded $\Sy_2$--modules lifts to morphisms in $\emph{\BQO}$ 
$$
\psi_{\mathcal{A}\mathcal{A}^{\prime}\mathcal{B}\mathcal{B}^{\prime}}:\quad
(\mathcal{A}\circ \mathcal{A}^{\prime}) \circledcirc (\mathcal{B}\circ \mathcal{B}^{\prime}) \to
(\mathcal{A}\circledcirc \mathcal{B}) \circ (\mathcal{A}^{\prime} \circledcirc \mathcal{B}^{\prime})
$$
which make $(\emph{\BQO}, \circledcirc, \circ, \psi )$ a lax 2-monoidal
category and $(\emph{\BQO}, \circledcirc,\circ,  \psi )$ a colax
2--monoidal category.
\end{proposition}

\begin{proof} As in the  quadratic data case, this proposition is Koszul dual to \cref{prop:2MonoOP} under finite dimensional assumptions. However, it holds in the general case by direct inspection. 
\end{proof}

\subsection{Applications}

\begin{corollary} \leavevmode
\begin{enumerate}
\item  Let $M,N$ be two monoids  in $\emph{\BQO}$
with respect to the black product $\bullet$ (resp. the $\circledcirc$  product). Then $M\underline{\circledcirc} N$ 
(resp. $M{\circ} N$) 
also has a natural
structure of a $\bullet$-monoid (resp. a $\circledcirc$-monoid).

\item Similarly, let $M,N$ be two comonoids  in $\emph{\BQO}$
with respect to the $\underline{\circledcirc}$  product 
(resp. the $\circ$  product). Then $M \bullet N$ 
(resp. $M{\circledcirc} N$)
also has a natural
structure of $\circledcirc$--comonoid (respectively $\circ$-comonoid).
\end{enumerate}
\end{corollary}

\begin{proof}
Again the proof relies on the fact that lax monoidal functors preserve monoids and that colax monoidal functors preserve comonoids.
\end{proof}

\begin{corollary}\label{cor:OPwhitemono}\leavevmode
\begin{enumerate}
\item Let $\P, \mathcal{Q}$ be two operads in the symmetric monoidal category $(\emph{\QD}, \bullet)$. Then 
their arity-wise $\uot$-product $(\P\,\uot\,\mathcal{Q})(n)\coloneq\P(n)\,\uot\, \mathcal{Q}(n)$ is again an operad in $(\emph{\QD}, \bullet)$.

\item Let $\P, \mathcal{Q}$ be two operads in the symmetric monoidal category $(\emph{\QD}, \ot)$. Then 
their arity-wise white product $(\P\circ\mathcal{Q})(n)\coloneq\P(n)\circ \mathcal{Q}(n)$ is again an operad in $(\emph{\QD}, \ot)$.
\end{enumerate}
\end{corollary}

\begin{proof}
The statement of \cref{prop:2monocatQD} actually says 
the functor $\uot$ 
is a lax monoidal functor from from $(\QD, \bullet)^2$ to $(\QD, \bullet)$. It is straightforward to see that it is also symmetric. The first statement thus follows from  \cref{Prop:SymMonoFunOp}. The second statement is proved in the way with the lax symmetric monoidal functor 
$\circ$ from $(\QD, \ot)^2$ to $(\QD, \ot)$ of \cref{cor:White2Lax}.

\end{proof}

This latter construction can be applied to the various examples of operads that we will give in \cref{Sec:HopfOp}. 

\begin{remark}
This result shows that one can refine the theory of 2-monoidal categories developed in \cite{Vallette08}: one can define a notion of a \emph{symmetric} $2$-monoidal category by requiring that the structural interchange law be a \emph{symmetric} monoidal functor. The present examples given in this paper will actually fall into this case; they provide us with symmetric 2-monoidal categories. We leave the details to  the interested reader.
\end{remark}

\subsection{Some more monoidal structures and Koszul dualities} 
As in the
case of quadratic data, mentioned in the last lines of \cref{sec:2MonoQ}, 
one can introduce several more pairs of monoidal structures in the context
of binary operadic quadratic data. Here is a list of possibilities in $\BQO$, including
the ones we have already considered.

\smallskip

We denote by $\mathcal{A}_1\circ_1 \mathcal{B}_1$ the sub-$\Sy_3$-module of  $\mathcal{T}(\mathcal{A}_1\oplus \mathcal{B}_1)(3)$
spanned by $\tau_1(b,a)$. We put
\begin{eqnarray*}
&\mathcal{A} \vee \mathcal{B} \coloneq (\mathcal{A}_1\oplus \mathcal{B}_1, R(\mathcal{A})\oplus R(\mathcal{B}))\ ,
&\\
&\mathcal{A} \oplus \mathcal{B} \coloneq (\mathcal{A}_1\otimes \mathcal{B}_1, R(\mathcal{A})\oplus 
\mathcal{A}_1\circ_1 \mathcal{B}_1 \oplus \mathcal{B}_1\circ_1 \mathcal{A}_1 \oplus       R(\mathcal{B}))\ ,
&\\ 
&\mathcal{A} \triangleleft \mathcal{B} \coloneq (\mathcal{A}_1\oplus \mathcal{B}_1, R(\mathcal{A})\oplus 
\mathcal{B}_1\circ_1 \mathcal{A}_1 \oplus        R(\mathcal{B}))\ ,
&\\
&\mathcal{A} \triangleright \mathcal{B} \coloneq (\mathcal{A}_1\otimes \mathcal{B}_1, R(\mathcal{A})\oplus 
\mathcal{A}_1\circ_1 \mathcal{B}_1 \oplus        R(\mathcal{B}))\ ,
&\\&
\mathcal{A} \,\underline{\circledcirc}\, \mathcal{B} \coloneq (\mathcal{A}_1\oplus  \mathcal{B}_1, R(\mathcal{A})\oplus 
\{\mathcal{A}_1, \mathcal{B}_1\} \oplus        R(\mathcal{B}))\ ,
&\\&
\mathcal{A}\circledcirc \mathcal{B} \coloneq (\mathcal{A}_1\oplus  \mathcal{B}_1, R(\mathcal{A})\oplus 
[\mathcal{A}_1, \mathcal{B}_1] \oplus        R(\mathcal{B}))\ .
\end{eqnarray*}

\begin{proposition}\leavevmode
\begin{enumerate}
\item  Let $\mathcal{P}_{\mathcal{A}}$ denote the operad
corresponding to the binary operadic quadratic data $\mathcal{A}$. Then $\mathcal{P}_{\mathcal{A}\vee \mathcal{B}}$
is the coproduct of $\mathcal{P}_{\mathcal{A}}$ and $\mathcal{P}_{\mathcal{B}}$ in the category of operads,
and furthermore
$$
\mathcal{P}_{\mathcal{A}\oplus \mathcal{B}}\cong \mathcal{P}_{\mathcal{A}}\oplus \mathcal{P}_{\mathcal{B}}\ , \quad
\mathcal{P}_{\mathcal{A}\triangleleft  \mathcal{B}}\cong \mathcal{P}_{\mathcal{A}}\otimes \mathcal{P}_{\mathcal{B}}\ , \quad
\mathcal{P}_{\mathcal{A}\triangleright  \mathcal{B}}\cong \mathcal{P}_{\mathcal{B}}\otimes \mathcal{P}_{\mathcal{A}}\ .
$$

\item These six monoidal structures are connected by the following Koszul duality involutions:
$$
(\mathcal{A} \vee \mathcal{B})^* \cong \mathcal{A}^*\oplus \mathcal{B}^* ,\quad
(\mathcal{A} \triangleleft \mathcal{B})^* \cong \mathcal{A}^*\triangleright \mathcal{B}^* ,\quad
(\mathcal{A} \,\underline{\circledcirc}\, \mathcal{B})^* \cong \mathcal{A}^*\circledcirc \mathcal{B}^* .
$$
\end{enumerate}
\end{proposition}

For more details, see \cite[Section~8.6]{LodayVallette12}.

\begin{proposition}  
 The interchange laws $\varphi$ and $\psi$ in the
category of graded $\Sy_2$--modules induce morphisms in the category
$\emph{\BQO}$ which make the quintuples
$$
(\emph{\BQO}, \bullet ,\vee ,\varphi ,\psi)\ ,\quad
(\emph{\BQO}, \bullet ,\oplus ,\varphi ,\psi)\ ,\quad
(\emph{\BQO}, \bullet ,\triangleleft ,\varphi ,\psi)\ ,\quad
(\emph{\BQO}, \bullet ,\triangleright ,\varphi ,\psi)
$$
and
$$
(\emph{\BQO}, \circ ,\vee ,\varphi ,\psi)\ ,\quad
(\emph{\BQO}, \circ ,\oplus ,\varphi ,\psi)\ ,\quad
(\emph{\BQO}, \circ ,\triangleleft ,\varphi ,\psi)\ ,\quad
(\emph{\BQO}, \circ ,\triangleright ,\varphi ,\psi) 
$$
into 2--monoidal categories, i.e. simultaneously lax and colax.
\end{proposition}

\begin{proof} The proof can be obtained by direct computations.
\end{proof}

\section{Lie operads and Hopf (co)operads}\label{Sec:HopfOp}

The purpose of this section is to provide a simple categorical setting for the \emph{automatic} construction of several (co)operads in categories of (co)algebras starting from just a single and simple operad structure. This framework applies to many operads which play a key role in the literature. In quantum groups, deformation quantization, algebraic topology and Grothendieck--Teichm\"uller groups, like in \cite{Drinfeld90, KontsevichManin94, Tamarkin03, SeveralWillwacher11, LambrechtsVolic14, Fresse17}, it is crucial to work with Lie operads or Hopf (co)operads, that is operads in the category of Lie algebras and (co)operads in the category of (co)algebras. These kind of  (co)operad structures are produced here from topological operads; this way, we recover the ones present in the above-mentioned theories, as well as  interesting new ones. 

\smallskip

When dealing with symmetric monoidal categories which are obviously strong,  we will drop this  adjective for simplicity.

\subsection{Operads, cooperads, and symmetric monoidal functors}
Since the  opposite category  of a  symmetric monoidal category  is again  symmetric monoidal, we can consider the following notion dual to that of an operad. 

\begin{definition}[Cooperad]
A \emph{cooperad} $\CC$ in $\C$ is an operad  in the opposite symmetric monoidal category $\C^{\text{op}}$.
\end{definition}

This means that we are given a functor $\CC : \Fin^{\text{op}} \to \C^{\text{op}}$ (or equivalently $\CC : \Fin \to \C$) with partial decompositions maps in $\C$: 
\[\delta_{X\subset Y}\ : \ \CC(Y) \to \CC(Y/X)\ot \CC(X)\ , \quad \text{for any} \ X\subset Y\ ,\]
and a counit  
$\varepsilon \ : \  \P(\{*\}) \to 1_\C$ satisfying the dual commutative diagrams. 

\begin{proposition}\label{Prop:SymMonoFunOp}\leavevmode
\begin{enumerate}
\item Any covariant  symmetric monoidal functor sends operads to operads and cooperads to cooperads.
\item Any contravariant  symmetric monoidal functor sends operads to cooperads and cooperads to operads.
\end{enumerate}
\end{proposition}

\begin{proof}
It is well-known that any covariant lax symmetric monoidal functor sends operads to operads. Thus any contravariant oplax symmetric monoidal functor, i.e. such that the associated covariant functor between the opposite categories, sends cooperads to cooperads. Let us just sketch the proof a little bit since we will use the transferred (co)operad structure later on.

Let $(\C, \ot_\C, 1_\C, \alpha_\C,  \lambda_\C,\allowbreak \rho_\C, \allowbreak \tau_\C)$, 
$(\D, \ot_\D, 1_\D, \alpha_\D,  \lambda_\D, \rho_\D, \tau_\D)$ be two symmetric monoidal categories and let 
$\F : \C \to \D$ be a covariant symmetric monoidal functor with structure maps
\[\psi : 1_\textbf{D} \to \F(1_\C) \qquad  \text{and} \qquad \varphi_{A,B} : \F(A)\ot_\textbf{D} \F(B) 
\to \F(A\ot_\C B)\ . \]
 For any operad $\P$, we consider the following structure maps of $\F \P$:
\begin{align*}
& \F\P(Y/X)\ot_\D \F\P(X) \xrightarrow{\varphi_{\P(Y/X), \P(X)}} \F(\P(Y/X)\ot_\C \P(X))\xrightarrow{\F(\circ_{X\subset Y})} \F \P(Y) 
\quad \text{and}  \\
& 1_\D  \xrightarrow{\psi} \F(1_\C) \xrightarrow{\F(\eta)} \F\P(\{*\}) \ .
\end{align*}
Dually, for any cooperad $\CC$, we consider the following structure maps of $\F \C$:
\begin{align*}
& 
\F \CC(Y) \xrightarrow{\F(\delta_{X\subset Y})}
 \F(\CC(Y/X)\ot_\C \CC(X)) \xrightarrow{\varphi^{\textrm{op}}_{\CC(Y/X), \CC(X)}} 
\F\CC(Y/X)\ot_\D \F\CC(X) 
\quad \text{and}  \\
& 
 \F\CC(\{*\}) \xrightarrow{\F(\varepsilon)}  
 \F(1_\C) \xrightarrow{\psi^{\textrm{op}}}
1_\D  
\ .
\end{align*}
It remains to check the various axioms of this new structure but this follows in a straightforward way from the defining axioms of the operad $\P$ (or the cooperad $\CC$), of the two monoidal categories $\C$ and $\D$, and the symmetric monoidal functor $\F$. 

 The second assertion is less present in the literature. It is however a formal consequence of the first assertion. Let 
 $\P : \Fin^{\text{op}} \to \C$ be an operad in $\C$. By definition, 
this means that $\P$ is a cooperad in the opposite category $\C^{\text{op}}$.
  It is thus sent to a cooperad in $\D$ 
 under the (covariant) symmetric monoidal functor $\F^{\text{op}} : \C^{\text{op}}\to \D$. 
\end{proof}

We have already been applying this result in \cref{cor:OPwhitemono}. Now \cref{Thm:ComDiag} and \cref{Prop:SymMonoFunOp} allow us to deduce seven operad structures and four cooperad structures out of the sole data of an operad structure in the category of skew-symmetric quadratic data.  Since the monoidal product $\oplus$ of this latter category is particularly simple, the data of an operad there is also not difficult to establish, as the following examples show. 

\begin{remark}
Notice that the left-to-right symmetric monoidal functors can all be inverted. So we could also induce transport (co)operad structures in the other way round. Moreover, one can often easily guess from a (co)operad structure in a category of (co)algebras the associated (co)operad structure in the above category of quadratic data. In the end, the global orientation of the diagram chosen here is not restrictive, but amounts rather to a choice of presentation. 
\end{remark}

\begin{definition}[Lie operad and (co)commutative Hopf (co)operad]
An operad in the symmetric monoidal category $(\Lie, \oplus)$ of Lie algebras is called 
a \emph{Lie operad}. 
An operad in the symmetric monoidal category 
$(\Comco, \otimes)$ of cocommutative coalgebras is called a \emph{cocommutative Hopf operad}. 
A cooperad in the symmetric monoidal category 
$(\Com, \otimes)$ of commutative algebras is called a \emph{commutative Hopf cooperad}. 
\end{definition}

\begin{remark}
The notion of a Lie operad should not be confused with the operad $\mathrm{Lie}$ encoding Lie algebras. 
\end{remark}

\emph{From now on, we work over the field $\QQ$ of rational numbers}.

\begin{exam}
The homology group functor $H_\bullet(-)\coloneqq H_\bullet(-, \QQ)$ is a covariant symmetric monoidal functor and 
the cohomology group functor $H^\bullet(-)\coloneqq H^\bullet(-, \QQ)$ is a contravariant symmetric monoidal functor. The former sends a topological operad to a cocommutative Hopf operad and the latter sends it to a commutative Hopf cooperad. 
\end{exam}

\subsection{Lie operads from pointed topological operads}\label{subsec:LiePointed}
Following the same pattern, we aim at producing functorially Lie operads from topological operads using  rational  fundamental groups. 
Suppose now that every component $\O(n)$ of the topological operad $\O$ admits a base point $*^n$ which is compatible with the operadic structure, i.e. $*^n \circ_k *^m=*^{n+m-1}$. In other words, this means that we consider an operad in the symmetric monoidal category of pointed topological spaces $(\textbf{Top}_*, \times)$.
In this case, one can consider the fundamental groups $\pi_1(\O(n))$ of each component and then their images under  the Magnus construction \cite{Magnus37, Lazard50}
\[\gr(G)\coloneqq\bigoplus_{k\geqslant 1} \Gamma_k G/\Gamma_{k+1} G\ , \]
which associates a Lie algebra over $\mathbb{Z}$ to any group $G$ by means of its lower central series, defined inductively by $\Gamma_1 G\coloneqq G$ and $\Gamma_{k+1}G\coloneqq [\Gamma_kG, G]$. Recall that  
the Lie bracket $[x, y]$ is induced by the group commutator $xyx^{-1}y^{-1}$. 

\begin{lemma}\label{lem:PiGrSymMono}\leavevmode
\begin{enumerate}
\item The fundamental group functor $\pi_1 : (\emph{\textbf{Top}}_*, \times) \to (\emph{\textbf{Gr}},\times)$ from the category of topological spaces to the category of groups is cartesian, i.e. strongly symmetric monoidal with respect to the products. 
\item The Magnus functor $\gr : (\emph{\textbf{Gr}}, \times) \to (\emph{\Lie}_{\mathbb{Z}},\oplus)$ from the category of groups to the category of Lie algebra over $\mathbb{Z}$ is cartesian.
\end{enumerate}
\end{lemma}

\begin{proof}
The proof is straightforward. 
\end{proof}

\begin{remark}
As usual, in order to get a nice behaviour of topological spaces with respect to products, one needs to restrict to the category of compactly generated Hausdorff spaces with Kelly product, which we implicitly do here. 
\end{remark}

\begin{proposition}
Any pointed topological operad $\O$ induces an operad in the category of Lie algebras over $\mathbb{Z}$: 
$$\gr\left(\pi_1(\O)\right)\ , $$ which is called the \emph{Magnus operad}. 
\end{proposition}

\begin{proof}
This is a direct corollary of \cref{Prop:SymMonoFunOp} and \cref{lem:PiGrSymMono}. 
\end{proof}

\subsection{Operadic quadratic data from topological operads}
Now we study how the three aforementioned functors producing respectively ``(co)homology'' Hopf (co)operads and ``homotopy'' Lie operads from topological operads lift to the quadratic data level. 

\begin{remark}
In this paper, we need to put homology and cohomology on the same footing in order to treat them with the framework described in \cref{sec:DefQD}. Since we use the homological degree convention and since cohomology will always appear as linear dual of homology, the cohomology groups will be non-positively graded. 
In other words, we use the opposite of the usual convention. 
\end{remark}

\medskip

Let $\O$ be a topological operad. 
The restriction $\cup  : H^1(\O(n))^{\odot 2} \subset H^1(\O(n))^{\otimes 2}\to H^2(\O(n))$ of the cup-product gives rise to the 
symmetric quadratic data 
\[\big(H^1(\O(n)), \ker \cup\big)\in \QD^+\ ,\]
where $H^1(\O(n))$ is concentrated in "homological degree" $-1$. 
 When $H^1(\O(n))$ is finite dimensional, for any $n\geqslant 0$, we consider the (degree-wise) linear dual symmetric quadratic data 
 \[\big(H^1(\O(n)), \ker \cup\big)^*\cong (H_1(\O(n)),  \mathrm{im}\, \Delta)\in \QD^+\ ,\]
  where $\Delta\coloneqq{}^t\cup :  H_2(\O(n)) \to H_1(\O(n))^{\odot 2}$ is the restriction of the coproduct of the homology coalgebra.
Finally, the Koszul duality functor gives rise to the following skew-symmetric quadratic data 
\[\big(H_1(\O(n)),  \mathrm{im}\, \Delta\big)^{\acc}\cong 
\big(s^{-1}H_1(\O(n)), s^{-2} \mathrm{im}\, \Delta\big)\in \QD^-\ . \] 

\begin{definition}[Holonomy Lie algebra, after Chen--Kohno \cite{Chen73, Kohno85}]
The holonomy Lie algebras of the topological spaces $\O(n)$ are the quadratic  Lie algebras induced by the above presentations: 
\[
\g_{\O(n)}\coloneqq \L(s^{-1}H_1(\O(n)), s^{-2} \mathrm{im}\, \Delta)\ .
\]
\end{definition}

When each component $\O(n)$ is path connected, for $n\geqslant 0$, the (co)algebras 
$H^\bullet(\O(n))$ (respectively $H_\bullet(\O(n))$) are (co)augmented. In this case, 
the two (co)operad structures $H^\bullet(\O)$ and $H_\bullet(\O)$  induce respectively 
a cooperad structure on the collection of symmetric quadratic data $\left\{\big(H^1(\O(n)), \ker \cup\big)\right\}$ in the symmetric monoidal category $(\QD^+, \vee)$ and 
an operad structure on the collection of quadratic data $\left\{(H_1(\O(n)), \mathrm{im}\,  \Delta\big)\right\}$ in the symmetric monoidal category 
$(\QD^+, \uot)$. 
Since the Koszul duality functor $\acc$ (in the opposite direction) is symmetric monoidal, it induces 
an operad structure on the collection of skew-symmetric quadratic data $\left\{(s^{-1}H_1(\O(n)), s^{-2}\mathrm{im}\,  \Delta\big)\right\}$ in the symmetric monoidal category 
$(\QD^-, \oplus)$. In the end, we get a canonical Lie operad structure $\g_\O$ on the level of the holonomy Lie algebras. 

\begin{definition}[Holonomy operad]
The \emph{holonomy operad} is the operad $\g_\O$ made up of the holonomy Lie algebras associated to a path connected topological operad $\O$.
\end{definition}

\begin{proposition}\label{prop:(Co)Ho(Co)Op}
Let $\O$ be a topological operad satisfying the following condition. 
\begin{cond}\label{Cond:ConditionI}
For any $n\geqslant 0$, the cohomology algebras $H^\bullet(\O(n))$ admits a finitely generated homogenous quadratic presentation with generators in $H^1(\O(n))$\ .
\end{cond}
In this case, the canonical map $H^1(\O(n))\to H^\bullet(\O(n))$ induce the following isomorphism of commutative Hopf cooperads
\[
H^\bullet(\O)\cong \S\big(H^1(\O), \ker \cup\big)
\]
and the following isomorphism of cocommutative Hopf operads
\[
H_\bullet(\O)\cong \S^c\big(H_1(\O), \mathrm{im}\,  \Delta\big)\ .
\]
\end{proposition}

\begin{proof}
The proof is straightforward. \cref{Cond:ConditionI} ensures first ensures that the underlying components of the topological operad are path connected and then provides us with the underlying isomorphisms. By definition, the (co)operad structure coincides on the level of the symmetric quadratic data. The universal property of the (co)free (co)commutative (co)algebra concludes the proof. 
\end{proof}

\begin{remark}
The above treatment holds true in the same way when the first (co)homology groups $H^1$ and $H_1$ are replaced by the first non-trivial (co)homology groups $H^i$ and $H_i$, for $i\geqslant 1$. 
\end{remark}

So under \cref{Cond:ConditionI}, the six above mentioned (co)operads contain the exact same amount of data; in other words, there is no loss of generality by considering the operadic structures on the level of quadratic data. 

\begin{remark}
Notice that any cocommutative Hopf operad  induced by an operad  in the category $(\QD^+, \uot)$ under the quadratic cocommutative coalgebra functor $\Sc$ contains canonically the operad $\mathrm{uCom}$ encoding unital commutative algebras: this latter one is simply made up of the counits of each coalgebras. 
\end{remark}

Let us recall the following seminal result due to D. Sullivan. 

\begin{theorem}[\cite{Sullivan77}, see also \cite{Kohno85}]\label{thm:FormalHolonomy}
Let $X$ be a pointed, path connected, and 1-finite topological space. When $X$ is (rationally) 1-formal, its holonomy Lie algebra is isomorphic to its rational Magnus Lie algebra
\[\g_{X}\cong  \gr\left(\pi_1(X)\right)\otimes \QQ \ .\]
\end{theorem}

The proof of this statement falls into two parts. First, one shows that the (cohomological) degree $1$ generators of the minimal model of the piece-wise linear forms ${A}^\bullet_{\rm PL}(X)$ give the rational Magnus Lie algebra. Then, under the formality assumption, one just needs to coin the minimal model of the cohomology algebra $H^\bullet(X)$. The linear dual of the space of degree $1$ generators is easily seen be the holonomy algebra, for instance by using the cobar-bar resolution and the homotopy transfer theorem. 

\medskip

In order to promote the above mentioned result to the operadic level (isomorphism between the  holonomy operad 
$\g_{\O}$ and the rational Magnus operad $\gr\left(\pi_1(\O)\right)\otimes \QQ$), one would need a \emph{rational Hopf (1-)formality property} satisfied by the topological operad $\O$ itself in order to control the operadic compatibility between the formality quasi-isomorphisms of dg commutative algebras 
\[
{A}^\bullet_{\rm PL}(\O(n)) \stackrel{\sim}{\longleftarrow} \bullet \cdots \bullet 
\stackrel{\sim}{\longrightarrow} H^\bullet(\O(n))\ .
\]
This general question  will be treated in the sequel of this paper, which will deal with the Hopf formality of topological operads. 

\medskip

We are now ready to give examples. 

\subsection{Berger--Kontsevich--Willwacher, i.e. graph operads}\label{subsec:BKW}
For  $n\geqslant 2$, we consider the complete graph $\Gamma_n$ on $n$ vertices labeled by $\{1, \ldots ,n\}$, that is with one and only one edge between every pair of distinct vertices. The edge between the vertices $i$ and $j$ is simply denoted by $ij=ji$. 
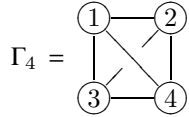
\[
\Gamma_4\,=\,\vcenter{\xymatrix@M=4pt@R=16pt@C=16pt{
		*+[o][F-]{1}\ar@{-}[r]\ar@{-}[d]\ar@{-}[dr] &*+[o][F-]{2} \ar@{-}[d]\ar@{-}[dl]|\hole\\
		*+[o][F-]{3}\ar@{-}[r] &*+[o][F-]{4}}}
\]

Let us now introduce a topological version $\GraS$ of the complete graph operad due to  C. Berger \cite{Berger96}, defined by the following (pointed) topological spaces 
\[\GraS(n)\coloneqq \{*\}\, ,\  \text{for} \  n=0\ \text{and}\ n=1,  \quad  \text{and}\quad \GraS(n)\coloneqq \big(S^1\big)^{\binom{n}{2}}\, ,\  \text{for} \  n\geqslant 2\ .\]
The elements $\{\mu_{ij}\}$ of $\GraS(n)$ can be thought of as elements of the circle $S^1$ labelling the edges $ij$ of the complete graph $\Gamma_n$. The  partial  composition products $\circ_p : \GraS(n)\times \GraS(m) \to 
\GraS(n+m-1)$ of two collections $\{\mu_{ij}\}$ and $\{\mu'_{i'j'}\}$ are defined as follows. The idea is to insert the complete graph $\Gamma_m$ at the $p$th vertex of the complete graph $\Gamma_n$ and to relabel the vertices accordingly: the labels of the vertices $1, \ldots, p-1$ of $\Gamma_n$ are stable, the labels of the vertices of $\Gamma_m$ are shifted by $p-1$, and the labels of the remaining vertices $p+1, \ldots, n$ of $\Gamma_n$ are shifted by $m-1$. 
If we denote the upshot of the partial composition product by $\{\nu_{kl}\}$, then $\nu_{kl}$ is equal to the corresponding element $\mu_{ij}$ when $k,l\in \{1, \ldots, p-1, p+m, \ldots, n+m-1\}$. It is equal to the corresponding element $\mu'_{i'j'}$ when $k,l\in \{p, \ldots, p+m-1\}$. When $k\in \{1, \ldots, p-1\}$ and $l\in \{p, \ldots, p+m-1\}$, we set $\nu_{kl}\coloneqq \mu_{k, p}$ and when
$k\in \{p, \ldots, p+m-1\}$ and $l\in \{p+m, \ldots, n+m-1\}$, we set $\nu_{kl}\coloneqq \mu_{p, l-m+1}$. 
With such a definition, the composite with $\GraS(1)$ is indeed the identity. 
The natural action of the symmetric group $\Sy_n$ on the vertices of the graph $\Gamma_n$ induces a right $\Sy$-module on $\GraS(n)$. The composite on the right-hand side with $\GraS(0)$ amounts to forgetting some data which, with the symmetric group action, produces a FI-module structure \cite{CEF15}. 

\begin{proposition}
The topological operad $\GraS$ is formal over $\mathbb{Z}$: there exists a  quasi-isomorphism of dg operads over $\mathbb{Z}$
\[C^{\mathrm{sing}}_\bullet(\GraS, \mathbb{Z})  \ \stackrel{\sim}{\longleftarrow}\  H^{\mathrm{sing}}_\bullet(\GraS, \mathbb{Z})\ .\]
\end{proposition}

\begin{proof}
The proof is straightforward and can be performed by the same arguments as in \cite[Section~8]{DotsenkoShadrinVallette15}.
\end{proof}

\begin{remark}
Any topological space $X$ can replace $S^1$ in order to form a similar topological operad. 
For instance, any topological space homotopy equivalent to the circle, like $\mathbb{C}\backslash\{0\}$ for instance,  would produce a homotopy equivalent operad. In this case, the formality property can again be proved easily by hand; it can also be shown directly using \cite{GSNPR05, CiriciHorel17}, which rely on   mixed Hodge structures.
\end{remark}

\begin{definition}[Berger--Kontsevich--Willwacher skew--symmetric quadratic data]
The \emph{Berger--Kon\-tse\-vich--Willwacher skew--symmetric quadratic data} are spanned by 
\[
\mathrm{BKW}(n)\coloneq\left(
\t^n_{ij}\, ,\, \t^n_{ij}\wedge \t^n_{kl}
\right)\ ,
\]
where 
the set of generators $\t^n_{ij}$ of degree $0$ runs over the set of edges $ij$ of $\Gamma_n$ and where the  set of relations runs over all pairs $(ij, kl)$ of  edges in $\Gamma_n$. 
For $n=0$ and for $n=1$, we set $\mathrm{BKW}(0)\coloneq(0, 0)$ and $\mathrm{BKW}(1)\coloneq(0, 0)$. 
\end{definition}

%Let us denote by $V(n)$ the space of generators and by $R(n)=V(n)^{\wedge 2}$ the space of relations, which is maximal here. 
We consider the following maps $\circ_k : \mathrm{BKW}(n)\oplus \mathrm{BKW}(m)\to \mathrm{BKW}(n+m-1)$:
\begin{equation}\label{eq:CompoOp}
\begin{array}{lcll}
& \t^n_{ij} & \mapsto & \left\{\begin{array}{lll}  
\t^{n+m-1}_{i+m-1\, j+m-1} & \text{for} & k<i,j\ ,\\
\rule{0pt}{12pt} \t^{n+m-1}_{i\, j+m-1}+\t^{n+m-1}_{i+1\, j+m-1}+\cdots+ \t^{n+m-1}_{i+m-2\, j+m-1}+\t^{n+m-1}_{i+m-1\, j+m-1} & \text{for} & k=i\ , \\
\rule{0pt}{12pt} \t^{n+m-1}_{i\, j+m-1} & \text{for} & i<k<j\ , \\
\rule{0pt}{12pt} \t^{n+m-1}_{i\, j}+\t^{n+m-1}_{i\, j+1}+\cdots+ \t^{n+m-1}_{i\, j+m-2}+\t^{n+m-1}_{i\, j+m-1} & \text{for} & k=j\ , \\
\rule{0pt}{12pt} \t^{n+m-1}_{ij} & \text{for} & i,j<k \ ,
\end{array}\right.\\
\rule{0pt}{12pt} & \t^m_{ij} & \mapsto & \quad  \t^{n+m-1}_{i+k-1\, j+k-1}\ .\\
\end{array}
\end{equation}

\begin{lemma}\label{lem:KWDataOp}
The above-mentioned data 
$\mathrm{BKW}\coloneqq\big(\{\mathrm{BKW(n)}\}, \{\circ_k\}\big)$
 forms an operad in the symmetric monoidal category \rm$(\QD^-, \oplus)$.
\end{lemma}

\begin{proof}
Since the spaces of relations are the full spaces $R(n)=V(n)^{\wedge 2}$, for any $n\geqslant 0$, the maps $\circ_k$ are morphisms of quadratic data. It is straightforward to check the sequential and parallel axioms, the equivariance with respect to the symmetric groups action, as well as the axioms for the unit. 
\end{proof}

The Lie operad $\L(\mathrm{BKW})$ is thus made up of the graphs with one edge (of degree 0) and with similar  partial composition maps. 

\begin{proposition}\label{prop:HoloGraS}
The holonomy operad and the rational Magnus operad associated to $\GraS$ are isomorphic to the Lie operad associated to skew-symmetric quadratic data $\mathrm{BKW}$: 
\[\g_{\GraS}\cong \L(\mathrm{BKW}) \cong \gr\left(\pi_1\big(\GraS\big)\right)\otimes \QQ \ .\]
\end{proposition}

\begin{proof}
The first isomorphism is obtained directly from the definition of the holonomy operad. 
The cohomology algebra of the circle is the algebra of dual number $H^\bullet(S^1)\cong \QQ {1} \,\oplus\, \QQ \varepsilon$, that is the free commutative algebra on one degree one element $\varepsilon$. This shows that the cohomology symmetric quadratic data is trivial 
\[\big(H^1(\GraS(n)), \ker \cup\big)=\big(
\varepsilon^n_{ij}, \emptyset
\big)\ ,\]
and thus that the holonomy skew-symmetric data is the Berger--Kontsevich--Willwacher one 
\[\big(s^{-1}H_1(\GraS(n)), s^{-2} \mathrm{im}\, \Delta\big)=\big(\t^n_{ij}\, ,\, \t^n_{ij}\wedge \t^n_{kl}
\big)\ ,\]
under the identification $\t^n_{ij}\cong s^{-1}\left(\varepsilon^n_{ij}\right)^*$.
In order to show that these isomorphisms commute with the respective operadic structures, one needs to describe the homology operad $H_\bullet\big(\GraS\big)$; this computation is performed in the core of the proof of \cref{prop:HoBKW} below. 

The second isomorphism is also straightforward from the definition of the rational Magnus operad. One has 
\[\gr\left(\pi_1\big(\GraS(n)\big)\right)\otimes \QQ\cong \QQ^{{\binom{n}{2}}}\]
and the partial composition maps agree. 
\end{proof}

Let us recall from \cite{Kontsevich93, Kontsevich97, Willwacher15} the definition of the \emph{graph operad} $\mathrm{Gra}$ of natural operations of polyvector fields of $\mathbb{R}^k$. Its underlying $\Sy$-modules are spanned by
subgraphs of $\Gamma_n$, that is 
 graphs with $n$ vertices labeled bijectively by $\{1, \ldots, n\}$ and possibly at most one edge of degree $1$  between any pair of vertices. The partial composition product $\gamma_1\circ_k\gamma_2$ amounts to first inserting the graph $\gamma_2$ at the $k$th vertex of $\gamma_1$, then relabelling accordingly the vertices, and finally consider the sum of all the possible ways to connect the edges in $\gamma_1$ originally plugged to the vertex $k$, to any possible vertex of $\gamma_2$. 

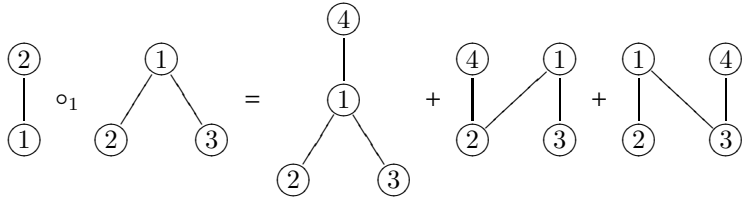
\[
\vcenter{\xymatrix@M=4pt@R=16pt@C=16pt{
		*+[o][F-]{2}\ar@{-}[d]\\
		*+[o][F-]{1}}}
\ \circ_1 \
\vcenter{\xymatrix@M=4pt@R=16pt@C=6pt{
		& *+[o][F-]{1}\ar@{-}[dl]\ar@{-}[dr]& \\
		*+[o][F-]{2}&&*+[o][F-]{3}}}
\ = \ 
\vcenter{\xymatrix@M=4pt@R=16pt@C=6pt{
		& *+[o][F-]{4}\ar@{-}[d]& \\
		& *+[o][F-]{1}\ar@{-}[dl]\ar@{-}[dr]& \\
		*+[o][F-]{2}&&*+[o][F-]{3}}}
\ + \ 
\vcenter{\xymatrix@M=4pt@R=16pt@C=6pt{
		*+[o][F-]{4}\ar@{-}[d]&& *+[o][F-]{1}\ar@{-}[dll]\ar@{-}[d] \\
		*+[o][F-]{2}&&*+[o][F-]{3}}}
\ + \ 		
\vcenter{\xymatrix@M=4pt@R=16pt@C=6pt{
		 *+[o][F-]{1}\ar@{-}[d]\ar@{-}[drr]&& *+[o][F-]{4}\ar@{-}[d]\\
		*+[o][F-]{2}&&*+[o][F-]{3}}}
\]
Every graded vector space $\mathrm{Gra}(n)$ forms a cocommutative coalgebra with the coproduct $\Delta(\gamma)$ made up of the pairs of graphs $\gamma'\otimes \gamma''$ with the same  $n$ vertices as $\gamma$ but  with edges from $\gamma$ distributed on $\gamma'$ and $\gamma''$. The partial composition products preserve these coproducts, thus  $\mathrm{Gra}$ forms a cocommutative Hopf operad. 

\begin{proposition}\label{prop:HoBKW}
The following three cocommutative Hopf operads are isomorphic 
\[H_\bullet\big(\GraS\big)\cong \S^c\big(\mathrm{BKW}^{\emph{\acc}}\big) \cong \mathrm{Gra}\ .\]
\end{proposition}

\begin{proof}
The topological operad $\GraS$ satisfies \cref{Cond:ConditionI} and thus the first isomorphism is produced by \cref{prop:(Co)Ho(Co)Op} using \cref{prop:HoloGraS}. 
Using the fact that the homology coalgebra of the circle is the coalgebra of dual number on one degree one generator, one can directly prove the isomorphism of cocommutative Hopf operad $H_\bullet(\GraS) \cong \mathrm{Gra}$. 
The direct isomorphism $\S^c\big(\mathrm{BKW}^{\emph{\acc}}\big) \cong \mathrm{Gra}$ can however be made explicit  as follows using the previous sections. Let us denote by $V(n)\coloneqq \QQ\big\{\t^n_{ij}\big\}$ the space of generators of the quadratic data  $\mathrm{BKW}(n)$. 
The underlying $\Sy$-module of the cocommutative Hopf operad $\S^c\big(\mathrm{BKW}^{\acc}\big) = \big(\{\S^c\big(\mathrm{BKW}^{\acc}(n)\big)\},\allowbreak \{\tilde{\circ}_k\}\big)$, obtained by applying the symmetric monoidal functors $\acc$ and then $\S^c$, is made up of cofree cocommutative coalgebras $S^c(sV(n))=S^c\big(s \t^n_{ij}\big)$, with $| s \t^n_{ij}|=1$, which admits for basis the monomials $s \t^n_{i_1j_1}\odot \cdots \odot s \t^n_{i_kj_k}$, where all the pairs $ij$ are  different. Such monomials are in one-to-one correspondence with the graphs of $\mathrm{Gra}(n)$. 

The partial composition products
$\tilde{\circ}_k : S^c(sV(n))\otimes S^c(sV(m))\to S^c(sV(n+m-1))$ 
 of the operad $\S^c\big(\mathrm{BKW}^{\acc}\big)$  are morphisms of cocommutative coalgebras; so they are characterised by their projections onto $sV(n+m-1)$. 
 
We denote by  $\emptyset^n$ the graph with $n$  vertices and with no edge, that  we identify  with the counit $\mathbb{1}^n$ of the cofree coalgebra $S^c(sV(n))$, i.e.  $\mathbb{1}^n \leftrightarrow \emptyset^n$.
We denote by $\gamma^n_{ij}$ the graph with $n$ vertices and with the only edge $ij$, that we identify with the generator $s\t^n_{ij}$ of  $S^c(sV(n))$, i.e.  $s\t^n_{ij} \leftrightarrow \gamma^n_{ij}$.
Under this correspondence, the isomorphism $S^c(sV(n))\otimes S^c(sV(m))\cong  S^c(s(V(n)\oplus V(m)))$ sends 
$\emptyset^n\otimes \emptyset^m$ to $\mathbb{1}^{n+m-1}$, 
$\gamma^n_{ij}\otimes  \emptyset^m$ to $s\t^n_{ij}$, and  
$\emptyset^n\otimes  \gamma^m_{ij}$ to $s\t^m_{ij}$.
The images of these three latter elements under the partial composition products $\circ_k$ of the operad $\mathrm{BKW}$ given in \cref{eq:CompoOp} coincide, under the above identifications, to the images of the three former elements under the partial composition products of the operad $\mathrm{Gra}$, which concludes the proof. 
\end{proof}

The operad $\mathrm{U}\L(\mathrm{BKW})=\A\Lambda(\mathrm{BKW})$ in associative algebras is similar to the operad $\mathrm{Gra}$ except that we consider graphs with possibly multiple edges (of degree 0) between vertices. The (algebra) product of two such graphs amounts to consider the union of their sets of edges.  

\begin{remark}
Using the recognition method of C. Berger \cite{Berger96}, one can see that $\GraS$ admits a 
(cellular) sub-operad which  is an $E_2$-operad, that is a topological operad having the same homotopy type then the little disks operad $\mathcal{D}_2(n)\sim \mathrm{Conf}_n(\mathbb{R}^2)$, see \cref{subsec:LDOp}. The operad $\mathrm{Gra}$ admits a map from the operad encoding shifted Lie algebras, so it can be twisted \`a la Willwacher to produce a differential graded operad $\mathrm{Tw} \mathrm{Gra}$, see \cite{Willwacher15} and \cite[Section~5]{DotsenkoShadrinVallette18} for more details. This latter operad plays a key role in the proof of the formality of the little disks operad in \cite{Kontsevich99, LambrechtsVolic14, FresseWillwacher15}. Since $\mathrm{Tw} \mathrm{Gra}$ forms a dg cocommutative Hopf operad, it is a good model for the rational homotopy type of the little disks operad; this point explains conceptually why the rational homotopy automorphim group of the little disks operad is isomorphic to the Grothendieck--Teichm\"uller group in \cite{Willwacher15, Fresse17}.
\end{remark}

One can perform the same arguments for the topological operad $\GraS\rtimes S^1$, which is obtained by adding a copy of $S^1$ at every input, see \cite{SalvatoreWahl03} for the semi-direct product of operads. This amounts to adding $n$ generators $\t^n_i$, for $1\leqslant i\leqslant n$, to the skew--symmetric quadratic and again considering the full space of relations. The same results hold true \emph{mutatis mutandis} by considering now graphs with possible tadpoles, that with possibly one  loop attached to each vertex.
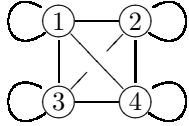
\[\vcenter{\xymatrix@M=4pt@R=16pt@C=16pt{
		 *+[o][F-]{1}\ar@{-}[r]\ar@{-}[d]\ar@{-}[dr] \ar@{-}@(dl,ul)[] &*+[o][F-]{2} \ar@{-}[d]\ar@{-}[dl]|\hole   \ar@{-}@(dr,ur)[] \\
		*+[o][F-]{3}\ar@{-}[r]  \ar@{-}@(dl,ul)[]&*+[o][F-]{4} \ar@{-}@(dr,ur)[]}}
\]

\subsection{Nonsymmetric analogue of the little disks operad, i.e. $\mathrm{As}_{S^1}$ and $\mathrm{As}_{S^1}\rtimes S^1$}\label{subsec:ncD}
A noncommutative  version of the notion of  Gerstenhaber algebras was introduced in \cite[Section~3]{DotsenkoShadrinVallette15} in relation with noncommutative deformation theory. This notion is modelled by the nonsymmetric (pointed) topological operad which is defined in a way similar to the aforementioned topological operad $\GraS$ but starting from  the complete linear graph $\Theta_n$  instead of the complete graph $\Gamma_n$. Explicitelty, $\Theta_n$ is the graph on $n$ vertices labeled by $\{1, \ldots ,n\}$ from left to right with one and only one edge between every consecutive pair of vertices $i\, i+1$. 
\[
\Theta_n\,=\,
\vcenter{\xymatrix@M=4pt@R=16pt@C=16pt{
		*+[o][F-]{1}\ar@{-}[r] &*+[o][F-]{2} \ar@{-}[r] &*+[o][F-]{3} \ar@{-}[r]&\cdots \ar@{-}[r]& *+[o][F-]{n}  }}
\]
The pointed topological ns operad 
$\As_{S^1}$ is defined by  $\As_{S^1}(n)\coloneqq \{*\}$, for $n=0$ and $n=1$, and by $\As_{S^1}(n)\coloneqq \big(S^1\big)^{n-1}$, for $n\geqslant 2$, with partial  composition products are given by  
 \[ 
(x_1,\ldots,x_{n-1})\circ_i(y_1,\ldots,y_{m-1}):=
(x_1,\ldots,x_{i-1},y_{1},\ldots,y_{m-1},x_i, \ldots,x_{n-1})\ . 
 \]

\begin{remark}
This ns topological operad is formal \cite[Corollary~8.1.1]{DotsenkoShadrinVallette15}. 
Again, any topological space $X$ can replace $S^1$ in order to form a similar nonsymmetric topological operad. 
For instance, any topological space homotopy equivalent to the circle, like $\mathbb{C}\backslash\{0\}$ for instance,  would produce a homotopy equivalent operad, which is also formal.
\end{remark}

\begin{definition}[skew--symmetric quadratic data $\LG$]
The \emph{skew--symmetric quadratic data} $\LG$, for Linear Graph,  are spanned by 
\[
\LG(n)\coloneq\left(
\e^n_{ii+1}\, ,\, \e^n_{ii+1}\wedge \e^n_{jj+1}
\right)\ ,
\]
where 
the set of generators $\e^n_{ii+1}$ of degree $0$ runs over the set of edges of $\Theta_n$ and where the  set of relations runs over all pairs of  edges of $\Theta_n$. 
For $n=0$ and for $n=1$, we set $\LG(0)\coloneq(0, 0)$ and $\LG(1)\coloneq(0, 0)$. 
\end{definition}

The morphisms $\circ_k : \LG(n)\oplus \LG(m) \to \LG(n+m-1)$ of skew-symmetric quadratic data defined by 
\[
\begin{array}{lcll}
& \e^n_{ii+1} & \mapsto & \left\{\begin{array}{lll}  
\e^{n+m-1}_{i+m-1 i+m} & \text{for} & k\leqslant i\ ,\\
\rule{0pt}{12pt} \e^{n+m-1}_{i i+1}& \text{for} & k>i\ , 
\end{array}\right.\\
\rule{0pt}{12pt} & \e^m_{ij} & \mapsto & \quad  \e^{n+m-1}_{j+k-1 j+k}\ ,\\
\end{array}
\]
endow the collection $\{\LG(n)\}$ with a nonsymmetric operad structure in the symmetric monoidal category $(\QD^-, \oplus)$. The Lie operad $\L(\LG)$ is  made up of the graphs with one edge (of degree 0) and with similar  partial composition maps. 

\begin{proposition}
The holonomy operad and the rational Magnus operad associated to $\As_{S^1}$ are isomorphic to the Lie operad associated to the skew--symmetric data $\LG$: 
\[\g_{\As_{S^1}}\cong \L(\LG) \cong \gr\left(\pi_1\big(\As_{S^1}\big)\right)\otimes \QQ \ .\]
\end{proposition}

\begin{proof}
This proof is similar to the proof of  \cref{prop:HoloGraS}.
\end{proof}

Mimicking the above definition of the operad $\mathrm{Gra}$, we introduce a nonsymmetric operad 
$\mathrm{LGra}$ made up of sub-graphs of the complete linear graph $\Theta_n$ and with the insertion at vertex $k$ for partial composition product: 
\[
\vcenter{\xymatrix@M=4pt@R=16pt@C=16pt{
		*+[o][F-]{1} \ar@{-}[r]& 		*+[o][F-]{2}   & *+[o][F-]{3}\ar@{-}[r] &*+[o][F-]{4}}}
\ \ \circ_3 \ \ 
\vcenter{\xymatrix@M=4pt@R=16pt@C=16pt{
		*+[o][F-]{1} \ar@{-}[r]&  *+[o][F-]{2}   & *+[o][F-]{3}}}
\ \  = \ \ 
\vcenter{\xymatrix@M=4pt@R=16pt@C=16pt{
		*+[o][F-]{1} \ar@{-}[r]& 		*+[o][F-]{2}   & *+[o][F-]{3}\ar@{-}[r] &*+[o][F-]{4}
		& *+[o][F-]{5} \ar@{-}[r]& *+[o][F-]{6}}}
\]
This actually forms a cocommutative Hopf nonsymmetric operad with the coproduct 
 $\Delta(\gamma)\coloneqq \sum \gamma'\otimes \gamma''$ where the edges from $\gamma$ are distributed on $\gamma'$ and $\gamma''$. 

\begin{proposition}
The following three cocommutative Hopf operads are isomorphic 
\[H_\bullet\big(\As_{S^1}\big)\cong \S^c\big(\LG^{\emph{\acc}}\big) \cong \mathrm{LGra}\ .\]
\end{proposition}

\begin{proof}
This proof is similar to the proof of \cref{prop:HoBKW}. 
\end{proof}

Again, the operad $\mathrm{U}\L(\LG)=\A\Lambda(\LG)$ in associative algebras is similar to the operad $\mathrm{LGra}$ but made up of linear graphs with possibly multiple edges (of degree 0) between consecutive vertices; the (algebra) product of two such graphs amounts to consider the union of their sets of edges.  

\medskip 

A noncommutative  version of the notion of  Batalin--Vilkovisky algebras was introduced in \cite[Section~3]{DotsenkoShadrinVallette15}; it is modelled by the nonsymmetric topological operad $\As_{S^1}\rtimes \,S^1$. 
The associated skew--symmetric quadratic data is similar but with $n$ extra generators 
 $\e^n_i$, for $1\leqslant i\leqslant n$. 
 The same results hold true \emph{mutatis mutandis} by considering now linear graphs with possible tadpoles.
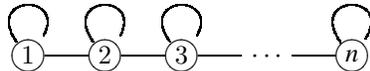
\[\rule{0pt}{20pt}
\vcenter{\xymatrix@M=4pt@R=16pt@C=16pt{
		*+[o][F-]{1}\ar@{-}[r] \ar@{-}@(ul,ur)[] &*+[o][F-]{2} \ar@{-}[r] \ar@{-}@(ul,ur)[] &*+[o][F-]{3} \ar@{-}[r]\ar@{-}@(ul,ur)[]&\cdots \ar@{-}[r]& *+[o][F-]{n} \ar@{-}@(ul,ur)[] }}
\]

\begin{remark}
The homology nonsymmetric operads $H_\bullet\big(\As_{S^1}\big)$ and $H_\bullet\big(\As_{S^1}\rtimes S^1\big)$ can also be twisted \`a la Willwacher to produce dg nonsymmetric operads. Their homology with respect to the twisted differential was computed in \cite[Section~6]{DotsenkoShadrinVallette18}. 
\end{remark}

\subsection{Drinfeld--Kohno and Arnold--Orlik--Solomon, i.e. $\mathcal{D}_2(n)\sim \mathrm{Conf}_n(\mathbb{C})$}\label{subsec:LDOp}
In this section, we refine the above mentioned Berger--Kontsevich--Willwacher skew--symmetric quadratic data following the works of Drinfeld \cite{Drinfeld90} and Kohno \cite{Kohno85}. We show that this refinement is canonical in a certain way. This theory corresponds to the topological operad $\mathcal{D}_2$, called the \emph{little disks operad}, which is made up of configurations of disks inside the unit disk. It is the mother of operads (the father being the endomorphism operad), which arose from the recognition of double loop spaces in \cite{BoardmanVogt73, May72}. Recall that the components of the little disks operad are homotopy equivalent to the configuration space of $n$ points  in the plane $\mathcal{D}_2(n)\sim \mathrm{Conf}_n(\mathbb{C})$. Notice that the little disks operad fails to be well pointed.

\begin{definition}[Drinfeld--Kohno skew--symmetric quadratic data]
The \emph{Drinfeld--Kohno skew--symmetric quadratic data} are spanned by 
\[
\mathrm{DK}(n)\coloneq\left(\t^n_{ij}\, , \, \t^n_{ij}\wedge \t^n_{kl} \ \&\  
\t^n_{ij}\wedge \big(\t^n_{ik}+\t^n_{jk}\big)
\right)\ ,
\]
where 
the set of generators $\t^n_{ij}$ of degree $0$ runs over the set of edges $ij$ of $\Gamma_n$, and where the first set of relations runs over pairs $(ij, kl)$ of disjoint edges and the second set of relations runs over triples of edges $(ij, jk, kl)$  which form a triangle in $\Gamma_n$. 
For $n=0$ and for $n=1$, we set $\mathrm{DK}(0)\coloneq(0, 0)$ and $\mathrm{DK}(1)\coloneq(0, 0)$. 
\end{definition}

We consider the same partial composition products as the ones for the Berger--Kontsevich--Willwacher quadratic data given in \cref{eq:CompoOp}. 

\begin{proposition}\label{prop:DKOp}
The Drinfeld--Kohno skew-symmetric quadratic data 
$\mathrm{DK}\coloneqq\big(\{\mathrm{DK(n)}\}, \{\circ_k\}\big)$
 forms an operad in the symmetric monoidal category \rm$(\QD^-, \oplus)$.
\end{proposition}

\begin{proof}
After the proof of \cref{lem:KWDataOp}, the only thing left to check is that the various maps $\circ_k$ induce morphisms of quadratic data, that is 
\[
(\circ_k)^{\wedge 2}\big(
R(n)\oplus [V(n), V(m)]_-\oplus 
R(m)
\big)\subset 
R(n+m-1)\ ,
\]
where we use the notation $\mathrm{DK}(n)=(V(n), R(n))$.
This can be proved by straightforward but tedious computations. It becomes much easier with the previous interpretation in terms of graph operad: one can see that any first (respectively second) type relation in $R(n)$ or $R(m)$ is sent to any first (respectively second) type relation in $R(n+m-1)$, that is pairs of disjoints edges (respectively graphs whose edges form a triangle). Regarding the relation $[V(n), V(m)]_-$, any of its elements $\t^n_{ij}\wedge \t^m_{kl}$ is sent, under $(\circ_k)^{\wedge 2}$, to a sum of relations of first and second type. \\
\end{proof}

The canonical morphisms of quadratic data $\mathrm{DK}(n)\to \mathrm{BKW}(n)$ induce a canonical morphism of operads $\mathrm{DK}\to \mathrm{BKW}$ in $\QD^-$. More generally, we call \emph{sub-operad of $\mathrm{BKW}$} any collection of skew-symmetric quadratic sub-data $\left(V(n)\, , \, R(n)\right)\subset \mathrm{BKW}(n)$ stable under the partial composition products $\circ_k$, where $V(n)$ is generated by the set of edges $\t^n_{ij}$.
 As usual, the intersection of all such sub-operads, explicitly given by the intersection of all the spaces of relations $R(n)$ for a fixed $n$ each time, produces the smallest sub-operad of  $\mathrm{BKW}$. 
The following statement provides us with a universal operadic characterisation of the  Drinfeld--Kohno quadratic data.

\begin{theorem}\label{thm:UnivCharDK}
 The operad $\mathrm{DK}$ is the smallest sub-operad of $\mathrm{BKW}$.
\end{theorem}
 
\begin{proof}
Let us continue to use the notation $\mathrm{DK}(n)=(V(n), R(n))$ and let us consider a sub-operad $\mathrm{P}(n)\coloneqq (V(n), S(n))\subset \mathrm{BKW}(n)$ of $\mathrm{BKW}$. We have to show that $R(n)\subset S(n)$ and this follows from the fact that the partial composition products 
$\circ_p : \mathrm{P}(n)\oplus \mathrm{P}(m) \to \mathrm{P}(n+m-1)$ sends 
$[V(n), V(m)]_-$ to $S(n+m-1)$ under $(\circ_p)^{\wedge 2}$. 
We begin with the relations of first type: $\t^n_{ij}\wedge \t^n_{kl}$. 
Using the action of the symmetric group, we can assume,  without any loss of generality, that $(i, j,k, l)=(1,2,3,4)$ and we conclude with  \[
(\circ_3)^{\wedge 2}\big(\t_{12}^{n-1}\wedge \t_{12}^2\big)=\t_{12}^n\wedge \t_{34}^n\ .
\]

 We treat now the relations of seconde type: $\t^n_{ij}\wedge \big(\t^n_{ik}+\t^n_{jk}\big)$. Using again the action of the symmetric group, the proof reduces to the case $(i, j,k)=(1,2,3)$, which is given by 
\[
(\circ_1)^{\wedge 2}\big(\t_{12}^{n-1}\wedge \t_{12}^2\big)=\big(\t_{13}^n+\t_{23}^n\big)\wedge \t_{12}^n\ .
\]
\end{proof} 
 
In the lattice of operads made up of skew-symmetric data with generators $\t_{ij}^n$ and partial composition products $\circ_k$, the Berger--Kontsevich--Willwacher operad $\mathrm{BKW}$ is the maximal element and the Drinfeld--Kohno operad $\mathrm{DK}$ is the minimal element. 

\begin{proposition}\label{prop:HoloD2}
The holonomy operad associated to $\mathcal{D}_2$ is isomorphic to the Lie operad associated to skew-symmetric quadratic data $\mathrm{DK}$: 
\[\g_{\mathcal{D}_2}\cong \L(\mathrm{DK})\ .\]
\end{proposition}

We will prove it below after the study of the (co)homology Hopf (co)operad. 
Since the little disks operad is not well pointed, we cannot consider directly a Lie operad of Magnus type here. 
Instead, one can consider the pointed topological operad $\mathcal{K}_2$ introduced by M. Kontsevich in \cite{Kontsevich99}, see also \cite{Sinha06}, since this latter one is homotopy equivalent to the little disks operad. Notice that both operads, $\mathcal{D}_2$ \cite{LambrechtsVolic14} and $\mathcal{K}_2$ \cite{PAST18},  are formal; they are even intrinsically Hopf formal by \cite{FresseWillwacher15}. 

\begin{remark}
Even if the little disks operad fails to be well pointed, its components  $\mathcal{D}_2(n)\sim \mathrm{Conf}_n(\mathbb{C})$ are path connected with fundamental groups isomorphic to the pure braid groups  $\pi_1\big(\mathrm{Conf}_n\allowbreak(\mathbb{C})\big)\cong \mathrm{PB}_n$. They are also  (rationally) formal by \cite{Arnold69}. So the Drinfeld--Kohno Lie algebras, as the holonomy Lie algebras of $\mathrm{Conf}_n(\mathbb{C})$ are the Lie algebras of infinitesimal braids
\[
\L(\mathrm{DK}(n))=\frac{Lie \big(\t^n_{ij}\big)}{\left(\big[\t^n_{ij}, \t^n_{kl}\big], \  \big[\t^n_{ij}, \t^n_{ik}+\t^n_{jk}\big]\right)}\cong \gr({\mathrm{PB}}_n)\otimes \QQ\ .
\] 
\end{remark}

\begin{remark}
 The fact that the little disks operad fails to be well pointed should be seen as a richness. 
Instead of considering the fundamental group $\pi_1$ of a pointed topological space, one can consider 
the fundamental groupoid $\Pi_1$. This latter functor $\Pi_1 : ({\textbf{Top}}, \times) \to ({\textbf{Grp}},\times)$ is cartesian and thus sends topological operads to operads in groupoids. The operad in groupoids $\Pi_1(\mathcal{D}_2)$ is equivalent to the operad in groupoids which encodes braided monoidal categories. Refining this operad with various "choices of base points" gives rise to various operads in groupoids and the  morphisms between them define the notion of Drinfeld's associators and Grothendieck--Teichm\"uller group(s), 
see \cite{Fresse17} for more details. 
\end{remark}

%\Bruno{1) On a la chord diagrams.}

The \emph{operad of chord diagrams} is the operad 
\[\mathrm{U}\L(\mathrm{DK})=\A\Lambda(\mathrm{DK})\]
 made up of  the  associative algebras of chord diagrams
\[
\mathrm{U}\L(\mathrm{DK}(n))=\frac{T\big(\t^n_{ij}\big)}{\left(\big[\t^n_{ij}, \t^n_{kl}\big], \  \big[\t^n_{ij}, \big(\t^n_{ik}+\t^n_{jk}\big)\big]\right)}
\ .\]
The name comes from the following pictorial way to represent its elements: 
\[
\t^5_{25}\t^5_{13}\t^5_{34}\t^5_{24}\,=\,\vcenter{\xymatrix@M=0pt@R=8pt@C=12pt{
		*{1} & *{2} & *{3} & *{4} & *{5} \\
		\ar@{-}[ddddd]&\ar@{-}[ddddd]&\ar@{-}[ddddd]&\ar@{-}[ddddd]&\ar@{-}[ddddd]\\
		&*{\bullet}\ar@{-}[rr]&&*{\bullet}&\\
		&&*{\bullet}\ar@{-}[r]&*{\bullet}&\\
		*{\bullet}\ar@{-}[rr]&&*{\bullet}&&\\
		&*{\bullet}\ar@{-}[rrr]&&&*{\bullet}\\
		&&&&
		}}\ \ .
\]

It plays a seminal role in the theory of Drinfled's associators \cite{Drinfeld90}, Grothendieck--Teichm\"uller group(s) \cite{Fresse17},  the formality of the little discs operad \cite{Tamarkin03, SeveralWillwacher11, FresseWillwacher15} and Vassiliev knot invariants \cite{BarNatan95}.

\medskip 

\begin{definition}[Arnold--Orlik--Solomon symmetric quadratic data]
The \emph{Arnold--Orlik--Solomon symmetric quadratic data} are spanned by 
\[
\mathrm{AOS}(n)\coloneq\left(\omega^n_{ij}\, ,\, 
\omega^n_{ij} \odot \omega^n_{jk} + \omega^n_{jk} \odot \omega^n_{ki} +\omega^n_{ki} \odot \omega^n_{ij}
\right)\ ,
\]
where 
the set of generators $\omega^n_{ij}$ of degree $-1$ runs over the set of edges $ij$ of $\Gamma_n$, and where  the  set of relations runs over increasing triples $i<j<k$.  
For $n=0$ and for $n=1$, we set $\mathrm{AOS}(0)\coloneq(0, 0)$ and $\mathrm{AOS}(1)\coloneq(0, 0)$. 
\end{definition}

Arnold proved in \cite{Arnold69} that the Orlik-Solomon algebras
\[
\S\big(\mathrm{AOS}(n)\big)=
\frac{S\big(\omega^n_{ij}\big)}
{\left(
\omega^n_{ij} \odot \omega^n_{jk} + \omega^n_{jk} \odot \omega^n_{ki} +\omega^n_{ki} \odot \omega^n_{ij} 
\right)}
\cong H^{\bullet}\left(\mathcal{D}_2(n)\right)
\]
compute the cohomology algebras of the configuration spaces of points in the plane. One can see by a direct computation that 
\[\mathrm{DK}^!\cong\mathrm{AOS}\ . \]
Equivalently, this means that $\mathrm{DK}^\acc\cong\mathrm{AOS}^*$, which provides us with the following presentation of the cocommutative coalgebras underlying the homology operad 
\[
\Sc\big(\mathrm{AOS}^*(n)\big)= \Sc\left(w^n_{ij}, w^n_{ij}\odot w^n_{kl} \ \&\  w^n_{ij}\odot \big(w^n_{ik}+w^n_{jk}\big)\right)
\cong H_\bullet(\mathcal{D}_2(n))\ , 
\] 
where $w^n_{ij}=\left(\omega^n_{ij}\right)^*=s\t^n_{ij}$ has degree $1$. The presentation of the homology operad was given by F.R. Cohen in \cite{Cohen76}: it is shown to be isomorphic to the the operad encoding Gerstenhaber algebras
 $H_\bullet(\mathcal{D}_2)\cong \mathrm{Gerst}$,  see \cite[Section~13.3]{LodayVallette12}. 

\begin{proof}[Proof of \cref{prop:HoloD2}]
We go back to the definition and we follow the same kind of arguments as in the proof of \cref{prop:HoBKW}. 
If we denote the operadic structure maps of $\mathrm{DK}$ by $\circ_k$, the ones of
the homology operad $H_\bullet(\mathcal{D}_2)$ by $\tilde{\circ}_k$,  and the counits of the homology coalgebras by $\mathbb{1}^n\in H_\bullet(\mathcal{D}_2(n))$, we have the following commutative diagram
\[\xymatrix@C=30pt@R=30pt{
H_1(\mathcal{D}_2(n))\oplus H_1(\mathcal{D}_2(m)) \ar[r]^{\circ_k}\ar[d]^{\cong}& H_1(\mathcal{D}_2(n+m-1))\ \, \ar@{=}[d]\\
H_1(\mathcal{D}_2(n))\ot \QQ\mathbb{1}^m\,  \oplus\,  \QQ\mathbb{1}^n \ot H_1(\mathcal{D}_2(m)) \ar[r]^(0.63){\tilde{\circ}_k}& H_1(\mathcal{D}_2(n+m-1))\ .
}
\]
The isomorphism of operads  $H_\bullet(\mathcal{D}_2)\cong \mathrm{Gerst}$ of \cite{Cohen76}, see also the survey \cite{Sinha13}, identifies the following elements 
\[
w^n_{ij}\longleftrightarrow 
\vcenter{\hbox{\begin{tikzpicture}[yscale=0.5,xscale=0.5]
\draw (0,-1) -- (0,2) -- (1,3);
\draw (0,2) -- (-1, 3);
\draw (0,0) -- (-1, 1);
\draw (0,0) -- (-2, 1);
\draw (0,0) -- (1, 1);
\draw (0,0) -- (2, 1);
\node at (-1,3.4) {\scalebox{0.8}{$i$}};
\node at (1,3.4) {\scalebox{0.8}{$j$}};
\node at (-2,1.4) {\scalebox{0.8}{$1$}};
\node at (2,1.4) {\scalebox{0.8}{$n$}};
\node at (1,1.4) {\scalebox{0.8}{$\cdots$}};
\node at (-1,1.4) {\scalebox{0.8}{$\cdots$}};
\node at (0,2) {$\bullet$};
\end{tikzpicture}}}
\qquad \text{and} \qquad 
\mathbb{1}^n \longleftrightarrow 
\vcenter{\hbox{\begin{tikzpicture}[yscale=0.5,xscale=0.5]
\draw (0,-1) -- (0,1) ;
\draw (0,0) -- (-1, 1);
\draw (0,0) -- (-2, 1);
\draw (0,0) -- (1, 1);
\draw (0,0) -- (2, 1);
\node at (-2,1.4) {\scalebox{0.8}{$1$}};
\node at (2,1.4) {\scalebox{0.8}{$n$}};
\node at (0,1.4) {\scalebox{0.8}{$\cdots$}};
\end{tikzpicture}}}
\ ,
\]
where $\bullet$ denotes the shifted Lie bracket and where  the bottom corollas denote the iterations of the commutative product. Under this correspondence, the operad structure on $\mathrm{Gerst}$ produces the 
formul\ae\ given in \cref{eq:CompoOp}. Let us illustrate this on the less trivial case: the partial composite 
$w^n_{ij} \circ_i \mathbb{1}^m$ amounts to graft the above right-hand side corolla with $m$ leaves at the  input $i$ of the left-hand side corolla. Using iteratively the Leibniz relation, one rewrites this 3-vertices trees into a sum of 2-vertices trees, which correspond to 
\[w^n_{ij} \circ_i \mathbb{1}^m =w^{n+m-1}_{i\, j+m-1}+w^{n+m-1}_{i+1\, j+m-1}+\cdots+ w^{n+m-1}_{i+m-2\, j+m-1}+w^{n+m-1}_{i+m-1\, j+m-1}\ .
\]
\end{proof}

The canonical morphism of operads $\mathrm{DK}\to \mathrm{BKW}$ in $\QD^-$ induces a canonical morphism  of operads 
in associative algebras between the operad of chord diagrams and the operad of graphs with multiple edges mentioned at the end of \cref{subsec:BKW}. It also induces the canonical morphism $\mathrm{Gerst} \to \mathrm{Gra}$ of cocommutative Hopf operads, whose deformation complex  gives the Grothendieck--Teichm\"uller Lie algebra $\mathfrak{grt}$  in \cite{Willwacher15}. 

\subsection{Hypergraphs}\label{subsec:Hygraph}
The purpose of this subsection is to extend \cref{subsec:BKW} from graphs to hypergraphs. 
This latter notion amounts  to ``graphs'' where ``edges'' can now join an arbitrary number of vertices. 

\begin{definition}[Hypergraph]
An \emph{hypergraph} is a pair $(V, E)$ where $V$ is a set of vertices and where $E$ is a set of subsets of $V$, called \emph{hyperedges}. 
\end{definition}

In the sequel, we will mainly consider the sets $V=\un=\{1, \ldots, n\}$, for $n\geqslant 2$. We will only consider hypergraphs where the elements of $E$ have all cardinal equal to $k$, for $k\geqslant 2$; they will be called \emph{$k$-hypergraphs}. 
For example, the \emph{complete} $k$-hypergraph $\Gamma^k_n$ on $n$ vertices is $(V,E)$, where $E$ is the set of all subset of $\{1, \ldots, n\}$ with $k$-elements. In the case $k=2$, we recover the complete graph $\Gamma_n=\Gamma^2_n$ of \cref{subsec:BKW}. 
\[
\Gamma^3_4\,=\,\vcenter{\hbox{
\begin{tikzpicture}[math3d, yscale=2,xscale=2]
\draw[fill, opacity=0.2] (1,0,0) -- (1.2,0.8,1) -- (1.5,1,0) --(1,0,0) ;
\draw[fill, opacity=0.15] (1.2,0.8,1) -- (1.5,1,0) -- (0,1,0) -- (1.2,0.8,1) ;

\draw[semithick] (1,0,0) -- (1.2,0.8,1); % 01
\draw[semithick] (1,0,0) -- (1.5,1,0); % 02
\draw[semithick] (1.5,1,0) -- (0,1,0); % 23
\draw[very thin,dashed] (1,0,0) -- (0,1,0); % 03
\draw[semithick] (1.2,0.8,1) -- (1.5,1,0); % 12
\draw[semithick] (1.2,0.8,1) -- (0,1,0); % 13

\node at (0.99,0.01,0) {$\bullet$};
\node[left] at (0.99,0.01,0) {\scalebox{0.8}{$1$}};
\node at (1.2,0.803,0.99) {$\bullet$};
\node[above] at (1.2,0.803,0.99) {\scalebox{0.8}{$2$}};
\node at (1.5,1,0) {$\bullet$};
\node[below] at (1.5,1,0) {\scalebox{0.8}{$3$}};
\node at (0,1,0) {$\bullet$};
\node[right] at (0,1,0) {\scalebox{0.8}{$4$}};
\end{tikzpicture}}}
\]
We define the topological operad of complete $k$-hypergraph by 
\[\GraS^k(n)\coloneqq \{*\}\, ,\  \text{for} \  n<k,  \quad  \text{and}\quad \GraS^k(n)\coloneqq \big(S^1\big)^{\binom{n}{k}}\, ,\  \text{for} \  n\geqslant k\ .\]
The elements $\big\{\mu_{I}Ê; I\subset\un\,, |I|=k \big\}$ of $\GraS^k(n)$ are thought of
as collections of labels, living in  the circle $S^1$, for every hyperedges $I$ of the complete $k$-hypergraph $\Gamma^k_n$.
 The  partial  composition products $\circ_p : \GraS^k(n)\times \GraS^k(m) \to 
\GraS^k(n+m-1)$ of two collections $\{\mu_{I}\}$ and $\{\mu'_{J}\}$ are defined in a way similar to that of the operad $\GraS$. 
We first insert the complete $k$-hypergraph $\Gamma^k_m$ at the $p$th vertex of the complete $k$-hypergraph $\Gamma^k_n$ and then we relabel the vertices accordingly. 
The hyperedges coming from $\Gamma^k_n$ (respectively  $\Gamma^k_m$)
are labeled by the according $\mu_{I}$ (respectively $\mu'_{J}$). 
The hyperedges made up of $k-1$ vertices $i_1, \ldots, i_{k-1}$ from $\Gamma^k_n$ and one vertex from $\Gamma^k_m$ are labelled by $\mu_{i_1, \ldots, i_{k-1}, p}$. All the other hyperedges are labelled by the base point $*$.

\begin{proposition}
The data $\GraS^k\coloneqq \big(\{\GraS^k(n)\}, \{\circ_p\}\big)$ forms a pointed topological operad, which is formal over $\mathbb{Z}$. 
\end{proposition}

\begin{proof}
It is straightforward to check the sequential and parallel axioms, the equivariance with respect to the symmetric groups action, as well as the axioms for the unit. The formality property is proved by the same arguments and computations as in \cite[Section~8]{DotsenkoShadrinVallette15}.

\end{proof}

The special case $k=2$ gives back the operad $\GraS=\GraS^2$ of \cref{subsec:BKW}.

\begin{definition}[$k$-Hypergraph skew--symmetric quadratic data]
The \emph{$k$-Hypergraph skew--symmetric quadratic data} are spanned by 
\[
\kHG(n)\coloneqq\left(
\t^n_{I}\, ,\, \t^n_{I}\wedge \t^n_{J}
\right)\ ,
\]
where 
the set of generators $\t^n_{I}$ of degree $0$ runs over the set of hyperedges $I$ of $\Gamma^k_n$ and where the  set of relations runs over all pairs $(I, J)$ of  hyperedges of $\Gamma^k_n$. 
For $n<k$, we set $\kHG(n)\coloneqq(0, 0)$.
\end{definition}

We consider the following maps $\circ_p : \kHG(n)\oplus \kHG(m)\to \kHG(n+m-1)$. Let us denote $I=\{i_1,\ldots, i_k\}$ and use the notation $I+a\coloneqq \{i_1+a,\ldots, i_k+a\}$. 
\begin{equation}\label{eq:CompoOpHYGRA}
\begin{array}{lcll}
& \t^n_{I} & \mapsto & \left\{\begin{array}{lll}  
\t^{n+m-1}_{I+m-1} & \text{for} & p<i_1\ ,\\
\rule{0pt}{12pt} 
\sum_{j=0}^{m-1} \t^{n+m-1}_{i_1, \ldots, i_{l-1}, i_l+j, i_{l+1}+m-1, \ldots,i_k+m-1}
& \text{for} & p=i_l\ , \\
\rule{0pt}{12pt} \t^{n+m-1}_{i_1, \ldots, i_{l}, i_{l+1}+m-1,\ldots, i_k+m-1} & \text{for} & i_l<p<i_{l+1}\ , \\
\rule{0pt}{12pt} \t^{n+m-1}_{I} & \text{for} & i_k<p \ ,
\end{array}\right.\\
\rule{0pt}{12pt} & \t^m_{I} & \mapsto & \quad  \t^{n+m-1}_{I+p-1}\ .\\
\end{array}
\end{equation}

\begin{lemma}\label{lem:kHGDataOp}
The aforementioned data 
$\kHG \coloneqq \big( \{k\text{-}\mathrm{HG}(n)\}, \{\circ_p\}\big)$
 forms an operad in the symmetric monoidal category \rm$(\QD^-, \oplus)$.
\end{lemma}

\begin{proof}
Since the spaces of relations are the full spaces, the maps $\circ_p$ are morphisms of quadratic data.
One can check directly that they form an operad structure. This can be done easily by viewing the elements $t^n_I$ as the $k$-hypergraph with one hyperedge $I$ and respectively by inserting the empty $k$-hypergraph at its $p$th vertex or by inserting it into the empty $k$-hypergraph. 
\end{proof}

In the special case $k=2$, we recover the Berger--Kontsevich--Willwacher operad $\mathrm{BKW}=2\text{-}\mathrm{HG}$.

\begin{proposition}\label{prop:HoloGraK}
The holonomy operad and the rational Magnus operad associated to $\GraS^k$ are isomorphic to the Lie operad associated to skew-symmetric quadratic data $\kHG$: 
\[\g_{\GraS^k}\cong \L(\kHG)\cong \gr\left(\pi_1\big(\GraS^k\big)\right)\otimes \QQ \ .\]
\end{proposition}

\begin{proof}
This proof is the same \emph{mutatis mutandis} as the one of \cref{prop:HoloGraS}
\end{proof}

Let us introduce the \emph{$k$-hypergraph graded operad} $\mathrm{Gra}^k$. Its underlying $\Sy$-modules are spanned by
sub-hypergraphs of $\Gamma^k_n$, where each hyperedge receives degree $1$. The partial composition product $\gamma_1\circ_p\gamma_2$ amounts to first inserting the $k$-hypergraph $\gamma_2$ at the $p$th vertex of $\gamma_1$, then relabelling accordingly the vertices, and finally considering the sum of all the possible ways to connect the hyperedges in $\gamma_1$ containing the vertex $p$, to any possible vertex of $\gamma_2$. 
 
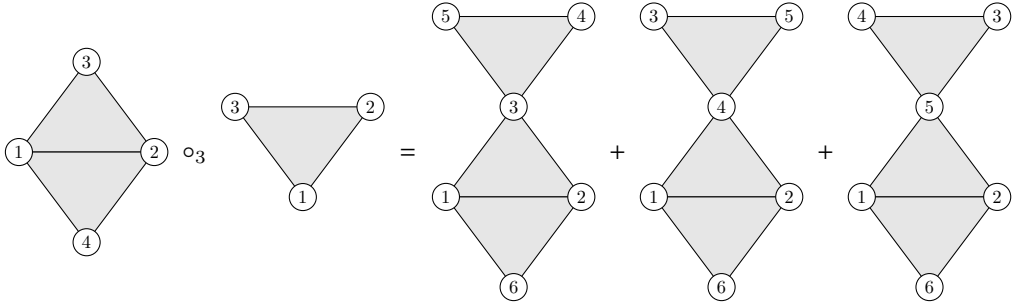
\[
\vcenter{\hbox{\begin{tikzpicture}[scale=0.6]
\draw[fill=gray!20] 
(-1.5,0)--(1.5,0)--(0, 2)--(-1.5,0);
\draw[fill=gray!20] 
(-1.5,0)--(1.5,0)--(0, -2)--(-1.5,0);

\draw [fill=white] (-1.5,0) circle [radius=0.3];
 \node at (-1.5,0) {\scalebox{0.7}{$1$}}; 
 \draw [ fill=white] (1.5,0) circle [radius=0.3];
 \node at (1.5,0) {\scalebox{0.7}{$2$}}; 
 \draw [ fill=white] (0,2) circle [radius=0.3];
 \node at (0,2) {\scalebox{0.7}{$3$}}; 
  \draw [ fill=white] (0,-2) circle [radius=0.3];
 \node at (0,-2) {\scalebox{0.7}{$4$}}; 
 \end{tikzpicture}}}
 \ \circ_3 \
\vcenter{\hbox{\begin{tikzpicture}[scale=0.6]
\draw[fill=gray!20] 
(-1.5,0)--(1.5,0)--(0, -2)--(-1.5,0);

\draw [fill=white] (-1.5,0) circle [radius=0.3];
 \node at (-1.5,0) {\scalebox{0.7}{$3$}}; 
 \draw [ fill=white] (1.5,0) circle [radius=0.3];
 \node at (1.5,0) {\scalebox{0.7}{$2$}}; 
  \draw [ fill=white] (0,-2) circle [radius=0.3];
 \node at (0,-2) {\scalebox{0.7}{$1$}}; 
 \end{tikzpicture}}}
\ = \ 
\vcenter{\hbox{\begin{tikzpicture}[scale=0.6]
\draw[fill=gray!20] 
(-1.5,0)--(1.5,0)--(0, 2)--(-1.5,0);
\draw[fill=gray!20] 
(-1.5,0)--(1.5,0)--(0, -2)--(-1.5,0);
\draw[fill=gray!20] 
(-1.5,4)--(1.5,4)--(0, 2)--(-1.5,4);

\draw [fill=white] (-1.5,4) circle [radius=0.3];
 \node at (-1.5,4) {\scalebox{0.7}{$5$}}; 
\draw [fill=white] (1.5,4) circle [radius=0.3];
 \node at (1.5,4) {\scalebox{0.7}{$4$}}; 
\draw [fill=white] (-1.5,0) circle [radius=0.3];
 \node at (-1.5,0) {\scalebox{0.7}{$1$}}; 
 \draw [ fill=white] (1.5,0) circle [radius=0.3];
 \node at (1.5,0) {\scalebox{0.7}{$2$}}; 
 \draw [ fill=white] (0,2) circle [radius=0.3];
 \node at (0,2) {\scalebox{0.7}{$3$}}; 
  \draw [ fill=white] (0,-2) circle [radius=0.3];
 \node at (0,-2) {\scalebox{0.7}{$6$}}; 
 \end{tikzpicture}}}
\ + \ 
\vcenter{\hbox{\begin{tikzpicture}[scale=0.6]
\draw[fill=gray!20] 
(-1.5,0)--(1.5,0)--(0, 2)--(-1.5,0);
\draw[fill=gray!20] 
(-1.5,0)--(1.5,0)--(0, -2)--(-1.5,0);
\draw[fill=gray!20] 
(-1.5,4)--(1.5,4)--(0, 2)--(-1.5,4);

\draw [fill=white] (-1.5,4) circle [radius=0.3];
 \node at (-1.5,4) {\scalebox{0.7}{$3$}}; 
\draw [fill=white] (1.5,4) circle [radius=0.3];
 \node at (1.5,4) {\scalebox{0.7}{$5$}}; 
\draw [fill=white] (-1.5,0) circle [radius=0.3];
 \node at (-1.5,0) {\scalebox{0.7}{$1$}}; 
 \draw [ fill=white] (1.5,0) circle [radius=0.3];
 \node at (1.5,0) {\scalebox{0.7}{$2$}}; 
 \draw [ fill=white] (0,2) circle [radius=0.3];
 \node at (0,2) {\scalebox{0.7}{$4$}}; 
  \draw [ fill=white] (0,-2) circle [radius=0.3];
 \node at (0,-2) {\scalebox{0.7}{$6$}}; 
 \end{tikzpicture}}}
\ + \ 		
\vcenter{\hbox{\begin{tikzpicture}[scale=0.6]
\draw[fill=gray!20] 
(-1.5,0)--(1.5,0)--(0, 2)--(-1.5,0);
\draw[fill=gray!20] 
(-1.5,0)--(1.5,0)--(0, -2)--(-1.5,0);
\draw[fill=gray!20] 
(-1.5,4)--(1.5,4)--(0, 2)--(-1.5,4);

\draw [fill=white] (-1.5,4) circle [radius=0.3];
 \node at (-1.5,4) {\scalebox{0.7}{$4$}}; 
\draw [fill=white] (1.5,4) circle [radius=0.3];
 \node at (1.5,4) {\scalebox{0.7}{$3$}}; 
\draw [fill=white] (-1.5,0) circle [radius=0.3];
 \node at (-1.5,0) {\scalebox{0.7}{$1$}}; 
 \draw [ fill=white] (1.5,0) circle [radius=0.3];
 \node at (1.5,0) {\scalebox{0.7}{$2$}}; 
 \draw [ fill=white] (0,2) circle [radius=0.3];
 \node at (0,2) {\scalebox{0.7}{$5$}}; 
  \draw [ fill=white] (0,-2) circle [radius=0.3];
 \node at (0,-2) {\scalebox{0.7}{$6$}}; 
 \end{tikzpicture}}}
\]
Every graded vector space $\mathrm{Gra}^k(n)$ forms a cocommutative coalgebra with the coproduct $\Delta(\gamma)$ made up of the pairs of graphs $\gamma'\otimes \gamma''$ with the same  $n$ vertices as $\gamma$ but  with hyperedges from $\gamma$ distributed on $\gamma'$ and $\gamma''$. The partial composition products preserve these coproducts, thus  $\mathrm{Gra}^k$ forms a cocommutative Hopf operad. 

\begin{proposition}\label{prop:HokHG}
The following three cocommutative Hopf operads are isomorphic 
\[H_\bullet\big(\GraS^k\big)\cong \S^c\big(\kHG^{\emph{\acc}}\big) \cong \mathrm{Gra^k}\ .\]
\end{proposition}

\begin{proof}
This proof is similar to that of \cref{prop:HoBKW}. 
\end{proof}

\begin{remark}
For any $k\geqslant 2$, there is  a canonical morphism $s\mathrm{Lie}^k\to \mathrm{Gra}^k$ from the operad of shifted $k$-Lie algebras which sends its generator to $\t^k_{1 \ldots k}$. This latter notion is made up of a ``Lie bracket'' of degree $1$ with $k$-inputs satisfying a generalised Jacobi relation, see \cite[Section~13.11.3]{LodayVallette12}. Since this operad (unshifted) is the unit for the black product of $k$-ary quadratic operads, one can develop a similar twisting procedure as that of  \cite{Willwacher15} according to \cite[Remark~5.8]{DotsenkoShadrinVallette18}. The study of the resulting dg operad $\mathrm{TwGra}^k$ is a very interesting subject. 
\end{remark}

\subsection{Etingof--Henriques--Kamnitzer--Rains, i.e. $\overline{\mathcal{M}}_{0,n+1}(\mathbb{R})$}\label{subsec:EHKR}
In the very same way as the 
Drinfeld--Kohno quadratic data
 refines the Berger--Kontsevich--Willwacher quadratic data in canonical way, 
 the Etingof--Henriques--Kamnitzer--Rains quadratic data refines the $3$-Hypergraph quadratic data in a canonical way. This new one actually comes from the 
topological operad made up of the \emph{real locus} of the moduli spaces of stable curves of genus $0$ with marked points $\overline{\mathcal{M}}_{0,n+1}(\mathbb{R})$,  studied in depth in \cite{EHKR10}. 
\begin{definition}[Etingof--Henriques--Kamnitzer--Rains skew-symmetric quadratic data]
The \emph{Etingof--Henriques--Kamnitzer--Rains skew--symmetric quadratic data} are spanned by 
\[
\mathrm{EHKR}(n)\coloneq\left(\t^n_{ijk}\, , \, \t^n_{ijk}\wedge \t^n_{lmn} \ \&\  
\t^n_{ijk}\wedge \big(\t^n_{lmi}+\t^n_{lmj}+\t^n_{lmk}\big)
\right)\ ,
\]
where 
the set of generators $\t^n_{ijk}$ of degree $0$ runs over the set of hyperedges $ijk$ of $\Gamma^3_n$, and where the first set of relations runs over pairs $(ijk, lmn)$ of disjoint hyperedges and the second set of relations runs over pairs  $(ijk, \{l,m\})$  formed by an hyperedge and two separate vertices of $\Gamma^3_n$. 
For $n<k$, we set $\mathrm{EHKR}(n)\coloneq(0, 0)$.
\end{definition}

We consider the same partial composition products as the ones for the $3$-Hypergraph quadratic data given in \cref{eq:CompoOpHYGRA}. 

\begin{proposition}\label{prop:EHKROp}
The Etingof--Henriques--Kamnitzer--Rains skew-symmetric quadratic data 
$\EHKR\coloneqq\big(\{\EHKR(n)\}, \{\circ_p\}\big)$
 forms an operad in the symmetric monoidal category \rm$(\QD^-, \oplus)$.
\end{proposition}

\begin{proof}
As in the proof of \cref{prop:DKOp}, we only need to check that the various maps $\circ_p$ induce morphisms of quadratic data. Again, this can be achieved  easily with the $3$-hypergraph description: 

the first (respectively second) type relation in $R(n)$ or $R(m)$ is sent to any first (respectively second) type relation in $R(n+m-1)$, that is pairs of disjoints hyperedges (respectively a sum of hypergraphs based on pentagons with a distinguished triangle). Any element $\t^n_{ijk}\wedge \t^m_{lmn}$ of the relation $[V(n), V(m)]_-$ is sent to a sum of relations of first and second type under $(\circ_p)^{\wedge 2}$.
\end{proof}

The canonical morphisms of quadratic data $\EHKR(n)\to 3\text{-}\mathrm{HG}(n)$ induces a canonical morphism of operads $\EHKR\to 3\text{-}\mathrm{HG}$ in $\QD^-$.
The following statement is a  universal operadic characterisation of the  Etingof--Henriques--Kamnitzer--Rains skew-symmetric quadratic data.

\begin{theorem}\label{thm:UnivCharEHKR}
 The operad $\mathrm{EHKR}$ is the smallest sub-operad of $3\text{-}\mathrm{HG}$.
\end{theorem}
 
\begin{proof}
This proof is similar to that of \cref{thm:UnivCharDK}. It is also the particular case $n=3$ of \cref{thm:UnivCharGENERAL}. 
\end{proof} 
 
In the lattice of operads made up of skew-symmetric data with generators $\t_{ijk}^n$ and partial composition products $\circ_p$, the $3$-hypergraphs operad $3\text{-}\mathrm{HG}$  is the maximal element and the Etingof--Henriques--Kamnitzer--Rains operad $\mathrm{EHKR}$ is the minimal element. 

\begin{proposition}\label{prop:HoloMR}
The holonomy operad associated to $\overline{\mathcal{M}}_{0,n+1}(\mathbb{R})$ is isomorphic to the Lie operad associated to skew-symmetric quadratic data $\mathrm{EKHR}$: 
\[\g_{\overline{\mathcal{M}}_{0,n+1}(\mathbb{R})}\cong \L(\EHKR)\ .\]
\end{proposition}

\begin{remark}
The topological operad $\overline{\mathcal{M}}_{0,n+1}(\mathbb{R})$ fails to be well pointed; its components are  connected with fundamental groups $\pi_1\big(\overline{\mathcal{M}}_{0,n+1}(\mathbb{R})\big)\cong\mathrm{PC}_n$ called the pure cactus group in \cite{EHKR10}. It important to notice that these topological spaces however fail to be formal for $n\geqslant 6$.  It is however conjectured in \cite{EHKR10} that the Etingof--Henriques--Kamnitzer--Rains holonomy Lie algebras are isomorphic to the Magnus construction
\[
\L(\EHKR(n))=\frac{Lie \big(\t^n_{ijk}\big)}{\left(\big[\t^n_{ijk}, \t^n_{lmn}\big], \  
\big[\t^n_{ijk}, \t^n_{lmi}+\t^n_{lmj}+\t^n_{lmk}\big]\right)}\cong \gr({\mathrm{PC}}_n)\otimes \QQ\ .
\] 
The operad in groupoids $\Pi_1\big(\overline{\mathcal{M}}_{0,n+1}(\mathbb{R})\big)$ is equivalent to the operad in groupoids which encodes coboundary monoidal categories, see \cite{HK06}. 
\end{remark}

It is proved in \cite[Proposition~3.1]{EHKR10} that the Koszul dual symmetric data $\EHKR^!$ is equal to 
\[
\EHKR^!(n)=\left(\omega^n_{ijk}\, , \, % \nu^n_{ijk}\odot \nu^n_{lmn} \ \&\  
\omega^n_{ijk}\odot \omega^n_{klm} +
\omega^n_{jkl}\odot \omega^n_{lmi} +
\omega^n_{klm}\odot \omega^n_{mij} +
\omega^n_{lmi}\odot \omega^n_{ijk} +
\omega^n_{mij}\odot \omega^n_{jkl} 
\right)\ ,
\]
where the generator have (homological) degree $-1$. 
The main theorem of \cite{EHKR10} asserts that this quadratic data provides us with a presentation of the cohomology algebra:
\begin{align*}
\S\big(\mathrm{EHKR}^!(n)\big)&=
\frac{S\big(\omega^n_{ijk}\big)}
{\left(
\omega^n_{ijk}\odot \omega^n_{klm} +
\omega^n_{jkl}\odot \omega^n_{lmi} +
\omega^n_{klm}\odot \omega^n_{mij} +
\omega^n_{lmi}\odot \omega^n_{ijk} +
\omega^n_{mij}\odot \omega^n_{jkl} 
\right)}\\
&\cong H^{\bullet}\left(\overline{\mathcal{M}}_{0,n+1}(\mathbb{R})\right)\ .
\end{align*}
A presentation for the homology operad is also given in \emph{loc. cit.}: it is shown to be isomorphic to the the operad encoding unital 2-Gerstenhaber algebras
 $H_\bullet(\overline{\mathcal{M}}_{0,n+1}(\mathbb{R}))\cong 2\text{-}\mathrm{uGerst}$. This kind of algebraic structure is made up of a degree $0$ unital commutative product and a degree $1$ skew-symmetric ``2-Lie bracket'' of arity $3$ which satisfy generalised Leibniz and Jacobi relations.

\begin{proof}[Proof of \cref{prop:HoloMR}]
This proof is similar to the one of \cref{prop:HoloD2}. 
The isomorphism of operads  
 $H_\bullet\big(\overline{\mathcal{M}}_{0,n+1}(\mathbb{R})\big)\cong 2\text{-}\mathrm{uGerst}$ of \cite{EHKR10}
identifies the following elements 
\[
w^n_{ijk}\longleftrightarrow 
\vcenter{\hbox{\begin{tikzpicture}[yscale=0.5,xscale=0.5]
\draw (0,-1) -- (0,2) -- (1,3);
\draw (0,2) -- (-1, 3);
\draw (0,0) -- (-1, 1);
\draw (0,0) -- (-2, 1);
\draw (0,0) -- (1, 1);
\draw (0,0) -- (2, 1);
\draw (0,2) -- (0,3);
\node at (-1,3.4) {\scalebox{0.8}{$i$}};
\node at (0,3.4) {\scalebox{0.8}{$j$}};
\node at (1,3.4) {\scalebox{0.8}{$k$}};
\node at (-2,1.4) {\scalebox{0.8}{$1$}};
\node at (2,1.4) {\scalebox{0.8}{$n$}};
\node at (1,1.4) {\scalebox{0.8}{$\cdots$}};
\node at (-1,1.4) {\scalebox{0.8}{$\cdots$}};
\node at (0,2) {$\bullet$};
\end{tikzpicture}}}
\qquad \text{and} \qquad 
\mathbb{1}^n \longleftrightarrow 
\vcenter{\hbox{\begin{tikzpicture}[yscale=0.5,xscale=0.5]
\draw (0,-1) -- (0,1) ;
\draw (0,0) -- (-1, 1);
\draw (0,0) -- (-2, 1);
\draw (0,0) -- (1, 1);
\draw (0,0) -- (2, 1);
\node at (-2,1.4) {\scalebox{0.8}{$1$}};
\node at (2,1.4) {\scalebox{0.8}{$n$}};
\node at (0,1.4) {\scalebox{0.8}{$\cdots$}};
\end{tikzpicture}}}
\ ,
\]
where the degree $1$ element 
$w_{ijk}^n$ stands for $\left(\omega^n_{ijk}\right)^*=s\t^n_{ijk}$ and where 
$\bullet$ denotes the shifted 2-Lie bracket. Under this correspondence, the operad structure on $2\text{-}\mathrm{uGerst}$ produces the 
formul\ae\ given in \cref{eq:CompoOpHYGRA}. For instance, the partial composite 
$w^n_{ijk} \circ_i \mathbb{1}^m$ gives
\[w^n_{ijk} \circ_i \mathbb{1}^m =
\sum_{l=0}^{m-1} w^{n+m-1}_{i+l,j+m-1,k+m-1}\ ,
\]
 by the generalised Leibniz relation.
\end{proof}

The canonical morphism of operads $\EHKR\to 3\text{-}\mathrm{HG}$ in $\QD^-$ induces a  morphism  of cocommutative Hopf operads 
\[
2\text{-}\mathrm{uGerst} \cong H_\bullet\big(\overline{\mathcal{M}}_{0,n+1}(\mathbb{R})\big)
\cong \Sc\big(\mathrm{EHKR}^{\acc}\big)
\to 
\mathrm{Gra^3} \cong H_\bullet\big(\GraS^3\big)\cong \Sc\big(3\text{-}\mathrm{HG}^{{\acc}}\big) \ .
\]
In the light of \cite{Willwacher15}, the study of the deformation complex of this morphism of operads is a very interesting  question. 
What is the analogue of the Grothendieck--Teichm\"uller Lie algebra $\mathfrak{grt}$ (case $n=2$) here (case $n=3$)? 

\subsection{Linear hypergraphs and real brick manifolds $\mathcal{B}_{\mathbb{R}}(n)$}
Following the same pattern, one can give a $k$-hypergraph generalisation of \cref{subsec:ncD} too. The starting point amounts to considering only \emph{linear $k$-hypergraphs}, i.e. the ones made up of intervals of length $k$. We denote the complete linear $k$-hypergraph by $\Theta^k_n$. 
We define the pointed topological ns operad 
$\As^k_{S^1}$  by  $\As^k_{S^1}(n)\coloneqq \{*\}$, for $n<k$, and by $\As_{S^1}(n)\coloneqq \big(S^1\big)^{n-k+1}$, for $n\geqslant k$. Its  elements $\big(
x_{1 k}, \ldots, x_{n-k+1 n}
\big)$ are seen as labels, living in $S^1$, of the intervals of lengths $k$ of $\{1, n\}$. The  partial  composition products are given by  $\big(
x_{1 k}, \ldots, x_{n-k+1  n}
\big)\circ_i\big(
y_{1 k}, \ldots, y_{m-k+1  m}
\big):=
$
\begin{align*}
\big(
x_{1 k}, \ldots, x_{i-k+1  i}, 
\underbrace{*, \ldots, *}_{k-2}, 
y_{1 k}, \ldots, y_{m-k+1  m}, 
\underbrace{*, \ldots, *}_{k-2}, 
x_{i i+k-1}, \ldots, x_{n-k+1 n}
\big)
\ . 
\end{align*}

The special case $k=2$ gives back the operad $\As_{S^1}=\As^2_{S^1}$ of \cref{subsec:ncD}.

\begin{definition}[Linear $k$-Hypergraph skew--symmetric quadratic data]
The \emph{Linear $k$-Hypergraph skew--symmetric quadratic data} are spanned by 
\[
\kLHG(n)\coloneqq\left(
\t^n_{i i+k-1}\, ,\, \t^n_{i i+k-1}\wedge \t^n_{j j+k-1}
\right)\ ,
\]
where 
the set of generators $\t^n_{i i+k-1}$ of degree $0$ runs over the set of linear hyperedges of $\Theta^k_n$ and where the  set of relations runs over all pairs of  hyperedges of $\Theta^k_n$. 
For $n<k$, we set $\kLHG(n)\coloneqq(0, 0)$.
\end{definition}

We consider the following maps $\circ_p : \kLHG(n)\oplus \kLHG(m)\to \kLHG(n+m-1)$.
% Let us denote $I=\{i_1,\ldots, i_k\}$ and use the notation $I+a\coloneqq \{i_1+a,\ldots, i_k+a\}$. 
\begin{equation}\label{eq:CompoOpLHYGRA}
\begin{array}{lcll}
& \t^n_{i i+k-1} & \mapsto & \left\{\begin{array}{lll}  
\t^{n+m-1}_{i+m-1 i+k+m-2} & \text{for} & p\leqslant i\ ,\\
\rule{0pt}{12pt} 0 & \text{for} & i<p<i+k-1\ , \\
\rule{0pt}{12pt} \t^{n+m-1}_{i i+k-1} & \text{for} & i+k-1\leqslant p \ ,
\end{array}\right.\\
\rule{0pt}{12pt} & \t^m_{j j+k-1} & \mapsto & \quad  \t^{n+m-1}_{j+p-1 j+k+p-2}\ .\\
\end{array}
\end{equation}

\begin{lemma}\label{lem:kLHGDataOp}
The aforementioned data 
$\kLHG \coloneqq \big( \{k\text{-}\mathrm{LHG}(n)\}, \{\circ_p\}\big)$
 forms a nonsymmetric operad in the symmetric monoidal category \rm$(\QD^-, \oplus)$.
\end{lemma}

\begin{proof}
The proof is straightforward. 
\end{proof}

In the special case $k=2$, we recover the Linear Graph nonsymmetric operad $\mathrm{LG}=2\text{-}\mathrm{LHG}$.

\begin{proposition}\label{prop:HoloGraK}
The holonomy operad and the rational Magnus operad associated to $\As^k_{S^1}$ are isomorphic to the Lie operad associated to skew-symmetric quadratic data $\kLHG$: 
\[\g_{\As^k_{S^1}} \cong \L(\kLHG)\cong \gr\left(\pi_1\big(\As^k_{S^1}\big)\right)\otimes \QQ\ .\]
\end{proposition}

\begin{proof}
This proof is the same \emph{mutatis mutandis} as the one of \cref{prop:HoloGraS}
\end{proof}

One can define a \emph{linear $k$-hypergraph graded operad} $\mathrm{LGra}^k$. Its underlying $\mathbb{N}$-modules are spanned by
sub-hypergraphs of $\Theta^k_n$, where each hyperedge receives degree $1$. The partial composition product $\gamma_1\circ_p\gamma_2$ amounts to first inserting the linear $k$-hypergraph $\gamma_2$ at the $p$th vertex of $\gamma_1$, then relabelling accordingly the vertices, and finally keeping only the hyperedges of $\gamma_1$ which do not contain $p$, or for which $p$ is a minimum or a maximum element. 

\begin{proposition}\label{prop:HokHG}
The following three cocommutative Hopf operads are isomorphic 
\[H_\bullet\big(\As^k_{S^1}\big)\cong \S^c\big(\kLHG^{\emph{\acc}}\big) \cong \mathrm{LGra^k}\ .\]
\end{proposition}

\begin{proof}
This proof is similar to that of \cref{prop:HoBKW}. 
\end{proof}

\begin{remark}
For any $k\geqslant 2$, there is a canonical morphism 
$s\mathrm{pAs}^k\to \mathrm{LGra}^k$ from the nonsymmetric operad of (shifted) partially associative $k$-algebras \cite[Section~13.11.1]{LodayVallette12} which sends its generator to $\t^k_{1 k}$. Since this operad (unshifted) is the unit for the black product of $k$-ary quadratic nonsymmetric operads, one can develop a similar twisting procedure as that of  \cite{Willwacher15} according to \cite[Remark~5.8]{DotsenkoShadrinVallette18}. The study of the resulting dg nonsymmetric operad $\mathrm{TwLGra}^k$ is again an interesting subject. 
\end{remark}

A non-commutative version for the moduli spaces $\overline{\mathcal{M}}_{0,n+1}$ of stable curves with marked points was given in 
\cite{DotsenkoShadrinVallette15} by means of toric varieties called \emph{brick manifolds} and denoted by $\mathcal{B}(n)$. This family was endowed with a topological nonsymmetric operad structure. The linear $3$-hypergraph quadratic data is related to this ns operad $\mathcal{B}_{\mathbb{R}}$ in the real case. 

\begin{proposition}\label{prop:HoloBR}
The holonomy operad associated to $\mathcal{B}_{\mathbb{R}}$ is isomorphic to the Lie operad associated to skew-symmetric quadratic data $3\text{-}\mathrm{LHG}$: 
\[\g_{\mathcal{B}_{\mathbb{R}}}\cong \L(3\text{-}\mathrm{LHG})\ .\]
\end{proposition}

\begin{proof}
The proof is similar to the proof of \cref{prop:HoloMR}. It relies on the isomorphism of ns operads  
 $H_\bullet\big(\mathcal{B}_{\mathbb{R}}\big)\cong 2\text{-}\mathrm{ncGerst}$ from \cite[Theorme~9.3.1]{DotsenkoShadrinVallette15}
which identifies the following elements 
\[
w^n_{ii+2}\longleftrightarrow 
\vcenter{\hbox{\begin{tikzpicture}[yscale=0.5,xscale=0.5]
\draw (0,-1) -- (0,2) -- (1,3);
\draw (0,2) -- (-1, 3);
\draw (0,0) -- (-1, 1);
\draw (0,0) -- (-2, 1);
\draw (0,0) -- (1, 1);
\draw (0,0) -- (2, 1);
\draw (0,2) -- (0,3);
\node at (-1,3.4) {\scalebox{0.8}{$i$}};
\node at (-0.1,3.4) {\scalebox{0.8}{$i+1$}};
\node at (1.05,3.4) {\scalebox{0.8}{$i+2$}};
\node at (-2,1.4) {\scalebox{0.8}{$1$}};
\node at (2,1.4) {\scalebox{0.8}{$n$}};
\node at (1,1.4) {\scalebox{0.8}{$\cdots$}};
\node at (-1,1.4) {\scalebox{0.8}{$\cdots$}};
\node at (0,2) {$\bullet$};
\end{tikzpicture}}}
\qquad \text{and} \qquad 
\mathbb{1}^n \longleftrightarrow 
\vcenter{\hbox{\begin{tikzpicture}[yscale=0.5,xscale=0.5]
\draw (0,-1) -- (0,1) ;
\draw (0,0) -- (-1, 1);
\draw (0,0) -- (-2, 1);
\draw (0,0) -- (1, 1);
\draw (0,0) -- (2, 1);
\node at (-2,1.4) {\scalebox{0.8}{$1$}};
\node at (2,1.4) {\scalebox{0.8}{$n$}};
\node at (0,1.4) {\scalebox{0.8}{$\cdots$}};
\end{tikzpicture}}}
\ ,
\]
where the degree $1$ element 
$w_{i i+2}^n$ stands for $s\t^n_{i i+2}$ and where 
$\bullet$ denotes the shifted 2-partially associative product. Under this correspondence, the ns operad structure on $2\text{-}\mathrm{ncGerst}$ produces exactly the 
formul\ae\ given in \cref{eq:CompoOpLHYGRA}. 
\end{proof}

\subsection{Generalisation} Even if the following definition is not prompted by a family of already known topological operads, it is still possible to produce these skew-symmetric quadratic data $\krHG$ 
refining $\kHG$, for any $k\geqslant 2$, following a general canonical procedure which coincides to the aforementioned examples in the cases $k=2$ (\cref{subsec:LDOp}) and $k=3$ (\cref{subsec:EHKR}). 

\begin{definition}[Refined $k$-Hypergraph skew--symmetric quadratic data]
The \emph{refined $k$-Hypergraph skew--symmetric quadratic data} are spanned by 
\[
\krHG(n)\coloneqq\left(
\t^n_{I}\, ,\, \t^n_{I}\wedge \t^n_{J} \ \&\  
\t^n_{i_1 \ldots i_k}\wedge \big(
\t^n_{J, i_1}+\cdots+\t^n_{J, i_k}
\big)
\right)\ ,
\]
where 
the set of generators $\t^n_{I}$ of degree $0$ runs over the set of hyperedges $I$ of $\Gamma^k_n$, and where the  first set of relations runs over  pairs $(I, J)$ of disjoint hyperedges of $\Gamma^k_n$, i.e. $I\cap J=\emptyset$, and the second set of relations 
runs over pairs $(I=\{i_1, \ldots, i_k\}, J)$  formed by an hyperedge $I$ and a disjoint set $J$ of $k-1$ vertices of $\Gamma^k_n$. 
For $n<k$, we set $\krHG(n)\coloneq(0, 0)$.
\end{definition}

We consider the same partial composition products as the ones for the $k$-Hypergraph quadratic data given in \cref{eq:CompoOpHYGRA}. 

\begin{proposition}\label{prop:krKROp}
The refined $k$-Hypergraph skew--symmetric quadratic data
$\krHG\coloneqq\big(\{\krHG(n)\}, \{\circ_p\}\big)$
 forms an operad in the symmetric monoidal category \rm$(\QD^-, \oplus)$.
\end{proposition}

\begin{proof}
The proof is similar to that of  \cref{prop:DKOp} and \cref{prop:EHKROp}. Relations of first type (respectively  second type)  are sent to relations of first type (respectively  second type)  under the partial composition maps 
$(\circ_p)^{\wedge 2}$. 
 Any element of the relation $[V(n), V(m)]_-$ is sent to a sum of relations of first and second type under $(\circ_p)^{\wedge 2}$: for instance, the image of  $\t^n_{J, p}\wedge \t^m_{I}$ is equal to 
 \[
\left( \sum_{i\in I} \t^{n+m-1}_{\tilde{J}, i+p-1}
+\sum_{i\notin I} \t^{n+m-1}_{\tilde{J}, i+p-1}
\right)
\wedge \t^{n+m-1}_{I+p-1}
=
\left( \sum_{i\in I} \t^{n+m-1}_{\tilde{J}, i+p-1}\right)
\wedge \t^{n+m-1}_{I+p-1}
+\sum_{i\notin I} \left(\t^{n+m-1}_{\tilde{J}, i+p-1}
\wedge \t^{n+m-1}_{I+p-1}\right)
\ , \]
where $\tilde{J}$ is the "image" of $J$ in the complete $k$-hypergraph $\Gamma^k_{n+m-1}$, which is produced after relabelling. The first term on the right-hand side is a relation of second type and the second term on the right-hand side is a sum of relations of  first type. 
\end{proof}

The canonical morphisms of quadratic data $\krHG(n)\to \kHG(n)$ induces a canonical morphism of operads $\krHG \to \kHG$ in $\QD^-$. The refined $k$-Hypergraph skew--symmetric quadratic data
is characterized by the following universal operadic property. 

\begin{theorem}\label{thm:UnivCharGENERAL}
 The operad $\krHG$  is the smallest sub-operad of $\kHG$.
\end{theorem}
 
\begin{proof}
We proceed in the same way in the proof of \cref{thm:UnivCharDK}. 
Let us  use the notation $\krHG(n)=(V(n), R(n))$ and let us consider a sub-operad $\mathrm{P}(n)\coloneqq (V(n), S(n))\subset \kHG(n)$ of $\kHG$. We show  that $R(n)\subset S(n)$. 
 We begin with the relations of first type: 
$\t^n_{I}\wedge \t^n_{J}$. 
Using the action of the symmetric group, we can assume,  without any loss of generality, that $I=\{1,\ldots, k\}$ and 
$J=\{k+1, \ldots, 2k-1\}$. We conclude with 
 \[
(\circ_{k+1})^{\wedge 2}\big(\t_I^{n-k+1}\wedge \t_I^k\big)=\t_I^n\wedge \t_J^n\ .
\]
We treat now the relations of second type: $\t^n_{i_1, \ldots, i_k}\wedge \big(
\t^n_{J, i_1}+\cdots+\t^n_{J, i_k}
\big)$.
Using again the action of the symmetric group, the proof reduces to the case $(i, j,k,l)=(1,2,3,4)$, which is given by 
 \[
(\circ_1)^{\wedge 2}\big(\t_I^{n-k+1}\wedge \t_I^k\big)=
\big(
\t^n_{J, 1}+\cdots+\t^n_{J, k}
\big)\wedge \t^n_I\ .
\]

\end{proof} 
 
In the lattice of operads made up of skew-symmetric data with generators $\t_I^n$
and partial composition products $\circ_p$, the $k$-hypergraphs operad $\kHG$  is the maximal element and 
the refined $k$-hypergraphs operad $\krHG$  is the minimal element. 

\begin{definition}[Unital $(k-1)$-Gerstenhaber algebra]
A \emph{unital $(k-1)$-Gerstenhaber algebra}, for $k\geqslant 2$, is a chain complex $A$ equipped with an element $u\in A_0$ and two operations 
$\mu : A^{\odot 2} \to A$ of degree $0$ and $\beta : A^{\odot k} \to A$ of degree $1$ satisfying the following relations. 
\begin{description}
\item[\sc Unit relations] $\mu(u, -)=\id$ and $\beta(u, -, \ldots, -)=0$\ . 
\smallskip

\item[\sc Associativity relation] $\mu \circ_1 \mu =\mu \circ_2 \mu$\ . 
\smallskip

\item[\sc Leibniz relation] $\beta \circ_1 \mu = (\mu \circ_1 \beta)^{(2\, 3\, \cdots\, k \, k+1)}+\mu\circ_2 \beta$\ . 
\smallskip

\item[\sc Jacobi relation] $\sum_{\sigma\in \mathrm{Sh}_{k, k-1}^{-1}} (\beta \circ_1 \beta)^\sigma=0  $\ ,
\end{description}
where $\mathrm{Sh}_{k, k-1}^{-1}$ denotes the set of inverse of $(k,k-1)$-shuffles, also known as $(k,k-1)$-unshuffles.
\end{definition}

We denote the associated operad by $\mathrm{uGerst}_{(k-1)}$, which is generated by three generators, that we still denote respectively by $u$, $\mu$, and $\beta$. We endow it with a 
cocommutative Hopf operad structure by the following assignment:
\begin{align*}
\Delta(u)&\coloneqq u\otimes u \\
\Delta(\mu)&\coloneqq \mu\otimes \mu \\
\Delta(\beta)&\coloneqq \mu^{k-1}\otimes \beta + \beta \otimes \mu^{k-1}Ê\ ,
\end{align*}
where $\mu^{k-1}\coloneqq \underbrace{\mu\circ_1 \mu \circ_1 \cdots \mu \circ_1 \mu}_{k-1 \ \text{times}\ \mu}$\ .

\begin{lemma}
The above assignment defines a cocommutative Hopf operad structure on $\mathrm{uGerst}_{(k-1)}$.
\end{lemma}

\begin{proof}
We first need to show that the coproduct $\Delta$ is well-defined on the quotient of the free operad on $u$, $\mu$, and $\beta$ by the above relations. One can treat in a straightforward way the unit and the Leibniz relations. The case of the associativity relation is given by the following computation performed in the free operad
\begin{align*}
\Delta(\mu \circ_1 \mu -\mu \circ_2 \mu)&=
\mu \circ_1 \mu\otimes \mu \circ_1 \mu - \mu \circ_2 \mu\otimes \mu \circ_2 \mu
\\
&=(\mu \circ_1 \mu -\mu \circ_2 \mu)\otimes \mu \circ_1 \mu
+ \mu \circ_2 \mu \otimes (\mu \circ_1 \mu -\mu \circ_2 \mu)\ .
\end{align*}
The case of the Jacobi relation is treated as follows. Notice first that iterating the Leibniz relation, one gets the relation 
\begin{align*}
\beta\circ_1 \mu^{k-1}=\sum_{i=1}^k \left( \mu^{k-1}\circ_1 \beta\right)^{\sigma_i}\ ,
\end{align*}
where 
\begin{align*}
\sigma_i\coloneqq 
\left[
\begin{array}{cccccccccc}
1 & 2 & \cdots & k & k+1 & \cdots & k+i-1 & k+i & \cdots & 2k-1 \\
i & k+1 & \cdots & 2k-1 & 1 & \cdots & i-1 & i+1 & \cdots & k
\end{array}
\right]\ .
\end{align*}
We denote the induced element in the free operad by 
$\mathbb{L}\coloneqq \beta\circ_1 \mu^{k-1}-\sum_{i=1}^k \left( \mu^{k-1}\circ_1 \beta\right)^{\sigma_i}$\ . Similarly, the element representing the Jacobi relation in the free operad is denoted by $\mathbb{J}$. We conclude with the following computation:
\begin{align*}
\Delta\left(\mathbb{J}\right)
&=\sum_{\sigma\in \mathrm{Sh}_{k, k-1}^{-1}} \bigg(
(\beta \circ_1 \beta)^\sigma\otimes \mu^{2(k-1)}
+ \mu^{2(k-1)} \otimes (\beta \circ_1 \beta)^\sigma \\
&\qquad \qquad \quad \ \,
+\left(\beta\circ_1 \mu^{k-1}\right)^\sigma \otimes \left(\mu^{k-1}\circ_1 \beta\right)^\sigma
-\left(\mu^{k-1}\circ_1 \beta\right)^\sigma\otimes  \left(\beta\circ_1 \mu^{k-1}\right)^\sigma
\bigg)
\\
&=\mathbb{J} \otimes \mu^{2(k-1)}+ \mu^{2(k-1)}\otimes \mathbb{J} +
\sum_{\sigma\in \mathrm{Sh}_{k, k-1}^{-1}} \left(
\mathbb{L} \otimes \mu^{k-1}\circ_1\beta 
- \mu^{k-1}\circ_1\beta  \otimes \mathbb{L}
\right)^\sigma \ .
\end{align*}
In the end, it is enough to check the cocommutativity and the coassociativity of the coproduct $\Delta$ on the generators $u$, $\mu$, and $\beta$. 
\end{proof}

\begin{proposition}[\cite{Kho19}]
The cocommutative Hopf operad $\Sc\Big(\krHG^{\emph{\acc}}\big)$ is isomorphic to the cocommutative Hopf operad encoding unitary $(k-1)$-Gerstenhaber algebras, i.e. 
\[\Sc\big(\krHG^{\emph{\acc}}\big)\cong \mathrm{uGerst}_{(k-1)}\ .\]
\end{proposition}

\begin{proof}
The full proof was sent to us by Anton Khoroshkin and will appear in \cite{Kho19}. We only sketch the main strategy which extends the method of \cite{MelanconReutenauer96,BDK07} to the general case $k\geqslant 2$.

Let us denote by $\mathrm{uCom}$ the operad encoding unital commutative algebras and by 
$\mathrm{sLie}_{(k-1)}$ the operad encoding shifted Lie $(k-1)$-algebras \cite{HanlonWachs95}. 
The defining relations of the operad $\mathrm{uGerst}_{(k-1)}$ can be interpreted as rewriting rules which induces a distributive law \cite{Markl96} and \cite[Section~8.6]{LodayVallette12}. As a consequence, the underlying 
$\Sy$-module of the operad $\mathrm{uGerst}_{(k-1)}$ is isomorphic to the operadic composite product 
\[
\mathrm{uGerst}_{(k-1)}\cong \mathrm{uCom} \circ \mathrm{sLie}_{(k-1)}
\ .\]
This latter $\Sy$-module admits a basis made up of (commutative) forests of $(k-1)$-trees, that is rooted trees with all vertices of valence equal to $k$, modulo the Jacobi relation. 

On the other hand, one can see that the quadratic algebras $\S\big(\krHG^{!}(n)\big)$, for $n\geqslant k$, admit the following Koszul dual presentation 
\[
\krHG^{!}(n)=\left(\omega^n_{I}\, ,\, 
\omega^n_{i_1\ldots i_k} \odot \omega^n_{i_k \ldots i_{2k-1}} +
\omega^n_{i_2\ldots i_{k+1}} \odot \omega^n_{i_{k+1} \ldots i_{2k-1} i_1} + \cdots +
\omega^n_{i_{2k-1}i_1\ldots i_{k-1}} \odot \omega^n_{i_{k-1} \ldots i_{2k-2}}
\right)\ ,
\]
where 
the set of generators $\omega^n_{I}=s^{-1}\left(\t^k_{1 \ldots k}\right)^*$ of degree $-1$ runs over the set of hyperedges $I$ of $\Gamma^k_n$, and where  the  set of relations runs over increasing $(2k-1)$-tuples  $i_1<\cdots<i_{2k-1}$ .  

We consider the pairing $\langle\, , \rangle:  \left(\mathrm{uCom} \circ \mathcal{T}(\beta)\right)(n)\otimes \S(\omega^n_{I})$ defined by 
\[
\left\langle\mu^{n-1},1^n\right\rangle=1 \quad \text{and}\quad 
\left\langle
(\mu^{n-k} \circ_1 \beta)^\sigma,
\omega^n_I
\right\rangle=1
\quad \text{and}\quad 
\left\langle
\chi,
\omega^n_I
\right\rangle=0\ ,
\]
where $\sigma$ is the $(k, n-k)$ shuffle which sends $\{1, \ldots, k\}$ to $I$ and 
$\{k+1, \ldots, n\}$
to $\un\setminus I$ and where $\chi$ is an basis element of $\mathrm{uCom} \circ \mathcal{T}(\beta)$ different from   $(\mu^{n-k} \circ_1 \beta)^\sigma$. It induces a well-defined and non-degenerate pairing 
$\langle\, , \rangle:  \left(\mathrm{uCom} \circ \mathrm{sLie}_{(k-1)}\right)(n)\otimes \S\big(\krHG^{!}(n)\big)$, which  proves the isomorphism on the level of the underlying  $\Sy$-modules
\[\Sc\big(\krHG^{{\acc}}\big)\cong \mathrm{uGerst}_{(k-1)}\ .\]

By the definition of the pairing, this isomorphism respects to the arity-wise coalgebra structures. It remains to show that it also respects the partial composition products: this can be done in a straightforward way by a computation similar to the ones performed in the proofs of \cref{prop:HoloD2} and \cref{prop:HoloMR}. 

The most difficult part of the proof, not covered here, amounts to proving that the pairing $\langle\, , \rangle$ is non-degenerate. This requires further elaborate work which is done in \cite{Kho19}. 
\end{proof}

\begin{remark}
Considering the partitions of $\{1, \ldots, n(k-1)+1\}$ of size $l(k-1)+1$, for $0 \leqslant l \leqslant n$, together with their refinement, one gets a poset denoted by $\Pi_{n(k-1)+1}^{(k)}$, see \cite{HanlonWachs95}. This poset is actually the operadic partition poset associated to the set-theoretical operad encoding algebras made up of a commutative operation of arity $k$ satisfying a totally associative relation, see \cite{Vallette07}. These posets are Cohen--Macaulay and their top Whitney homology groups produce the cooperad $(k-1)\text{-}\mathrm{uGerst}^*$. 
\end{remark}

\begin{remark}
We refer the reader to the forthcoming paper \cite{Kho19} for the homological properties of the quadratic data $\krHG(n)$ and their associated operads. 
\end{remark}

``In the other way round'', one can define a family of pointed topological operads from the aforementioned skew-symmetric quadratic data as follows. 
We first consider the operad $\widehat{\L}\big(\krHG\big)$ in the category of \emph{complete} Lie algebras, see \cite[Section~2]{DotsenkoShadrinVallette18}. As a right adjoint, the functor, denoted here by $\mathrm{R}$, of \cite{BFMT15, Robert-Nicoud17} from complete dg Lie algebras to pointed simplicial sets is cartesian, it thus  sends operads to operads. This produces a pointed simplicial operad 
$\mathrm{R}\left(\widehat{\L}\big(\krHG\big)\right)$. Finally, the geometric realisation functor, again cartesian, provides us with the pointed topological operads
\[
\mathcal{O}_k\coloneqq \left|\mathrm{R}\left(\widehat{\L}\big(\krHG\big)\right)\right|\ .
\]
More details on the above mentioned cartesian functor $\mathrm{R}$, intimately related to the rational homotopy theory of operads, will be given in the sequel of this paper.

\begin{proposition}
The holonomy operad and the rational Magnus operad associated to $\mathcal{O}_k$ are isomorphic to the Lie operad associated to skew-symmetric quadratic data $\krHG$: 
\[\g_{\mathcal{O}_k}\cong \L\big(\krHG\big)\cong \gr\left(\pi_1\big(\mathcal{O}_k\big)\right)\otimes \QQ \ .\]
\end{proposition}

\begin{proof}
We first claim that, for any rational Lie algebra $\g$, the rational Magnus Lie algebra of $\left|\mathrm{R}\left(\g\right)\right|$ is isomorphic to $\g$. 
Recall that 
$\mathrm{R}(\g)\coloneqq \Hom_{\textbf{dg-Lie-alg}}\left(\mathfrak{mc}^\bullet, \g \right)$, where the cosimplicial complete dg Lie algebra $\mathfrak{mc}^\bullet$ is given by quasi-free complete dg Lie algebras
\[ \mathfrak{mc}^n\coloneqq \left(\widehat{Lie}\left(s^{-1} \mathrm{\Delta}^n\right), \mathrm{d}Ê\right)\]
on the desuspension of the standard $n$-simplicies. 
For a Lie algebra $\g$, this implies $\mathrm{R}(\g)_0\cong\{0\}$ and
\[
\mathrm{R}(\g)_n\cong\left\{(x_{ij})_{0\leqslant i<j\leqslant n}\in \g^{\binom{n+1}{2}}\mid \mathrm{BCH}(x_{ij}, x_{jk})=x_{ik}, \ \text{for}\  1\leqslant i<j<k\leqslant n\right\}\ ,\]
for $n\geqslant 1$, where $\mathrm{BCH}$ stands for the Baker--Campbell--Hausdorff formula. Its simplicial maps
are given by 
\begin{align*}
&s_k\left((x_{ij})\right)= 
 \left\{\begin{array}{lll}  
x_{ij} & \text{for} & i<j\leqslant k\ ,\\
\rule{0pt}{12pt} 
x_{ij-1} & \text{for} & i< k<j\ ,\\
\rule{0pt}{12pt}
x_{i-1j-1} & \text{for} & k<i<j\ ,\\
\rule{0pt}{12pt}
0 & \text{for} & i=k, j=k+1\ ,
\end{array}\right. \quad \text{and} \quad
d_k\left((x_{ij})\right)=
\left\{\begin{array}{lll}  
x_{ij} & \text{for} & i<j\leqslant k\ ,\\
\rule{0pt}{12pt} 
x_{ij+1} & \text{for} & i< k\leqslant j\ ,\\
\rule{0pt}{12pt}
x_{i+1j+1} & \text{for} & k\leqslant i<j\ .
\end{array}\right.
\end{align*}
It is a Kan complex canonically pointed by $0$, see for instance \cite{RV19} for details. 
It is straightforward to compute its first simplicial homotopy group: $\left(\pi_1(\mathrm{R}(\g)), \cdot\right) \cong 
\left(\g, \mathrm{BCH}\right)$. This produces the first isomorphism of Lie algebras 
\[\gr\left(\pi_1\big(
\left|\mathrm{R}\left(\g\right)\right|
\big)\right)\otimes \QQ \cong \g\ ,\] 
by \cite{Lazard50}. 

Finally, we claim that, for any skew-symmetric quadratic data $(V,R)$,  the holonomy Lie algebra of 
$\big|\mathrm{R}\big(\widehat{\L}(V,R)\big)\big|$ is isomorphic to the quadratic Lie algebra 
$\L(V,R)$. With the above description, it is straightforward to compute the rational simplicial groups of $\mathrm{R}(\g)$, which gives $H_1(\mathrm{R}(\g))\cong sV$ and $\mathrm{im}\, \Delta \cong s^2 R$. 
This implies the isomorphism of of Lie algebras:
\[\g_{|\mathrm{R}(\widehat{\L}(V,R))|}\cong \L(V,R)\ .\]
All these isomorphisms are natural and respect the operad structures. 
\end{proof}

Let us sum up the results of the previous sections into the following table. 

\begin{center}
\begin{tabular}{|c|c|c|c|}
\hline
&\rule{0pt}{12pt} $k=2$    &    $k=3$   &   $k\geqslant 4$ \\
\hline
{\sc Maximum} &\rule{0pt}{12pt} $\mathrm{BKW} \leftrightarrow \GraS$  &  $3\text{-}\mathrm{HG} \leftrightarrow \GraS^3$ & 
$\kHG \leftrightarrow \GraS^k$\\
\hline
{\sc Minimum} &\rule{0pt}{12pt}$\mathrm{DK} \leftrightarrow \mathcal{D}_2$  &  $\EHKR \leftrightarrow \overline{\mathcal{M}}_{0,n+1}(\mathbb{R})$ & 
$\krHG \leftrightarrow \left|\mathrm{R}\left(\widehat{\L}(\krHG)\right)\right|$\\
\hline
\end{tabular}
\end{center}

\smallskip

\begin{remark}
In the case of  linear hypergraphs, one can see that the lattice of nonsymmetric operads made up of skew-symmetric data with generators $\t^n_{i i+k-1}$ and partial composition products $\circ_p$ contains only one element: the nonsymmetric operad
$\kLHG$. Therefore, there is no way to refine it following the above pattern. 
\end{remark}

The canonical morphism of operads $\krHG\to \kHG$ in $\QD^-$ induces a  morphism  of cocommutative Hopf operads 
\[
 (k-1)\text{-}\mathrm{uGerst} \cong  H_\bullet\left(\left|\mathrm{R}\left(\widehat{\L}(\krHG)\right)\right|\right)\cong \Sc\big(\krHG^{{\acc}}\big)
\to 
\mathrm{Gra^k} \cong H_\bullet\big(\GraS^k\big)\cong \S^c\big(\kHG^{{\acc}}\big) \ .
\]
Studying the associated  deformation complex would solve the following question: 
what is the  $k^{\rm th}$ analogue (case $n=k$) of the Grothendieck--Teichm\"uller Lie algebra $\mathfrak{grt}$ (case $n=2$)?

\end{document}